\documentclass[11pt,reqno]{amsart}
\usepackage{xifthen,setspace}
\usepackage{bbm}
\usepackage{esvect}
\usepackage{amsmath} % provides numberwithin \left(and lots more\right)
\usepackage{style}
\usepackage{mathtools}
\usepackage[numbers, comma, sort]{natbib}
\newtheorem{assumption}{Assumption}
\newtheorem{example}{Example}
\newtheorem{prop}{Proposition}

\newcommand{\op}{\mathrm{op}}

\newcommand{\p}{{\P}}
\newcommand{\e}{{\E}}

\newcommand{\bs}{{\boldsymbol{\sigma}}}
\newcommand{\bldf}{{\bm f}}

\newcommand{\bt}{{\boldsymbol{\tau}}}

\DeclareMathOperator{\dett}{det_2}
\title[Scaling Limits for Ising Models on Inhomogeneous Random Graphs and Applications]{Scaling Limits for Ising Models on Inhomogeneous Random Graphs and Applications}

\author[Bhowal]{Sanchayan Bhowal}
\address{Department of Statistics, Stanford University. {\tt sbhowal@stanford.edu}}

\author[Chatterjee]{Anirban Chatterjee}
\address{Department of Mathematics and Statistics, Boston University. {\tt achatter@bu.edu}}

\author[Mukherjee]{Somabha Mukherjee} 
\address{Department of Statistics and Data Science, National University of Singapore, Singapore. {\tt somabha@nus.edu.sg}}

\begin{document}
\begin{abstract}
In this paper, we derive quenched scaling limits for linear functionals and the empirical spin field of Ising models on inhomogeneous random graphs generated by a graphon (encompassing both dense and sparse graphs), in the high-temperature regime. We first prove a joint central limit theorem (CLT) for finite collections of linear statistics of the spin configurations, where the limiting covariance is characterized by the resolvent of the associated graphon integral operator. Building on this result, we establish functional CLTs for the average magnetization and for the spin field indexed by suitable classes of regular test functions. We further prove convergence of the full empirical spin field, viewed as a random generalized function in negative Sobolev spaces. These scaling limits provide applications to both Bayesian neural networks and causal inference. Specifically, for the former, we derive infinite-width Gaussian-process limits for two-layer Bayesian neural networks with Ising-dependent output-layer signs, while for the latter, we establish the asymptotic normality of H\'{a}jek estimators for average treatment effects under network interference.
\end{abstract}
\maketitle

\section{Introduction}

Dependence in network-indexed spin systems is shaped jointly by the
strength of interaction and the geometry of the underlying network.
The Ising model provides the canonical probabilistic framework for
studying this interplay. Introduced in statistical physics to model ferromagnetism \cite{ising1925}, the Ising model has since become a central object in probability, statistics, machine learning, computer vision, spatial modeling, and
network science \cite{montanari,morningstar2018deep,geman,sudipto,hopfield}. Given a graphical network, the model assigns spins $\sigma_i\in\{-1,1\}$ to the vertices of the graph, and encodes dependence through pairwise interactions along the underlying network. Formally, given a symmetric interaction matrix
$J_N\coloneq ((J_N(i,j)))_{1\leq i,j\leq N}$, the Ising model with inverse-temperature parameter $\beta\geq 0$ is the probability measure on the hypercube $\{-1,1\}^N$, given by:
\begin{equation}
    \mathbb P_{\beta}(\bs)
    =
    \frac{1}{2^N Z_N(\beta)}
    \exp\{-\beta H_N(\bs)\},\qquad \bs \in \{-1,+1\}^N
    \label{model_def_general}
\end{equation}
where
   \(H_N(\bs)
    =
    -\sum_{1\leq i,j\leq N}
    J_N(i,j)\sigma_i\sigma_j\)
is the Hamiltonian and
\[
    Z_N(\beta)
    =
    \frac{1}{2^N}
    \sum_{\bs\in\{-1,+1\}^N}
    \exp\{-\beta H_N(\bs)\}
\]
is the partition function. The parameter $\beta$ controls the strength of alignment among interacting spins. In network applications, $J_N$ is typically chosen as a suitably normalized adjacency matrix, so that the quadratic interaction term aggregates the pairwise dependence carried by the network.

Fluctuations of magnetization and partition functions for Ising models on random graphs have been studied extensively. A popular and tractable model is the Curie--Weiss model (the Ising model on the complete graph), for which phase transitions and fluctuations of the magnetization have been known at all temperatures since the classical works of \cite{ellis1985entropy,ellisnewmanrosen1980}. For Ising models on Erd\H{o}s--R\'enyi graphs, Bovier and Gayrard \cite{BovierGayrard1993} identified the phase transition and established concentration of the magnetization. Kabluchko, L\"owe and Schubert proved central limit theorems for the magnetization in dense and sparse Erd\H{o}s--R\'enyi regimes \cite{kabluchko2019fluctuations,kabluchko2020fluctuations}, with further results in low-temperature and external-field regimes \cite{KabluchkoLoeweSchubert2022}. Fluctuations of the partition function in the high-temperature Erd\H{o}s--R\'enyi setting were obtained in \cite{kabluchko2021fluctuations}. These lines of work were subsequently extended to bounded-degree regimes in \cite{coja2026fluctuations,prodromidis2026distribution}. For approximately regular graphs, \cite{debmukh} established universality results for fluctuations of the average magnetization. Related results on random-field Ising models and applications to high-dimensional Bayesian linear regression can be found in \cite{lee2025fluctuations,lee2025clt}.

The present paper takes a broader field-level perspective, viewing the spin configuration as a random field and incorporating arbitrary network interactions specified by a graphon model, in both sparse and dense regimes.
More precisely, consider the following empirical signed measure associated with $\bs$:
\begin{equation}\label{d:etaNprel}
    \eta_N := \frac{1}{\sqrt{N}}\sum_{i=1}^N \sigma _i \delta_{i/N},
\end{equation}
 which we will refer to as the empirical spin field. For a test function $f:[0,1]\to\mathbb R$, define the
following linear statistic: 
\begin{equation}
    \sigma_N(f)
    \coloneq \int f \,\mathrm d \eta_N =
    \frac{1}{\sqrt N}
    \sum_{i=1}^N
    f\left(\frac{i}{N}\right)\sigma_i .
    \label{eq:intro_linear_stat}
\end{equation}
The choice $f\equiv 1$ gives the normalized magnetization, while indicators of intervals yield partial sums. General test functions give smeared observables of the spin field, which are natural in graphon-based models because the index $\frac{i}{N}$ represents the latent location of vertex $i$. Starting from a quenched multivariate central limit theorem for these
macroscopic observables, we obtain two complementary
infinite-dimensional extensions: a Donsker-type limit for the
partial-sum process and a Gaussian limit for $\eta_N$
in negative Sobolev spaces. We also establish function-indexed weak
convergence of the process \eqref{eq:intro_linear_stat} over compact classes of regular test functions. Taken together, these results give a unified field-level characterization of Gaussian fluctuations in heterogeneous network Ising models.

\subsection{Ising Models on Inhomogeneous Random Graphs}
\label{sec:isingW}

Inhomogeneous random graphs offer a versatile framework for encoding latent heterogeneity across network nodes
\cite{bickel2009nonparametric,bickel2011method,crane2018probabilistic,bollobas2007phase,hoff2002latent}. A standard construction of such random graphs is provided by a graphon, which is a symmetric measurable function
$W:[0,1]^2\to[0,1]$ \cite{Borgs2008,lovasz2012large}.
Given a sparsity parameter $\theta_N$, the associated sparse graphon model is defined as follows.

\begin{defn}[Sparse graphon model]
Fix a sparsity parameter $\theta_N\in(0,1]$ and a graphon $W$ such that
$\int_{[0,1]^2}W(x,y)\,\mathrm dx\,\mathrm dy>0$. The $W$-random graph with sparsity $\theta_N$, denoted by $G(N,\theta_N,W)$, is the random graph on the vertex set
$[N]\coloneq \{1,2,\ldots,N\}$ in which the vertices $i$ and $j$ are joined with probability
$
\theta_N
W(\frac{i}{N},\frac{j}{N}),
$
independently over all $1\leq i\leq j\leq N$.
\label{defn:W}
\end{defn}

This framework covers a broad range of familiar network models. In particular, it reduces to the Erd\H{o}s--R\'enyi random graph when $W\equiv 1$ and to a stochastic block model when $W$ is piecewise constant on a block partition \cite{bickel2009nonparametric,holland1983stochastic}; it also covers various other latent-variable network models. Sparse graphon models encompass many standard latent-variable network
models and are widely used in community detection, subgraph analysis,
and nonparametric network estimation.
 They also provide a natural class of random interaction structures for Markov random fields \cite{Borgs2012,basak2017universality,bhattacharya2018inference}.

Let $G_N\sim G(N,\theta_N,W)$ and let
$A_N=((A_N(i,j)))_{1\leq i,j\leq N}$ denote its adjacency matrix. Conditional on $G_N$, we define the Ising model by taking the interaction matrix in \eqref{model_def_general} to be
\(
    J_N
    =
    \frac{1}{2N\theta_N}A_N.
\)
Equivalently, the Hamiltonian is
\begin{equation}
    H_N(\bs)
    \coloneq
    -\frac{1}{2N\theta_N}
    \sum_{1\leq i,j\leq N}
    A_N(i,j)\sigma_i\sigma_j ,
    \label{eq:sufficientstatistics}
\end{equation}
where the random variables $\{A_N(i,j)\}_{1\leq i\leq j\leq N}$ are independent Bernoulli variables with parameters
\(
    \{
    \theta_N
    W(\frac{i}{N},\frac{j}{N})
    \}_{1\leq i\leq j\leq N}.
\)
The corresponding conditional Gibbs measure is then given by
\begin{equation}
    \P_{\beta,\theta_N,W}(\bs)
    =
    \frac{\exp\{-\beta H_N(\bs)\}}
    {2^N Z_N(\beta,\theta_N,W)},
    \qquad
    \bs\in\{-1,+1\}^N,
    \label{model_def}
\end{equation}
where
\[
    Z_N(\beta,\theta_N,W)
    =
    \frac{1}{2^N}
    \sum_{\bs\in\{-1,+1\}^N}
    \exp\{-\beta H_N(\bs)\}.
\]
The normalization by $N\theta_N$ in \eqref{eq:sufficientstatistics} places the energy on the mean-field scale: when the expected degree diverges, the aggregate contribution incident to a typical vertex remains of constant order.

Let $T_W:L^2[0,1]\to L^2[0,1]$ denote the integral operator associated with the graphon, i.e.
\[
    (T_W f)(x)
    =
    \int_0^1
    W(x,y)f(y)\,\mathrm dy,
\]
and let $\|W\|_{\mathrm{op}}\coloneq \|T_W\|_{\mathrm{op}}$. $T_W$ being a self-adjoint, compact operator, admits a countable spectrum $\{\lambda_i(W)\}_{i\ge 1}$. We will denote by $\{\phi_i\}_{i\ge 1}$ an orthonormal eigenbasis of $T_W$. We work in the high-temperature regime of the model \eqref{model_def}, defined by the condition
\begin{equation}\label{eq:subcriticdef8}
  \beta\|W\|_{\mathrm{op}}<1 .  
\end{equation}
In this regime, the operator $I-\beta T_W$ is invertible, and the resolvent $R_\beta \coloneq (I-\beta T_W)^{-1}$ governs the limiting covariance structure of the spin field. 

\subsection{Summary of Results}
\label{sec:results}

We now describe the main results of the paper. All convergence statements below are quenched: the relevant
conditional laws given the random graph $A_N$ converge weakly in
probability with respect to the graph randomness. The corresponding
unconditional limits follow as consequences.

\begin{itemize}

    \item \textit{Multivariate CLT for macroscopic linear statistics.}
    Our first result (Theorem \ref{thm:isingWeakConv}) is a multivariate central limit theorem for the linear statistics \eqref{eq:intro_linear_stat}. Under the sparsity condition
    $\theta_N=\Omega(N^{-2/3})$, for Riemann integrable test functions $f_1,\ldots,f_k:[0,1]\to\mathbb R$, we prove that
    \[
        \big(\sigma_N(f_1),\ldots,\sigma_N(f_k)\big)
        \xrightarrow{d}
        \mathcal N_k(\boldsymbol 0,\Sigma),
    \]
    where
    \[
        \Sigma_{ij}
        =
        \left\langle
        f_i,(I-\beta T_W)^{-1}f_j
        \right\rangle .
    \]
    Note that when $\beta=0$, the covariance reduces to $\langle f_i,f_j\rangle$, corresponding to independent Rademacher spins. For $\beta>0$, the joint fluctuations of arbitrary macroscopic observables retain the full operator structure of the graphon through its resolvent $(I-\beta T_W)^{-1}$.
 \item \textit{Scaling limit for the partial sum process.}
    We next consider the partial-sum process
    \[
        S_N(t)
        =
        \frac{1}{\sqrt N}
        \sum_{i=1}^{\lfloor Nt\rfloor}
        \sigma_i,
        \qquad 0\leq t\leq 1.
    \]
    We prove in Theorem \ref{thm:donsker}, convergence of $S_N$ in the Skorokhod topology to a centered Gaussian process $X$ with covariance
    \[
        \mathbb E[X(s)X(t)]
        =
        \left\langle
        \mathbf 1_{[0,s]},
        (I-\beta T_W)^{-1}\mathbf 1_{[0,t]}
        \right\rangle .
    \]
    Equivalently, if $\{\lambda_i,\phi_i\}_{i\geq 1}$ denotes the spectrum and an orthonormal eigenbasis of $T_W$, and
    \[
        \Phi_i(t)
        =
        \int_0^t
        \phi_i(x)\,\mathrm dx,
    \]
    then the limiting process admits the representation
    \[
        X(t)
        =
        B(t)
        +
        \sum_{i=1}^{\infty}
        \left(
        \frac{1}{\sqrt{1-\beta\lambda_i}}-1
        \right)
        \Phi_i(t)Z_i,
    \]
    where $B$ denotes the standard Brownian motion and
    $Z_i=\int_0^1\phi_i(t)\,\mathrm dB(t)$. Thus, the classical Brownian scaling limit arising in the independent-spin case
is modified by a spectral correction determined by the inhomogeneous
interaction kernel.

            \item \textit{Function-indexed CLT for linear statistics.}
            Next, we pass from finitely many test functions
to compact classes of absolutely continuous functions. Specifically, we establish in Theorem \ref{thm:process_class}, weak convergence of the function-indexed process
    \[
        \mathbb G_N
        =
        \bigl(\sigma_N(f)\bigr)_{f\in\mathcal C},
    \]
    where $\mathcal C$ is a compact class of absolutely continuous test functions, to a centered Gaussian process $\mathbb G$ given by
    \[
        \mathbb G(f)
        =
        f(1)X(1)
        -
        \int_0^1
        X(t)f'(t)\,\mathrm dt,
        \qquad f\in\mathcal C.
    \]

        \item \textit{Gaussian limit of the empirical spin field.}
    We next consider the empirical spin field $\eta_N$ defined in \eqref{d:etaNprel}, which we canonically identify with a random
    element of the negative Sobolev space $H^{-s}(0,1)$ (see Section \ref{sobolevcan8} for the definition), for every $s>1/2$. We prove in Theorem \ref{thm:sob_conv} that
    \[
        \eta_N
        \xrightarrow{d}
        \eta
        \qquad\text{in }H^{-s}(0,1),
    \]
    where the limiting centered Gaussian random element admits the
    representation
    \[
        \eta
        =
        \sum_{k=0}^{\infty}
        \chi(e_k)e_k,
    \]
    with $\{e_k\}_{k\geq0}$ being the cosine orthonormal basis of
    $L^2[0,1]$, and $\chi$ being the centered Gaussian linear process with
    covariance
    \[
        \mathbb E[\chi(f)\chi(g)]
        =
        \left\langle
        f,(I-\beta T_W)^{-1}g
        \right\rangle .
    \]
    Further, viewing these fields as
    $H^{-1}(0,1)$-valued random elements, we establish weak convergence
    of their duality actions over compact classes of
    $H^1(0,1)$ test functions.

    \item \textit{Applications to Bayesian neural networks and causal inference.}
    Our results have natural applications in Bayesian neural networks and causal inference.
    First, we consider two-layer Bayesian neural-network priors in
    which the output-layer signs are sampled from the Ising measure
    rather than independently. We prove in Theorem \ref{thm:ising-neural-network} an infinite-width finite-dimensional Gaussian-process limit and identify the explicit rank-one correction to the classical neural-network Gaussian-process kernel induced by Ising dependence among the output-layer signs \cite{neal2012bayesian}. Second, we interpret the spins as
    network-dependent treatment assignments in a causal model with
    outcome interference through the fraction of treated peers. In Theorem \ref{thm:hajek-graphon-ising}, we establish asymptotic normality of the H\'ajek estimator for the conditional direct average treatment effect on a complete outcome-interference network, when the treatment assignments are generated according to an inhomogeneous Ising model.
\end{itemize}

We now briefly indicate the main ideas behind the proofs. The finite-dimensional fluctuation theory is driven by a rank-one
perturbation analysis of the partition function. The conditional
moment-generating function of $\sigma_N(f)^2$ is expressed as a ratio
of random partition functions, which are shown to concentrate around
their annealed counterparts. A 2-modified Fredholm determinant identity
then yields the resolvent covariance
$\langle f,R_\beta f\rangle$. The Donsker-type limit is obtained by
combining quenched finite-dimensional convergence with ferromagnetic
moment bounds, while the Sobolev-field limit follows from convergence of finite-dimensional projections together with a uniform control of the tail.

\subsection{Asymptotic Notation}
\label{sec:aNbN}

Throughout the paper, we use the following asymptotic notations. For two nonnegative sequences $a_N$ and $b_N$, we write $a_N\lesssim b_N$ if there exists a constant $C<\infty$ such that $a_N\leq Cb_N$ for all sufficiently large $N$. We write $a_N\gtrsim b_N$ if $b_N\lesssim a_N$. 

\subsection{Preliminaries}
In this section, we collect the notation and operator-theoretic preliminaries used throughout the paper. Although some of this notation has been introduced informally above, we provide formal definitions here for completeness.

To begin with, define the graphon integral operator $T_W:L^2[0,1]\to L^2[0,1]$ by
\begin{equation}
    \label{eq:defT_W}
    (T_W f)(x) = \int_0^1 W(x,y)f(y)~dy~.
\end{equation}
Note that $T_W$ is a self-adjoint, compact operator. Therefore, $T_W$ has a discrete spectrum $\{\lambda_i(W)\}_{i \in \N}$, where $\N$ denotes the set of natural numbers. Moreover, $\sum_i \lambda_i(W)^2$ is finite, and in fact equals $\int_{[0,1]^2} W^2(x,y)~dx dy$. Thus, $T_W$ is also a Hilbert-Schmidt operator. Let $\{\phi_i\}_{i\in \N}$ be an orthonormal eigenbasis of $T_W$.
\begin{defn}
Let $H$ be a Hilbert space, and let
$\{s_j(T)\}_{j\geq1}$ denote the singular values of a compact
operator $T:H\to H$, that is, the eigenvalues of
$(T^*T)^{1/2}$, counted with multiplicity. The operator $T$ is
called \emph{trace--class} if
\[
    \sum_{j\geq1}s_j(T)<\infty,
\]
and \emph{Hilbert--Schmidt} if
\[
    \sum_{j\geq1}s_j(T)^2<\infty.
\]
If $T$ is self-adjoint, then its singular values are the absolute
values of its eigenvalues $\{\lambda_j(T)\}_{j\ge 1}$, and hence these conditions are equivalent
to
\[
    \sum_{j\geq1}|\lambda_j(T)|<\infty
    \qquad\text{and}\qquad
    \sum_{j\geq1}\lambda_j(T)^2<\infty,
\]
respectively.
\end{defn}
\begin{defn}
    Let $T$ be a trace--class operator with eigenvalues $\lambda_i$, counted with algebraic multiplicity. Then the \emph{Fredholm determinant} of $I+T$ ($I$ being the identity operator) is defined as,
    \begin{equation*}
        \det(I+T) = \prod_i(1+\lambda_i)
    \end{equation*}
\end{defn}
Note that an arbitrary Hilbert--Schmidt operator may not be trace--class, and so its Fredholm determinant may not exist. For such operators, we define a different type of determinant:
\begin{defn}
    Let $T$ be a Hilbert--Schmidt operator with eigenvalues $\lambda_i$,
counted with algebraic multiplicity. Then the \emph{2-modified Fredholm determinant} of $I+T$ is defined as
    \begin{equation*}
        \dett(I+T) = \prod_i e^{-\lambda_i}(1+\lambda_i)
    \end{equation*}
\end{defn}
Hence, for a trace--class operator $T$, $\dett(I+T)=\det(I+T)e^{-\tr T}$.

\begin{defn}
    Let us consider $f,g \in L^2[0,1]$. Define the following \emph{rank-one operator} $\opb{g}{f}:L^2[0,1] \to L^2[0,1]$ as,
    \begin{equation*}
        \opb{g}{f}h=\ipb{f}{h} g,
    \end{equation*}
    where $\ipb{f}{h} \coloneqq \int_0^1 f(x) h(x)~dx$.
\end{defn}
For a more detailed background on these concepts, see
\cite{simon-2010}. The central objects of this paper are the following linear statistic of the spins and the empirical spin field, defined below.

\begin{defn}
We define the empirical spin field as the following empirical signed measure:
$$\eta_N := \frac{1}{\sqrt{N}}\sum_{i=1}^N \sigma _i \delta_{i/N}.$$
    Further, for any function $f: [0,1]\to \R$, we define:
    $$\sigma_N(f) \coloneqq \int f \,\mathrm d \eta_N= \frac{1}{\sqrt{N}}\sum_{i=1}^N f\left(\frac{i}{N}\right)\sigma_i.$$
\end{defn}
 In this notation, the partial-sum process is defined as:
\[
    S_N(t)
    \coloneq
    \sigma_N\bigl(\mathbf 1_{[0,t]}\bigr)
    =
    \frac{1}{\sqrt N}
    \sum_{i=1}^{\lfloor Nt\rfloor}
    \sigma_i,
    \qquad 0\leq t\leq 1.
\]

Our main results establish weak convergence in probability for the conditional laws, given \(A_N\), of \(\sigma_N(f)\) and \(S_N(t)\), at different levels: one-dimensional marginals, joint finite-dimensional distributions, and the full process. To this end, we give the following definition.

\begin{defn}
    Let $\mathcal{S}$ be a Polish space and $\mathcal{P}(\mathcal S)$ denote the set of all probability measures on the Borel $\sigma$-field on $\mathcal S$, equipped with the topology $\Gamma$ of weak convergence. Let $\{\mu_n\}_{n\ge 1}$ be a family of $\mathcal P(\mathcal S)$-valued random probability measures defined on a common probability space $(\Omega,\mathscr{F},\P)$ (i.e. each $\mu_n$: $(\Omega,\mathscr{F}) \to (\mathcal P(\mathcal S), \mathcal{B}\bigl(\mathcal{P}(\mathcal S),\Gamma\bigr))$ is a measurable function, where $\mathcal{B}\bigl(\mathcal{P}(\mathcal S),\Gamma\bigr)$ denotes the Borel $\sigma$-field on $\mathcal P(\mathcal{S})$ generated by $\Gamma$). We say that $\mu_n$ converges weakly in probability to a measure $\mu \in  \mathcal P(\mathcal S)$, if for all $U \in \Gamma$ containing $\mu$, we have $$\P(\mu_n \in U) \rightarrow 1.$$ 
\end{defn}

Since \(\mathcal S\) is Polish and the weak topology $\Gamma$ is metrizable, the same definition is equivalent to saying that
\(
    d(\mu_n,\mu)\xrightarrow{P}0
\)
for any metric \(d\) metrizing weak convergence on $\mathcal P(\mathcal S)$.

Finally, throughout the paper, we work under some standing assumptions, which we list below. 

\begin{assumption}\label{assumption12} Throughout, we impose the following assumptions on $\beta$, $W$, and
$\theta_N$.

\begin{enumerate} 

\item  $\beta \in (0, \frac{1}{\|W\|_{\mathrm{op}}})$, where $\|W\|_{\mathrm{op}}$ is the operator norm of the graphon $W$. 

\item The graphon $W: [0, 1]^2 \rightarrow [0, 1]$ and its diagonal map $\mathrm{diag}_W : [0, 1] \rightarrow [0, 1]$, defined as $\mathrm{diag}_W(x) \coloneq  W(x, x)$, are both Riemann integrable. Moreover, $\int_{[0, 1]^2} W(x, y) \mathrm d x \mathrm d y > 0$. 
 
\item $N^{\frac{2}{3}}\theta_N \gtrsim 1$ and $\theta_N \rightarrow \theta \in [0, 1]$. 

\end{enumerate} 
\end{assumption}

\section{Main Results}
In this section, we present our main results on the asymptotic fluctuations of linear functionals in the Ising model \eqref{model_def}.

\subsection{A Multivariate CLT for Linear Statistics of the Spin Field}
We start with a multivariate central limit theorem for linear statistics in the Ising model \eqref{model_def}.
\begin{thm}
    \label{thm:isingWeakConv}
    For Riemann integrable functions \(f_1,\ldots,f_k:[0,1]\to\mathbb R\),
    the conditional law of
    \[
        \big(\sigma_N(f_1),\ldots,\sigma_N(f_k)\big)
    \]
    given \(A_N\), viewed as a random probability measure on \(\mathbb R^k\),
    converges weakly in probability to \(\mathcal N_k(\boldsymbol 0,\Sigma)\), where
    \[
        \Sigma_{ij}
        =
        \left\langle
        f_i,(I-\beta T_W)^{-1}f_j
        \right\rangle .
    \]
    In particular,
    \[
        \big(\sigma_N(f_1),\ldots,\sigma_N(f_k)\big)
        \xrightarrow{d}
        \mathcal N_k(\boldsymbol 0,\Sigma).
    \]
\end{thm}

The proof of Theorem \ref{thm:isingWeakConv} is given in Section \ref{proofthmisingWeakConv}. To illustrate the implications of Theorem \ref{thm:isingWeakConv}, we compute the limiting joint distribution of linear statistics for the Ising model \eqref{model_def} in two representative settings: first, when the underlying graph is generated by a two-block stochastic block model (SBM), and second, when it is generated by a rank-one graphon model. Together, these settings cover several widely used graphon models.

\begin{example}\label{example:block-linear-stat} (Block models)
Let the underlying graphon be the 2-block graphon
\begin{align}\label{eq:Wpq-linear-stat}
W(x, y) =
\begin{cases}
p  & \text{ for } (x, y) \in [0, \frac{1}{2}]^2
        \bigcup (\frac{1}{2}, 1]^2, \\[2mm]
q  & \text{ for } (x, y) \in \left\{[0, \frac{1}{2}] \times (\frac{1}{2}, 1]\right\}
        \bigcup \left\{(\frac{1}{2}, 1] \times [0, \frac{1}{2}]\right\},
\end{cases}
\end{align}
where $p,q\in[0,1]$ and $p+q>0$. This corresponds to the graphon generating a
balanced stochastic block model with within and between block connection
probabilities $p$ and $q$, respectively. In this case, the only possibly non-zero eigenvalues of $T_W$ are
\begin{equation}\label{eigvpm8}
    \lambda_+ = \frac{p+q}{2}
    \qquad \text{and} \qquad
    \lambda_- = \frac{p-q}{2},
\end{equation}
with corresponding orthonormal eigenfunctions
\[
    \phi_+(x)\equiv 1
    \qquad \text{and}\qquad
    \phi_-(x)=
    \mathbf 1_{[0,\frac12]}(x)-\mathbf 1_{(\frac12,1]}(x),\quad \text{respectively}.
\]
Hence, $\|W\|_{\mathrm{op}} = (p+q)/2$, so the region under consideration is $0<\beta < 2/(p+q)$. 

Let 
\(
    P_\pm h=\langle h,\phi_\pm\rangle\phi_\pm
\)
 denote the corresponding orthogonal projections. Then \(
    T_W=\lambda_+P_+ + \lambda_-P_-,
\)
and hence,
\[
    I-\beta T_W
    =
    (I-P_+-P_-)
    +(1-\beta\lambda_+)P_+
    +(1-\beta\lambda_-)P_-.
\]
Since \(P_+^2=P_+\), \(P_-^2=P_-\), and \(P_+P_-=P_-P_+=0\), it follows that
\[
    (I-\beta T_W)^{-1}
    =
    (I-P_+-P_-)
    +
    \frac{1}{1-\beta\lambda_+}P_+
    +
    \frac{1}{1-\beta\lambda_-}P_- = 
    I+
    \frac{\beta\lambda_+}{1-\beta\lambda_+}P_+
    +
    \frac{\beta\lambda_-}{1-\beta\lambda_-}P_-.
\]
Substituting the values for $\lambda_{\pm}$ from \eqref{eigvpm8}, we have:
\[
    (I-\beta T_W)^{-1}
    =
    I+
    \frac{\beta(p+q)}{2-\beta(p+q)}P_+
    +
    \frac{\beta(p-q)}{2-\beta(p-q)}P_-.
\]

Next, for Riemann integrable functions $f_1,\ldots,f_k:[0,1]\to\R$, define
\[
    m_i^+
    \coloneq
    \ipb{f_i}{\phi_+}
    =
    \int_0^1 f_i(x)\,\mathrm dx
    \qquad \text{and}\qquad
    m_i^-
    \coloneq
    \ipb{f_i}{\phi_-}
    =
    \int_0^{1/2} f_i(x)\,\mathrm dx
    -
    \int_{1/2}^1 f_i(x)\,\mathrm dx .
\]

Then, by Theorem \ref{thm:isingWeakConv}, we have:
\begin{align}\label{eq:block-linear-clt}
    \left(\sigma_N(f_1),\ldots,\sigma_N(f_k)\right)
    \xrightarrow{d}
    \cN_k\left(\boldsymbol 0,\Sigma^{\mathrm{block}}\right),
\end{align}
where
\begin{align}\label{eq:Sigma-block-linear}
    \Sigma^{\mathrm{block}}_{ij}
    =
    \ipb{f_i}{f_j}
    +
    \frac{\beta(p+q)}{2-\beta(p+q)}\,m_i^+m_j^+
    +
    \frac{\beta(p-q)}{2-\beta(p-q)}\,m_i^-m_j^- .
\end{align}

Asymptotics of the scaled magnetization $N^{-1/2}\sum_{i=1}^N \sigma_i$ can be derived by taking $k=1$ and $f_1\equiv 1$, which gives $m_1^+ = 1$ and $m_1^-=0$. Hence, 
$$\frac{1}{\sqrt{N}}\sum_{i=1}^N \sigma_i \xrightarrow{d} \cN\left(0,\frac{2}{2-\beta(p+q)}\right).$$

The following are some important special cases of this example.
\begin{itemize}

\item {\it Erd\H{o}s--R\'enyi model $G(N,\theta_N)$.}
This corresponds to $W\equiv 1$, that is, $p=q=1$, and hence, the asymptotics in this case are valid for $0<\beta<1$. The covariance in
\eqref{eq:Sigma-block-linear} thus simplifies to
\begin{align*}
    \Sigma^{\mathrm{ER}}_{ij}
    =
    \ipb{f_i}{f_j}
    +
    \frac{\beta}{1-\beta}
    \left(\int_0^1 f_i(x)\,\mathrm dx\right)
    \left(\int_0^1 f_j(x)\,\mathrm dx\right).
\end{align*}
In particular, for the scaled magnetization, we have: 
\[
    \frac{1}{\sqrt{N}}\sum_{i=1}^N \sigma_i 
    \xrightarrow{d}
    \cN\left(0,\frac{1}{1-\beta}\right),
\]
which aligns with Theorem 1.1 in \cite{kabluchko2019fluctuations}.

\item {\it Random bipartite graph.}
This corresponds to taking $p=0$ and $q=1$, and hence, the asymptotics in this case are valid for $0<\beta<2$. In this case,
\eqref{eq:Sigma-block-linear} becomes
\begin{align*}
    \Sigma^{\mathrm{RB}}_{ij}
    =
    \ipb{f_i}{f_j}
    +
    \frac{\beta}{2-\beta}\,m_i^+m_j^+
    -
    \frac{\beta}{2+\beta}\,m_i^-m_j^- .
\end{align*}
For the scaled magnetization, we have:
\[
    \frac{1}{\sqrt{N}}\sum_{i=1}^N \sigma_i 
    \xrightarrow{d}
    \cN\left(0,\frac{2}{2-\beta}\right). 
\] 
\end{itemize}
\end{example}

\begin{example}[Rank-one graphon model]\label{example:rank-one-linear-stat}
Consider the rank-one graphon
\[
    W(x,y)=g(x)g(y),
\]
where $g:[0,1]\to[0,1]$ is Riemann integrable, and not
identically zero. In this case, we have: 
\[
    T_W h = g \ipb{g}{h},
    \qquad \text{and}\qquad
    T_W^2=\mu_2T_W
    \quad\text{where}\quad
    \mu_2\coloneq \|g\|_2^2.
\]
In fact, in this case $\|W\|_{\mathrm{op}}=\mu_2$, so the region under consideration is $0<\beta< 1/\mu_2$. Also, a straightforward computation gives:
\[
    (I-\beta T_W)^{-1}
    =
    I+\frac{\beta}{1-\beta\mu_2}T_W .
\]

Now, for Riemann integrable functions $f_1,\ldots,f_k:[0,1]\to\R$, define
\[
    r_i
    \coloneq
    \ipb{f_i}{g}
    =
    \int_0^1 f_i(x)g(x)\,\mathrm dx .
\]
Then, by Theorem \ref{thm:isingWeakConv}, we have:
\begin{align}\label{eq:rank-one-linear-clt}
    \left(\sigma_N(f_1),\ldots,\sigma_N(f_k)\right)
    \xrightarrow{d}
    \cN_k\left(\boldsymbol 0,\Sigma^{\mathrm{rank}}\right),
\end{align}
where
\begin{align}\label{eq:Sigma-rank-one-linear}
    \Sigma^{\mathrm{rank}}_{ij}
    =
    \ipb{f_i}{f_j}
    +
    \frac{\beta}{1-\beta\mu_2}\,
    r_i r_j .
\end{align}

Once again, for deriving asymptotics of the scaled magnetization $N^{-1/2}\sum_{i=1}^N \sigma_i$, we take $k=1$ and $f_1\equiv 1$. This gives $r_1 = \int_0^1 g(x)\,\mathrm dx$, and hence,
$$\frac{1}{\sqrt{N}}\sum_{i=1}^N \sigma_i \xrightarrow{d} \cN\left(0,1+\frac{\beta(\int_0^1 g)^2}{1-\beta\int_0^1 g^2}\right).$$
\end{example}

\begin{remark}
  The linear functionals $\sigma_N(f)$ studied here have close analogues
in recent work on scaling limits for lattice and long-range Ising
models, where similar weighted sums of spins are used to study the large-scale
behavior of the system. Such smeared observables have been used
to establish Gaussianity and triviality of critical scaling limits,
and to identify fractional Gaussian free field correlations in
subcritical long-range models \cite{aizenman2021marginal,panis2023triviality,
gunaratnam2025emergence}. 
\end{remark}

\subsection{Scaling Limit for the Partial Sum Process}
Theorem \ref{thm:isingWeakConv} can be used to identify the scaling limit of the partial sum process $S_N(t) \coloneqq N^{-1/2} \sum_{i=1}^{\lfloor Nt\rfloor}\sigma_i$ as $N \rightarrow \infty$. Let us define the following continuous Gaussian process:
\begin{equation}
\label{defXt}
    X(t)\coloneqq B(t)+ \sum_{i=1}^\infty \left(\frac{1}{\sqrt{1-\beta \lambda_i}}-1\right) \Phi_i(t) Z_i,
\end{equation}
where $Z_i= \int_0^1\phi_i(t)\,\mathrm dB(t)$, $\Phi_i(t)=\int_0^t \phi_i(x)\, \mathrm dx$ and $\{B(t)\}_{t\in [0,1]}$ denotes the standard Brownian motion. The well-definedness of this process along with its covariance kernel are established in Lemma \ref{constructXt}. We will shortly prove that the process $X$ is the correct candidate for the functional weak limit of $S_N$ in an appropriate sense. Let us first derive the exact form of this limiting process for Ising models on the various well-known graph ensembles discussed in the examples above.

\begin{example}[Block models: limiting process]
\label{example:block-limiting-process}
For the 2-block graphon model in Example
\ref{example:block-linear-stat}, 
\[
    \Phi_+(t)=t,\qquad
    Z_+=B(1),\qquad
    \Phi_-(t)=t\wedge(1-t),\qquad
    Z_-=2B\left(\frac12\right)-B(1).
\]
Since all remaining eigenvalues are zero, \eqref{defXt} gives
\begin{eqnarray*}
    &&X^{\mathrm{block}}(t)\\
    &=&
    B(t)+ \left(\sqrt{\frac{2}{2-\beta(p+ q)}}-1\right)tB(1)
    +\left(\sqrt{\frac{2}{2-\beta(p-q)}}-1\right)(t\wedge(1-t))
    \left(2B\left(\frac12\right)-B(1)\right).
\end{eqnarray*}
for $0<\beta<2/(p+q)$. In particular, note that
\(
    X^{\mathrm{block}}(1)
    =
    \sqrt{\frac{2}{2-\beta(p+q)}}\,B(1),
\)
recovering the scaled-magnetization limit in Example
\ref{example:block-linear-stat}.

For the Erd\H{o}s--R\'enyi and random bipartite special cases,
respectively, the limiting processes simplify to
\begin{align*}
    X^{\mathrm{ER}}(t)
    &=
    B(t)
    +
    \left(
        \frac{1}{\sqrt{1-\beta}}-1
    \right)tB(1),
    \qquad 0<\beta<1,
    \\
    X^{\mathrm{RB}}(t)
    &=
    B(t)
    +
    \left(
        \sqrt{\frac{2}{2-\beta}}-1
    \right)tB(1) \notag\\
    &\quad+
    \left(
        \sqrt{\frac{2}{2+\beta}}-1
    \right)(t\wedge(1-t))
    \left(
        2B\left(\frac12\right)-B(1)
    \right),
    \qquad 0<\beta<2.
   % \label{eq:RB-limiting-process}
\end{align*}
\end{example}

\begin{example}[Rank-one graphon model: limiting process]
\label{example:rank-one-limiting-process}
For the rank-one graphon model in Example
\ref{example:rank-one-linear-stat}, $T_W$ has only one nonzero eigenvalue $\|g\|_2^2$, with corresponding orthonormal eigenvector $g/\|g\|_2$. Hence, $Z_1 := \|g\|_2^{-1} \int_0^1 g(t) d B(t)$, and therefore, the limiting process in this case is
\[
    X^{\mathrm{rank}}(t)
    =
    B(t)
    +
    \frac{1}{\|g\|_2^2}\left(
        \frac{1}{\sqrt{1-\beta\|g\|_2^2}}-1
    \right)
    \left(\int_0^t g(x) dx\right)
    \left(\int_0^1g(x)\,\mathrm dB(x)\right), \qquad 0<\beta<\frac{1}{\|g\|_2^2}.
\]
In particular, note that  $ X^{\mathrm{rank}}(1)$ is a centered Gaussian with variance given by:
\begin{eqnarray*}
\operatorname{Var}\left(X^{\mathrm{rank}}(1)\right)
&=&
1+
\frac{2}{\|g\|_2^2}
\left(
    \frac{1}{\sqrt{1-\beta\|g\|_2^2}}-1
\right)
\left(\int_0^1 g(x)\,\mathrm dx\right)
\operatorname{Cov}\left(
    B(1),\int_0^1 g(x)\,\mathrm dB(x)
\right)\\
&&\quad+
\frac{1}{\|g\|_2^4}
\left(
    \frac{1}{\sqrt{1-\beta\|g\|_2^2}}-1
\right)^2
\left(\int_0^1 g(x)\,\mathrm dx\right)^2
\operatorname{Var}\left(
    \int_0^1 g(x)\,\mathrm dB(x)
\right)\\
&=&
1+
\frac{2}{\|g\|_2^2}
\left(
    \frac{1}{\sqrt{1-\beta\|g\|_2^2}}-1
\right)
\left(\int_0^1 g(x)\,\mathrm dx\right)^2\\
&&\quad+
\frac{1}{\|g\|_2^2}
\left(
    \frac{1}{\sqrt{1-\beta\|g\|_2^2}}-1
\right)^2
\left(\int_0^1 g(x)\,\mathrm dx\right)^2\\
&=&
1+
\frac{1}{\|g\|_2^2}
\left[
    \frac{1}{1-\beta\|g\|_2^2}-1
\right]
\left(\int_0^1 g(x)\,\mathrm dx\right)^2\\
&=&
1+
\frac{\beta}{1-\beta\|g\|_2^2}
\left(\int_0^1 g(x)\,\mathrm dx\right)^2,
\end{eqnarray*}
recovering the scaled-magnetization limit in Example \ref{example:rank-one-linear-stat}.
\end{example}

The limiting Gaussian process $X$ also admits a natural dynamical representation in terms of stationary Ornstein--Uhlenbeck processes, providing a principled way to simulate correlated copies of $X$ with tunable dependence; see Appendix \ref{rem:OU-representation-X} for details.
The following finite-dimensional weak convergence result for the process $S_N(t)$ is crucial towards establishing a Donsker-type scaling limit for the same. Its proof is given in Section \ref{sec:SNfindimcv8}.

\begin{prop}\label{lem:donsker_fdd}
For every \(k\ge 1\) and every \(t_1,\ldots,t_k\in[0,1]\), the conditional law of 
\((S_N(t_1),\ldots,S_N(t_k))\) given $A_N$, viewed as a random probability measure on $\R^k$, converges weakly in probability to the law of $(X(t_1),\ldots,X(t_k))$. Consequently,  \((S_N(t_1),\ldots,S_N(t_k))\xrightarrow{d} (X(t_1),\ldots,X(t_k))\).
\end{prop}

We now establish a notion of Donsker-type convergence of the sequence of processes $S_N$ to the process $X$. Towards this, let \(D[0,1]\) denote the space of real-valued c\`adl\`ag functions on
\([0,1]\), i.e.
\[
    D[0,1]
    :=
    \left\{
    f:[0,1]\to\mathbb R:
    f \text{ is right-continuous on }[0,1)
    \text{ and } f(t-) \text{ exists for every } t\in(0,1]
    \right\}.
\]
We equip \(D[0,1]\) with the Skorokhod topology, induced by the metric
\[
    d_{\mathrm{Sk}}(f,g)
    =
    \inf_{\lambda\in\Lambda}
    \left\{
    \sup_{0\le t\le1}|\lambda(t)-t|
    \vee
    \sup_{0\le t\le1}|f(\lambda(t))-g(t)|
    \right\}
\]
where \(\Lambda\) is the collection of strictly increasing bijections
\(\lambda:[0,1]\to[0,1]\). For more details on the space $D[0,1]$ and the Skorokhod topology, see \cite[Chapter 3, Section 12]{billingsley1999convergence}.
The following theorem, proved in Section \ref{proof:donsker}, establishes a Donsker-type scaling limit for the process $S_N$.

    \begin{thm}
\label{thm:donsker}

The conditional law of $S_N$ given $A_N$, viewed as a
$\mathcal P(D[0,1])$-valued random probability measure, where
$D[0,1]$ is equipped with the Skorokhod topology and
$\mathcal P(D[0,1])$ with the corresponding topology of weak convergence,
converges weakly in probability to the law of $X$. In particular, the
unconditional laws of $S_N$ converge weakly to the law of $X$ as probability
measures on $D[0,1]$.
\end{thm}

An illustrative alternative derivation of the unconditional version of Theorem \ref{thm:donsker} for the Curie-Weiss model is given in Section \ref{rem:CW-alternative-donsker} of the Appendix.

\subsection{A Functional Central Limit Theorem for Linear Statistics} We next give a process-level version of Theorem
\ref{thm:isingWeakConv} indexed by a compact class of absolutely
continuous test functions.
\begin{defn}
    A function $f: [0,1] \to \mathbb{R}$ is absolutely continuous if for every $\varepsilon > 0$, there exists a $\delta > 0$ such that for any finite collection of pairwise disjoint subintervals $\{(x_i, y_i)\}_{i=1}^n$ of $[0,1]$ satisfying
\(
\sum_{i=1}^n (y_i - x_i) < \delta
\),
we have
\(
\sum_{i=1}^n |f(y_i) - f(x_i)| < \varepsilon
\).
\end{defn}

Let us denote the class of all absolutely continuous functions $f: [0,1] \to \R$ by $AC[0,1]$.
It follows from the fundamental theorem of calculus for Lebesgue integrals, that every $f \in AC[0,1]$ is differentiable Lebesgue almost everywhere, and its derivative $f' \in L^1[0,1]$ satisfies: 
$$f(t) = f(0) + \int_0^t f'(s)\,\mathrm ds\qquad\text{for all}~t\in [0,1].$$
Let $\mathcal C$ be a compact subset of $AC[0,1]$ with respect to the metric 
\[
d_{AC}(f,g)
\coloneq
|f(1)-g(1)|
+
\int_0^1 |f'(t)-g'(t)|\,\mathrm dt.
\]

For example, for fixed $M,L,H<\infty$ and $\alpha\in(0,1]$, the class
\[
\mathcal C_{M,L,H,\alpha}
\coloneq
\left\{
f\in C^1[0,1]:
\|f\|_\infty\le M,\ 
\|f'\|_\infty\le L,\ 
|f'(t)-f'(s)|\le H|t-s|^\alpha
\text{ for all }s,t\in[0,1]
\right\}
\]
is a compact subset of $AC[0,1]$ under the metric $d_{AC}$ by the Arzel\`a--Ascoli theorem (applied to the class of derivatives $\{f':f\in\mathcal C_{M,L,H,\alpha}\}$), where $C^1[0,1]$ denotes the set of continuously differentiable functions on $[0,1]$. For any fixed $m\geq 1$ and constants
$M_0,\ldots,M_m<\infty$, the finite-dimensional polynomial class
\[
\mathcal{PC}_{m,M_0,\ldots,M_m} := \left\{
\sum_{j=0}^m a_jt^j:
|a_j|\le M_j,\ 0\le j\le m
\right\} \subset \mathcal{C}_{\sum_j M_j, ~\sum_j jM_j,~ \sum_j j(j-1) M_j,~1}
\]
is closed, hence compact under $d_{AC}$. In what follows, we will implicitly assume that the space $AC[0,1]$ (and hence, all its subsets) is equipped with the metric $d_{AC}$.

Next, define the map
\[
\Psi:C[0,1]\longrightarrow C(\mathcal C),
\]
where $C(\mathcal C)$ denotes the set of all real-valued continuous functions on $\mathcal C$, by
\[
\Psi(x)(f)
\coloneq
x(1)f(1)-\int_0^1x(t)f'(t)\,\mathrm dt,
\qquad x\in C[0,1],\ f\in\mathcal C.
\]
Indeed, $\Psi(x)\in C(\mathcal C)$ for every $x\in C[0,1]$, since
\[
|\Psi(x)(f)-\Psi(x)(g)|
\le
\|x\|_\infty d_{AC}(f,g).
\]
We then
define the centered Gaussian process $(\mathbb G(f))_{f\in \mathcal C}$ as:
\[
\mathbb G(f)
\coloneq
\Psi(X)(f)
=
X(1) f(1)-\int_0^1 X(t)f'(t)\,\mathrm dt,
\qquad f\in\mathcal C,
\]
where the process $X$ is defined in \eqref{defXt}. Denote the process $(\mathbb G(f))_{f\in \mathcal C}$ by $\mathbb G$, and let $\mathbb G_N := (\sigma_N(f))_{f\in\mathcal C}$. Note that for any two $f,g \in \mathcal C$,
\[
\begin{aligned}
|\sigma_N(f)-\sigma_N(g)|
&=
\left|
\frac{1}{\sqrt N}
\sum_{i=1}^N
\left(
f\left(\frac{i}{N}\right)
-
g\left(\frac{i}{N}\right)
\right)\sigma_i
\right|\\
&\le
\frac{1}{\sqrt N}
\sum_{i=1}^N
\left|
f\left(\frac{i}{N}\right)
-
g\left(\frac{i}{N}\right)
\right|\\
&\le
\sqrt N\,\|f-g\|_\infty\\
&\le
\sqrt N\,d_{AC}(f,g).
\end{aligned}
\]
This shows that the map $f\to \sigma_N(f)$ is continuous on the domain $\mathcal{C}$, and hence, $\mathbb G_N \in C(\mathcal{C})$. The following theorem establishes weak convergence of the process $\mathbb G_N$ to the process $\mathbb G$ in probability.

\begin{thm}\label{thm:process_class}
Let $\mathcal{C}$ be a compact subset of the metric space $(AC[0,1],d_{AC})$. Then, the conditional law of
$\mathbb G_N$ given $A_N$, viewed as a
$\mathcal P(C(\mathcal C))$-valued random probability measure, converges
weakly in probability to the law of $\mathbb G$. In particular, the
unconditional laws of $\mathbb G_N$ converge weakly
to the law of $\mathbb G$ as probability measures on $C(\mathcal C)$.
\end{thm}

The proof of Theorem \ref{thm:process_class} is given in Section \ref{proff:thmproclass222}. Related function-indexed weak convergence results, based on integration-by-parts representations and bounded-variation classes, appear in \cite{radulovic2015copula,radulovic2017stationary}. Gaussian limits for linear statistics indexed by regular test functions have also been studied extensively in random matrix theory; see, for example, \cite{RiderVirag2007,sosoe2013regularity}.

\subsection{A Distributional Limit for the Empirical Spin Field} We finally complement
Theorems~\ref{thm:isingWeakConv}--\ref{thm:process_class}
with convergence of the full empirical spin field 
\[
\eta_N = \frac1{\sqrt N}
\sum_{i=1}^N\sigma_i\delta_{i/N}
\]
in a negative
Sobolev space. Negative-Sobolev and related generalized-function formulations of
spectral fluctuations have appeared in several parts of random matrix
theory, including for circular Dyson Brownian motion, the logarithms
of characteristic polynomials of the CUE and GUE, Wigner
eigenvalue-counting fields, and the Ginibre ensemble
\cite{Spohn1998Dyson,HughesKeatingOConnell2001,
FyodorovKhoruzhenkoSimm2016,CipolloniLopatto2026,RiderVirag2007}.
These works include weak-convergence results in explicit Sobolev or
closely related Hilbert spaces, as well as Gaussian-free-field limits
formulated through test-function pairings, and provide useful
precedents for treating spectral fluctuations as random generalized
functions rather than only through finitely many scalar observables.

\subsubsection{Sobolev embedding and canonical representation of signed measures}\label{sobolevcan8} To describe our setup, let us denote by
\[
e_0(x):=1,
\qquad
e_k(x):=\sqrt{2}\cos(\pi kx),\quad k\geq1,
\]
the cosine orthonormal basis of \(L^2[0,1]\), and set
\(\kappa_k:=1+\pi^2k^2\) for \(k\geq0\).
For \(r\in\mathbb R\), let \(H^r(0,1)\) denote the Hilbert space
obtained by completing
\(S^\#:=\operatorname{span}\{e_k:k\geq0\}\) under the norm
\[
\|u\|_{H^r}^2
:=
\sum_{k=0}^{\infty}
\kappa_k^r|\langle u,e_k\rangle|^2,
\qquad u\in S^\#.
\]
Let \(L\) denote the nonnegative self-adjoint Neumann Laplacian on
\(L^2[0,1]\), acting as \(Lu=-u''\) on its domain, subject to the
boundary conditions
\(u'(0)=u'(1)=0\). Since
\[
Le_k=\pi^2k^2e_k
\qquad\text{for }k\geq0,
\]
the preceding norm may equivalently be written, for \(u\in S^\#\), as
\[
\|u\|_{H^r}
=
\bigl\|(I+L)^{r/2}u\bigr\|_2.
\]
Thus, in the terminology of
\cite[Section~2]{MikhailetsMurachZinchenko2021},
\(\{H^r(0,1)\}_{r\in\mathbb R}\) is the Hilbert scale generated by
\((I+L)^{1/2}\). We may thus naturally refer to this family as the
inhomogeneous spectral Sobolev scale associated with the Neumann
Laplacian. For the analogous construction of spectral spaces from the
eigenfunctions and eigenvalues of a Laplacian with prescribed boundary
conditions, see
\cite[Section~2.1, equations~(2.4)--(2.5)]{CusimanoEtAl2018}.

Next, for a finite signed Borel measure \(\nu\) on \([0,1]\), set
\[
\widehat{\nu}_k
:=
\int_{[0,1]}e_k\,d\nu\quad (k\geq0)\qquad \text{and} \qquad \iota_m\nu
:=
\sum_{k=0}^m\widehat{\nu}_k e_k.
\]
It follows from Lemma \ref{soblimexists} that the limit
\(
\iota_s\nu
:=
\lim_{m\to\infty}\iota_m\nu
\)
exists in \(H^{-s}(0,1)\) for every $s>1/2$, and if \(\mathcal M([0,1])\) denotes the set of
all finite signed Borel measures on \([0,1]\), then the map
\[
\iota_s:\mathcal M([0,1])\to H^{-s}(0,1)
\]
is linear and injective.
We therefore identify \(\nu\) with its canonical image
\(\iota_s\nu\) in \(H^{-s}(0,1)\). In particular, for every
\(s>1/2\),
\[
\|\nu\|_{H^{-s}}^2
=
\lim_{m\to\infty}\|\iota_m\nu\|_{H^{-s}}^2
=
\lim_{m\to\infty}
\sum_{k=0}^m
\kappa_k^{-s}|\widehat{\nu}_k|^2
=
\sum_{k=0}^{\infty}
\kappa_k^{-s}
\left|
\int_{[0,1]}e_k\,d\nu
\right|^2
\leq
2\|\nu\|_{\mathrm{TV}}^2
\sum_{k=0}^{\infty}\kappa_k^{-s}
<\infty.
\]

In this sense, the random signed measure $\eta_N$
can be viewed canonically as an \(H^{-s}(0,1)\)-valued random element
for every \(s>1/2\), through its image \(\iota_s\eta_N\). We
henceforth suppress the embedding \(\iota_s\) and denote
this random element by \(\eta_N\).

\subsubsection{The limiting Sobolev element}\label{sec:limelem688} Under Assumption~\ref{assumption12},
\[
R_\beta:=(I-\beta T_W)^{-1}
\]
is a bounded, self-adjoint, strictly positive operator on
\(L^2[0,1]\). Choose an orthonormal basis
\((\psi_j)_{j\geq1}\) of \(L^2[0,1]\), let
\((Z_j)_{j\geq1}\) be i.i.d. standard normal random variables defined
on an auxiliary probability space
\((\Omega^*,\mathcal F^*,\mathbb P^*)\), and define, for
\(f\in L^2[0,1]\),
\begin{equation}\label{eqchifser}
\chi(f)
:=
\sum_{j=1}^\infty
\langle R_\beta^{1/2}f,\psi_j\rangle Z_j.
\end{equation}
Lemma \ref{lchiprop8} guarantees $L^2(\Omega^*)$ convergence of the RHS of \eqref{eqchifser}, and the following form for the covariance of the centered Gaussian linear process $\chi$:

$$\e^*[\chi(f)\chi(g)] = \langle f, R_\beta g\rangle.$$
Hence, for every \(s>1/2\), we have
\[
\sum_{k=0}^\infty
\mathbb E^*\!\left[
\kappa_k^{-s}\chi(e_k)^2
\right]
=
\sum_{k=0}^\infty
\kappa_k^{-s}\mathbb E^*[\chi(e_k)^2]
=
\sum_{k=0}^\infty
\kappa_k^{-s}\langle e_k,R_\beta e_k\rangle
\leq
\|R_\beta\|_{\mathrm{op}}
\sum_{k=0}^\infty\kappa_k^{-s}
<\infty.
\]
Therefore, the partial sums
\(\sum_{k=0}^n\chi(e_k)e_k\) are Cauchy in
\(L^2(\Omega^*;H^{-s}(0,1))\), and so the series
\[
\eta
:=
\sum_{k=0}^\infty\chi(e_k)e_k
\]
converges in \(L^2(\Omega^*;H^{-s}(0,1))\). Thus, for each fixed
\(s>1/2\), \(\eta\) is an \(H^{-s}(0,1)\)-valued random element, i.e.
\[
\eta\in H^{-s}(0,1)
\qquad
\mathbb P^*\text{-almost surely}.
\]

\subsubsection{Weak convergence of empirical fields} The following theorem, proved in Section \ref{sec:proofsobolevmain34}, establishes weak convergence in probability of
the conditional laws of the signed empirical measures \(\eta_N\),
viewed through their canonical embeddings as random elements of
\(H^{-s}(0,1)\), to the law of the \(H^{-s}(0,1)\)-valued random
element \(\eta\).
\begin{thm}
\label{thm:sob_conv}
For every \(s>1/2\), the conditional law of \(\eta_N\) given \(A_N\),
viewed as a
\(\mathcal P(H^{-s}(0,1))\)-valued random probability measure, where
\(H^{-s}(0,1)\) is equipped with the topology induced by the norm $\|\cdot\|_{H^{-s}}$,
and \(\mathcal P(H^{-s}(0,1))\) with the corresponding topology of weak
convergence, converges weakly in probability to the law of \(\eta\)
under \(\mathbb P^*\). In particular, the unconditional laws of
\(\eta_N\) converge weakly to the law of \(\eta\) as probability
measures on \(H^{-s}(0,1)\).
\end{thm}

\begin{remark}[Weak convergence of Sobolev actions]\label{remk:sobac72}
The preceding theorem, with \(s=1\), gives weak convergence in
probability of the conditional laws of the full empirical field in
\(H^{-1}(0,1)\). We now record the corresponding weaker convergence
of its actions on a fixed compact class of \(H^1(0,1)\) test
functions, which follows directly from
Theorem~\ref{thm:process_class}. This is weaker in the sense that it
retains only the actions of the field on the prescribed class, rather
than its full law in the \(H^{-1}(0,1)\)-topology.

Let \((\Omega^*,\mathcal F^*,\mathbb P^*)\) be the auxiliary
probability space on which \(Z_1,Z_2,\ldots\), and hence, \(\chi\) and
\(\eta\), are defined. Since \(\eta\in H^{-1}(0,1)\) almost surely,
its action on \(H^1(0,1)\) can be defined via the
\(H^{-1}\)--\(H^1\) duality pairing, i.e., for any
\(h\in H^1(0,1)\),
\begin{equation}\label{etahactiondef6}
\langle\eta,h\rangle_{-1,1}
:=
\sum_{k=0}^{\infty}
\chi(e_k)\langle h,e_k\rangle.
\end{equation}

It follows from Lemma \ref{sobactlem552} that the series in the RHS of \eqref{etahactiondef6} is absolutely convergent almost surely, and that for each fixed \(h\in H^1(0,1)\), the two random variables
\(\langle\eta,h\rangle_{-1,1}\) and \(\chi(h)\) are equal
\(\mathbb P^*\)-almost surely. The centered Gaussian process
\(t\mapsto\chi(\mathbf 1_{[0,t]})\) has covariance
\[
\operatorname{Cov}\left(
\chi(\mathbf 1_{[0,t]}),
\chi(\mathbf 1_{[0,s]})
\right)
=
\left\langle
\mathbf 1_{[0,t]},
R_\beta\mathbf 1_{[0,s]}
\right\rangle
=
K(s,t),
\]
where \(K\) is the covariance kernel of \(X\), as established in
Lemma~\ref{constructXt}. Hence it has the same finite-dimensional
distributions as \(X\). Moreover, the proof of
Lemma~\ref{holderX} applies verbatim and yields a continuous
modification of \(t\mapsto\chi(\mathbf 1_{[0,t]})\). Replacing \(X\)
by an equal-in-law copy, we may therefore realize \(X\) on
\((\Omega^*,\mathcal F^*,\mathbb P^*)\) as this continuous
modification.

Now, every \(h\in H^1(0,1)\), identified with its absolutely
continuous representative, satisfies
\begin{equation}
\label{h01abctr2}
h
=
h(1)\mathbf 1_{[0,1]}
-
\int_0^1h'(t)\mathbf 1_{[0,t]}\,dt
\end{equation}
in \(L^2[0,1]\). Applying the bounded linear map
\(\chi:L^2[0,1]\to L^2(\Omega^*)\), and using
\(\chi(\mathbf 1_{[0,t]})=X(t)\) in \(L^2(\Omega^*)\), gives
\[
\langle\eta,h\rangle_{-1,1}
=
\chi(h)
=
h(1)X(1)-\int_0^1X(t)h'(t)\,dt
\qquad\text{in }L^2(\Omega^*).
\]
Since \(X\) has continuous paths and \(h'\in L^1[0,1]\), the integral
on the right also exists pathwise.

Now let \(\mathcal C\) be a compact subset of \(H^1(0,1)\), equipped
with the \(H^1\)-norm topology, and equip \(C(\mathcal C)\) with the
supremum norm. For every \(u\in H^{-1}(0,1)\),
\[
|\langle u,h\rangle_{-1,1}
-
\langle u,g\rangle_{-1,1}|
\leq
\|u\|_{H^{-1}}\|h-g\|_{H^1},
\qquad h,g\in\mathcal C,
\]
so both
\[
\bigl(\langle\eta_N,h\rangle_{-1,1}\bigr)_{h\in\mathcal C}
\quad\text{and}\quad
\bigl(\langle\eta,h\rangle_{-1,1}\bigr)_{h\in\mathcal C}
\]
are \(C(\mathcal C)\)-valued random elements. Moreover, under the
canonical embedding of finite signed measures into \(H^{-1}(0,1)\),
we have
\[
\eta_N
=
\sum_{k=0}^{\infty}
\left(
\int_{[0,1]}e_k\,d\eta_N
\right)e_k
\qquad\text{in }H^{-1}(0,1),
\]
and hence
\[
\begin{aligned}
\langle\eta_N,h\rangle_{-1,1}
&=
\sum_{k=0}^{\infty}
\left(
\int_{[0,1]}e_k\,d\eta_N
\right)
\langle h,e_k\rangle\\
&=
\int_{[0,1]}h\,d\eta_N
=
\frac1{\sqrt N}
\sum_{i=1}^N
\sigma_i h\left(\frac{i}{N}\right)
=
\sigma_N(h),
\qquad h\in H^1(0,1).
\end{aligned}
\]
The second equality follows by passing to the finite cosine sums,
which converge uniformly to the continuous representative of
\(h\in H^1(0,1)\).

Now, it is not hard to show that
\[
d_{AC}(h,g)
=
|h(1)-g(1)|+\|h'-g'\|_{L^1}
\leq
C\|h-g\|_{H^1}
\]
for some constant \(C>0\). Consequently, \(\mathcal C\) is also
compact under \(d_{AC}\). Moreover, the \(H^1\)- and
\(d_{AC}\)-topologies agree on \(\mathcal C\).

Since \(\mathcal C\) is separable and both limiting indexed processes
have continuous sample paths, the preceding pointwise identity
identifies them \(\mathbb P^*\)-almost surely as
\(C(\mathcal C)\)-valued random elements. Therefore,
Theorem~\ref{thm:process_class} directly gives
\[
\mathcal L\left(
\bigl(\langle\eta_N,h\rangle_{-1,1}\bigr)_{h\in\mathcal C}
\Big|A_N
\right)
\xrightarrow{d}
\mathcal L_{\mathbb P^*}\left(
\bigl(\langle\eta,h\rangle_{-1,1}\bigr)_{h\in\mathcal C}
\right)
\]
in probability as probability measures on \(C(\mathcal C)\). 
\end{remark}

\section{Applications to Bayesian Neural Networks and Causal Inference}
In this section, we give two applications of the main results in this paper.

\subsection{Ising-Dependent Bayesian Neural Networks }
\label{sec:bayesian-neural-networks}

Let \(n\geq 1\), and let \(a:\mathbb R\to\mathbb R\) be a
measurable activation function satisfying
\begin{equation}\label{momentasmp82}
    \mathbb E[a(rZ)^2]<\infty
    \qquad\text{for every }r\geq 0,
\end{equation}
where \(Z\sim\mathcal N(0,1)\). For each \(k\), let \(A_k\) be the
adjacency matrix of a graph sampled from \(G(k,\theta_k,W)\).
Conditional on \(A_k\), let
\(\bm s=(s_1,\ldots,s_k)\) have the Ising law
\(\mathbb P_{\beta,\theta_k,W}\) in \eqref{model_def}. Independently
of \((A_k,\bm s)\), let
\(\bm w_1,\ldots,\bm w_k\stackrel{\mathrm{i.i.d.}}{\sim}
\mathcal N(\boldsymbol 0,\bm I_n)\). We consider the random
two-layer network
\begin{equation*}
    \mathcal P_k^a(\bm x)
    \coloneq
    \frac{1}{\sqrt{k}}
    \sum_{i=1}^k
    s_i a(\bm w_i^\top \bm x),
    \qquad \bm x\in\mathbb R^n.
\end{equation*}
For independent weight priors, the corresponding infinite-width
Gaussian-process limit is classical; see
\cite[Chapter~2]{neal2012bayesian}. Here the output-layer signs are
instead coupled through the Ising measure.

Let \(\bm Z\sim\mathcal N(\boldsymbol 0,\bm I_n)\), and define
\[
    m_a(\bm x)
    \coloneq
    \mathbb E a(\bm Z^\top \bm x),
    \qquad
    C_a(\bm x,\bm y)
    \coloneq
    \operatorname{Cov}\left(
        a(\bm Z^\top\bm x),
        a(\bm Z^\top\bm y)
    \right).
\]
Also set
\[
    R_\beta\coloneq(I-\beta T_W)^{-1},
    \qquad
    q_\beta\coloneq
    \langle\mathbf 1,R_\beta\mathbf 1\rangle,
\]
and introduce the covariance kernel
\begin{equation}
    K_{\beta,a}(\bm x,\bm y)
    \coloneq
    C_a(\bm x,\bm y)
    +
    q_\beta m_a(\bm x)m_a(\bm y).
    \label{eq:ising-neural-kernel}
\end{equation}

\begin{thm}
\label{thm:ising-neural-network}
Suppose that \(\beta, W\) and $\theta_k$ satisfy Assumption~\ref{assumption12}. Then, for every
\(m\geq1\) and every
\(\bm x_1,\ldots,\bm x_m\in\mathbb R^n\), the conditional law of
\(
    \bigl(
        \mathcal P_k^a(\bm x_1),\ldots,
        \mathcal P_k^a(\bm x_m)
    \bigr)
\)
given \(A_k\) converges weakly in probability to
\(\mathcal N_m(\boldsymbol 0,\bm K)\), where
\[
    \bm K_{bc}
    =
    K_{\beta,a}(\bm x_b,\bm x_c),
    \qquad 1\leq b,c\leq m.
\]
Consequently, the unconditional finite-dimensional distributions of
\(\mathcal P_k^a\) converge to those of the centered Gaussian
process with covariance kernel \(K_{\beta,a}\). In particular, for
every \(\bm x\in\mathbb R^n\),
\[
    \mathcal P_k^a(\bm x)
    \xrightarrow{d}
    \mathcal N\left(
        0,
        \operatorname{Var} 
            a(\bm Z^\top\bm x)
        +
        q_\beta
            \mathbb E^2
                a(\bm Z^\top\bm x)
    \right).
\]
\end{thm}

Theorem \ref{thm:ising-neural-network} is proved in Section \ref{nntwproof} of the appendix.

\begin{remark}[Effect of Ising dependence]
The limiting kernel \eqref{eq:ising-neural-kernel} can be written as
\begin{equation}\label{altexpker66}
     K_{\beta,a}(\bm x,\bm y)
    =
    \mathbb E\left[
        a(\bm Z^\top\bm x)
        a(\bm Z^\top\bm y)
    \right]
    +
    (q_\beta-1)
    m_a(\bm x)m_a(\bm y).
\end{equation}
Thus, relative to the classical kernel obtained under independent
output-layer signs, the Ising coupling contributes only the rank-one
term
\(
(q_\beta-1)m_a(\bm x)m_a(\bm y).
\) At \(\beta=0\) (which corresponds to i.i.d. Rademacher output-layer signs), \(q_0=1\), recovering the
usual neural-network Gaussian-process kernel \cite[Chapter~2, Section~2.1.1, equations~(2.3)--(2.4)]
{neal2012bayesian}. Also, if
\(m_a(\bm x)=0\) for every \(\bm x\) (which is true if the activation function $a$ is odd), the Ising dependence does not
affect the first-order Gaussian-process limit.
\end{remark}

\begin{remark}[ReLU activation]
The ReLU activation \(a(t):=t^+\) is unbounded and is therefore not
covered by the bounded-activation formulation explicitly stated in
\cite[Section~2.1.1]{neal2012bayesian}. It nevertheless
satisfies the second-moment condition of
Theorem~\ref{thm:ising-neural-network}, since
\[
    \mathbb E\left[
        a(\bm Z^\top\bm x)^2
    \right]
    =
    \frac{\|\bm x\|_2^2}{2}.
\]
For nonzero \(\bm x,\bm y\in\mathbb R^n\), let
\[
    \vartheta_{\bm x,\bm y}
    \coloneq
    \cos^{-1}\left(
        \frac{\bm x^\top \bm y}
        {\|\bm x\|_2\|\bm y\|_2}
    \right).
\]
Then the limiting covariance is
\begin{equation}\label{reluker22}
    K_{\beta,\mathrm{ReLU}}(\bm x,\bm y)
    =
    \frac{\|\bm x\|_2\|\bm y\|_2}{2\pi}
    \left[
        \sin\vartheta_{\bm x,\bm y}
        +
        (\pi-\vartheta_{\bm x,\bm y})
        \cos\vartheta_{\bm x,\bm y}
        +
        q_\beta-1
    \right].
\end{equation}
Indeed, if
\(U:=\langle\bm Z,\bm x\rangle\) and
\(V:=\langle\bm Z,\bm y\rangle\), then it is straightforward to verify that
\[
     \mathbb E U^+
    =
    \frac{\|\bm x\|_2}{\sqrt{2\pi}},
    \qquad
    \mathbb E V^+
    =
    \frac{\|\bm y\|_2}{\sqrt{2\pi}},\qquad \mathbb E (U^+V^+)
    =
    \frac{\|\bm x\|_2\|\bm y\|_2}{2\pi}
    \left[
        \sin\vartheta_{\bm x,\bm y}
        +
        (\pi-\vartheta_{\bm x,\bm y})
        \cos\vartheta_{\bm x,\bm y}
    \right].
\]
The claimed expression \eqref{reluker22} now follows from
\eqref{altexpker66}. At \(\beta=0\), \(q_0=1\), and the kernel \eqref{reluker22} reduces to the standard
ReLU neural-network Gaussian-process covariance; see
\cite[Appendix~B, Eq.~(11)]{lee2018deep}.
\end{remark}

\subsection{Causal Inference with Dependent Treatment Assignment and Outcome Interference}
\label{sec:causal-graphon-ising}

Recently, there has been growing interest in causal inference under cross-unit dependence and interference, motivated by applications in social, epidemiological, and economic networks. In such settings, one unit’s treatment may affect the outcomes of other units and, in some cases, the treatment assignments themselves may be dependent; see, for example,
\cite{bhadra2025causal,bhattacharya2025causal,
ogburn2024causal,forastiere2021identification,leung2022approximate,li2022randomgraph,tchetgen2012interference,hu2022directindirect,hudgens2008interference,park2026convergent,paschalidis2026sequential}.
Valid causal inference in such contexts requires careful modeling of both the treatment-assignment mechanism and the resulting outcome interference. 

We apply our fluctuation theory to a setting in
which treatment assignments follow the network-dependent Ising model
considered in this paper, while outcomes exhibit interference through
the fraction of treated peers. Cattaneo, He, and Yu \cite{cattaneo2025robust} study the classical
H\'ajek estimator when outcomes exhibit network interference and the
treatment assignments follow a Curie--Weiss law. In the
high-temperature regime, they obtain a Gaussian limit whose variance
contains an additional contribution from treatment-assignment
dependence. We extend their high-temperature Gaussian limit by replacing the
homogeneous Curie--Weiss treatment-assignment law with the
inhomogeneous Ising assignment framework
of this paper. Thus, treatment assignments may now interact through
a general graphon-induced network rather than through homogeneous
all-to-all coupling.
To isolate this extension, we take the outcome-interference network
to be complete, so that each potential outcome depends on the
fraction of treated peers. This remains a genuine interference model
and is a special case of the smooth exposure model in
\cite{cattaneo2025robust}.

Consider a population of \(N\) units, each of which is assigned either
to treatment or control. We encode the assignment of unit \(i\) by
\[
    D_i\coloneq\frac{1+\sigma_i}{2}\in\{0,1\},
    \qquad i=1,\ldots,N,
\]
so that \(D_i=1\) denotes treatment and \(D_i=0\) denotes control.
Conditional on the observed network \(A_N\), the assignment vector
\(\bm D=(D_1,\ldots,D_N)\) is generated from the Ising law
\(\bm\sigma\sim\mathbb P_{\beta,\theta_N,W}\). Thus, unlike under
independent randomization, the treatment assignments may be dependent
across units, with the dependence structure inherited from the
underlying network.

We next allow for outcome interference. For each unit \(i\), let
\(F_i\) describe how its outcome depends both on its own treatment and
on the overall treatment exposure among the remaining units. We assume
that \(F_1,F_2,\ldots\) are i.i.d. random functions, independent of
\((A_N,\bm\sigma)\), with \(F_i(d,\cdot)\in C^2[0,1]\) for
\(d\in\{0,1\}\), and that, for some deterministic \(C<\infty\),
\[
    \max_{d\in\{0,1\}}
    \sup_{x\in[0,1]}
    \sum_{r=0}^2
    \left|\partial_2^rF_i(d,x)\right|
    \leq C
    \qquad\text{almost surely},
\]
where $\partial_2^r F_i$ denotes the $r^\mathrm{th}$ derivative of $F_i$ with respect to its second argument ($x$).
For an assignment vector
\(\bm d_{-i}\in\{0,1\}^{N-1}\) of all units other than \(i\), define
the potential outcome of unit \(i\) under own treatment \(d\) by
\[
    Y_i(d;\bm d_{-i})
    \coloneq
    F_i\left(
        d,
        \frac{1}{N-1}\sum_{j\ne i}d_j
    \right).
\]
Hence, the outcome of unit \(i\) may depend not only on its own
treatment, but also on the fraction of its peers who are treated. The
observed outcome is
\(
    Y_i
    =
    Y_i(D_i;\bm D_{-i}).
\)

We use the classical H\'ajek difference-in-means estimator for the
conditional direct average treatment effect defined below. This conditional
target captures the effect of switching a unit from control to treatment,
averaged over its peers' assignments. The estimator is given by
\begin{equation}
    \widehat\tau_N
    \coloneq
    \frac{\sum_{i=1}^N D_iY_i}{\sum_{i=1}^N D_i}
    -
    \frac{\sum_{i=1}^N(1-D_i)Y_i}
         {\sum_{i=1}^N(1-D_i)}.
    \label{eq:hajek-interference}
\end{equation}
We center it by the following conditional target
\begin{equation}
    \tau_N
    \coloneq
    \frac1N\sum_{i=1}^N
    \mathbb E\left[
        Y_i(1;\bm D_{-i})-Y_i(0;\bm D_{-i})
        |
        A_N,F_1,\ldots,F_N
    \right].
    \label{eq:conditional-direct-effect}
\end{equation}

Next, set
\[
    R_i
    \coloneq
    F_i\left(1,\frac12\right)
    +
    F_i\left(0,\frac12\right),
    \qquad
    { Q
    \coloneq
    \frac12\,
    \mathbb E\left[
        \partial_2F_1\left(1,\frac12\right)
        -
        \partial_2F_1\left(0,\frac12\right)
    \right]},
\]
and define
\[
    L_i\coloneq R_i-\mathbb E[R_1]+Q,
    \qquad
    \kappa_r\coloneq\mathbb E[L_1^r],
    \quad r=1,2.
\]
In particular, \(\kappa_1=Q\). Finally, let
\[
    R_\beta\coloneq(I-\beta T_W)^{-1},
    \qquad
    q_\beta
    \coloneq
    \left\langle
        \mathbf 1,R_\beta\mathbf 1
    \right\rangle .
\]

\begin{thm}
\label{thm:hajek-graphon-ising}
Suppose that \(\beta, W\) and \(\theta_N\) satisfy
Assumption~\ref{assumption12}.
Then the two denominators in \eqref{eq:hajek-interference} are
nonzero with probability tending to one. Moreover, the conditional
law of
\(
    \sqrt N\bigl(\widehat\tau_N-\tau_N\bigr)
\)
given \(A_N\) converges weakly in probability to
\(
    \mathcal N\left(
        0,
        \kappa_2+\kappa_1^2(q_\beta-1)
    \right).
\)
Consequently,
\[
    \sqrt N\bigl(\widehat\tau_N-\tau_N\bigr)
    \xrightarrow{d}
    \mathcal N\left(
        0,
        \kappa_2+\kappa_1^2(q_\beta-1)
    \right).
\]
\end{thm}

Theorem \ref{thm:hajek-graphon-ising} is proved in Section \ref{cstwproof} of the appendix.

\begin{remark}[Erd\H{o}s--R\'enyi and Curie--Weiss assignments]
If \(W\equiv1\), then \(T_W\mathbf 1=\mathbf 1\), and hence
\[
    q_\beta
    =
    \left\langle
        \mathbf 1,(I-\beta T_W)^{-1}\mathbf 1
    \right\rangle
    =
    \frac{1}{1-\beta}~.
\]
Therefore, for every sparsity sequence \(\theta_N\) satisfying
Assumption~\ref{assumption12}\textup{(3)}, the limiting variance in
Theorem~\ref{thm:hajek-graphon-ising} becomes
\begin{equation}\label{limasmvar66}
    \kappa_2
    +
    \kappa_1^2\frac{\beta}{1-\beta}~.
\end{equation}
Thus, the same asymptotic variance holds for Erd\H{o}s--R\'enyi
Ising treatment assignment throughout the sparsity regime considered
in this paper. In the special case \(\theta_N=1\), the assignment
graph is complete and the treatment-assignment law is Curie--Weiss;
\eqref{limasmvar66} then recovers exactly the high-temperature
variance in
\cite[Theorem~3.1, Eq.~(6)]{cattaneo2025robust}, specialized to the
complete outcome-interference network.
\end{remark}

\section{Proofs of the Main Results}\label{sec:proofmain2}
In this section, we prove the main results in this paper.

\subsection{Proof of Theorem \ref{thm:isingWeakConv}}\label{proofthmisingWeakConv}
To begin with, assume that $k=1$ and for ease of notation, denote $f_1$ simply by $f$.
Since
\(
    \beta\|T_W\|_{\mathrm{op}}<1,
\)
we may choose $t_0>0$ such that
\begin{equation}\label{eq:t0-absolute-kernel}
 \beta\|T_W\|_{\mathrm{op}}
    +
    2t_0\|f\|_2^2
    <1.   
\end{equation}
Throughout the remainder of the one-dimensional argument, fix an
arbitrary $t\in(-t_0,t_0)$. 
Now, the moment-generating function (m.g.f.) of $\sigma_N(f)^2$ conditional on $A_N$ at the point $t$ can be expressed as: 
\begin{equation*}
    \psi_N(t) = \frac{1}{2^N} \sum_{\bs\in \{-1,1\}^N} \frac{e^{t\sigma_N(f)^2 + \frac{\beta}{2N\theta_N}\bs^\top A_N \bs}}{Z_N(\beta)}=\frac{1}{2^N} \sum_{\bs\in \{-1,1\}^N} \frac{e^{t\sigma_N(f)^2} T(\bs)}{\hat{Z}_N(\beta)},
\end{equation*}
where \[
    T(\bs) \coloneqq \exp \left(
    \gamma \sum_{i,j=1}^N A_N(i,j)\,\sigma_i \sigma_j
    - \gamma \xi_N
    - \gamma^2 \zeta_N
    \right),\]

\[\quad \xi_N \coloneqq \sum_{i=1}^N A_N(i,i)
    , \quad
    \zeta_N \coloneqq \sum_{i,j=1}^N A_N(i,j),\quad \gamma \coloneqq \frac{\beta}{2N\theta_N}
\]
and
\begin{equation}\label{def:znhatising}
    \hat{Z}_N(\beta) \coloneqq \frac{1}{2^N} \sum_{\bs \in \{-1,+1\}^N} T(\bs)
    = Z_N(\beta)\, \exp\!\left( -\gamma \xi_N - \gamma^2 \zeta_N \right).
\end{equation}

Let us now define: $\bldf \coloneqq (f\left(\frac{1}{N}\right),f\left(\frac{2}{N}\right),\ldots, f(1))^\top$ and
\begin{equation}\label{def:znhatfising}
 \hat{Z}^f_N(\beta,t) \coloneqq \frac{1}{2^N} \sum_{\bs\in \{-1,+1\}^N} e^{\frac{t}{N}\bs^T \bldf \bldf^T\bs} T(\bs).
\end{equation}
In this notation, we have:
\begin{equation*}
    \psi_N(t) = \frac{\hat{Z}^f_N(\beta,t)}{\hat{Z}_N(\beta)}.
\end{equation*}
By Lemma \ref{var_sum}, which applies for every
$t\in(-t_0,t_0)$ by \eqref{eq:t0-absolute-kernel}, we have: 
\begin{equation}\label{psinreat66}
    \psi_N(t) = \frac{\E\left[\hat{Z}^f_N(\beta,t)\right]}{\E\left[\hat{Z}_N(\beta)\right]} (1+o_P(1)).
\end{equation}
Now, by Lemma \ref{lem:expect2}, we have:
\begin{multline*}
    \E\left[\hat{Z}^f_N(\beta,t)\right]\\ = \frac{1}{2^N}\sum_\bs \exp\left(-\frac{\beta^2}{4N^2\theta_N}\sum_{i=1}^{N}W\left(\frac{i}{N},\frac{i}{N}\right)
    -\frac{\beta^2}{2N^2}\sum_{i> j}W^2\left(\frac{i}{N},\frac{j}{N}\right)-\frac{\beta^4}{12 N^4\theta_N^3}\sum_{i> j} W\left(\frac{i}{N},\frac{j}{N}\right)\right.\\+o\left(\frac 1{N\theta_N}\right)
    \left.+\frac{\beta}{N}\sum_{i> j}W\left(\frac{i}{N},\frac{j}{N}\right)\sigma_i \sigma_j+ \frac{t}{N}\sum_{i,j}f\left(\frac{i}{N}\right)f\left(\frac{j}{N}\right)\sigma_i \sigma_j + O(Q_N(\bs))\right).
\end{multline*}
where $Q_N(\bs) \coloneqq \frac{1}{N^3 \theta_N^2}\left|\sum_{i>j}W\left(\frac{i}{N},\frac{j}{N}\right)\sigma_i \sigma_j\right|$.
Therefore again by Lemma \ref{lem:expect2} and \eqref{psinreat66}, we have: 
\begin{align}\label{eq:psiNlimit}
    &\psi_N(t)\nonumber\\ =& \frac{\frac{1}{2^N}\sum_\bs \exp\left(\frac{\beta}{N}\sum_{i> j}W\left(\frac{i}{N},\frac{j}{N}\right)\sigma_i \sigma_j+ \frac{t}{N}\sum_{i,j}f\left(\frac{i}{N}\right)f\left(\frac{j}{N}\right)\sigma_i \sigma_j + O(Q_N(\bs))\right)}{\frac{1}{2^N}\sum_\bs \exp\left(\frac{\beta}{N}\sum_{i> j}W\left(\frac{i}{N},\frac{j}{N}\right)\sigma_i \sigma_j+ O(Q_N(\bs))\right)}(1+ o_P(1))            \nonumber \\
    =              & e^{\frac{t}{N}\sum_{i=1}^{N}f(\frac{i}{N})^2}\frac{\E_\mu \exp\left(\frac{\beta}{N}\sum_{i> j}W\left(\frac{i}{N},\frac{j}{N}\right)\sigma_i \sigma_j+ \frac{2t}{N}\sum_{i>j}f\left(\frac{i}{N}\right)f\left(\frac{j}{N}\right)\sigma_i \sigma_j + O(Q_N(\bs)) \right)}{\E_\mu \exp\left(\frac{\beta}{N}\sum_{i> j}W\left(\frac{i}{N},\frac{j}{N}\right)\sigma_i \sigma_j+ O(Q_N(\bs)) \right)}(1 + o_P(1)),
\end{align}
where $\mu$ denotes the uniform measure on $\{-1,+1\}^N$.

Next, for a symmetric kernel $K$, define
\begin{equation}\label{circ22}
    Y_{f,N} \coloneqq \exp\left(\frac{1}{N}\sum_{i> j}K\left(\frac{i}{N},\frac{j}{N}\right)\sigma_i \sigma_j + O(Q_N(\bs))\right).
\end{equation}

Now, let $ K(x,y) \coloneqq \beta W(x,y) + 2t f(x) f(y)$. Set
\(
    \rho_t
    \coloneq
    \beta\|T_W\|_{\mathrm{op}}
    +
    2|t|\|f\|_2^2.
\)
Since $W\geq0$, we have 
\[
    |K(x,y)|
    \leq
    \beta W(x,y)
    +
    2|t|\,|f(x)||f(y)|.
\]
Consequently,
\begin{equation}\label{eq:absolute-kernel-operator-bound}
    \|T_{|K|}\|_{\mathrm{op}}
    \leq
    \beta\|T_W\|_{\mathrm{op}}
    +
    2|t|\|f\|_2^2
    =
    \rho_t
    <1,
\end{equation}
where the final inequality follows from
\eqref{eq:t0-absolute-kernel} and $|t|<t_0$.
Lemma \ref{circlem} now gives   
\begin{equation}
    \label{Y_fNlimit}
    \lim_{N \to \infty} \E_\mu Y_{f,N} =  \frac{1}{\sqrt{\dett(I - T_K)}}~.
\end{equation}
Also, note that with $Y_N :=Y_{0,N}$, we have:
\begin{equation}
    \label{Y_Nlimit}
    \lim_{N \to \infty} \E_\mu Y_{N} = \prod_{i=1}^{\infty} \frac{\exp \left(-\frac{\beta \lambda_{i}(W)}{2}\right)}{\sqrt{1-\beta \lambda_{i}(W)}}= \frac{1}{\sqrt{\dett(I - \beta T_W)}}
\end{equation}
Putting \eqref{eq:psiNlimit}, \eqref{Y_fNlimit} and \eqref{Y_Nlimit} together, we thus have:
\begin{equation*}
   \psi_N(t) =   e^{\frac{t}{N}\sum_{i=1}^{N}f(\frac{i}{N})^2} \frac{\E_\mu Y_{f,N}}{\E_\mu Y_{N}} (1+o_P(1)) = e^{t\int f^2}\sqrt{\frac{\dett(I-\beta T_W)}{\dett(I-\beta T_W-2t\opb{f}{f})}} + o_P(1).
\end{equation*}
Therefore, by Lemma \ref{lem:detfact}, we have:
\begin{equation}\label{psineqconv6}
    \psi_N(t) \xrightarrow{P} \frac{1}{\sqrt{(1-2t\ipb{f}{(I-\beta T_W)^{-1}f})}}.
\end{equation}
The right hand side of \eqref{psineqconv6} can be easily identified as the m.g.f. of 
 $\ipb{f}{(I-\beta T_W)^{-1}f}Z^2$ where $Z$ is a standard normal random variable, and hence, by Lemma \ref{l2ratio762}, we conclude that 
 \begin{equation*}
    \sigma_N(f)^2 \xrightarrow{d} \ipb{f}{(I-\beta T_W)^{-1}f}Z^2\quad\text{and} \quad  \E(\sigma_N(f)^{2k}| A_N) \xrightarrow{P} \E\left(\sqrt{\ipb{f}{(I-\beta T_W)^{-1}f}}Z\right)^{2k}
\end{equation*}
for all $k\in \N$. It is also easy to see that the conditional distribution of $\sigma_N(f)$ given $A_N$ is symmetric about 0 (since conditional on $A_N$, $\bs \stackrel{D}{=} -\bs$). Hence,
$\E[\sigma_N(f)^{2k+1}|A_N] = 0$.
Therefore for all $k \in \N$,
\begin{equation}\label{condmomconv44}
    \E(\sigma_N(f)^{k}|A_N) \xrightarrow{P} \E \left(\sqrt{\ipb{f}{(I-\beta T_W)^{-1}f}}Z\right)^{k}.
\end{equation}

Next, set
\[
    v_f\coloneqq \left\langle f,(I-\beta T_W)^{-1}f\right\rangle,
    \qquad
    \nu_{N,f}\coloneqq \mathcal L(\sigma_N(f)\mid A_N),
    \qquad \text{and}\qquad
    \nu_f\coloneqq \mathcal N(0,v_f).
\]
For \(r\ge1\) and \(\nu,\rho\in\mathcal P(\mathbb R^r)\) (the set of all probability measures on $\R^r$), define
\[
    d_{\mathrm{BL},r}(\nu,\rho)
    \coloneqq
    \sup_{\varphi\in\mathrm{BL}_1(\mathbb R^r)}
    \left|
    \int \varphi\,d\nu-\int \varphi\,d\rho
    \right|,
\]
where
\[
    \mathrm{BL}_1(\mathbb R^r)
    \coloneqq
    \left\{
    \varphi:\mathbb R^r\to\mathbb R:
    \|\varphi\|_\infty\le1,\ 
    |\varphi(x)-\varphi(y)|\le \|x-y\|_2
    \ \forall x,y\in\mathbb R^r
    \right\}.
\]
The metric \(d_{\mathrm{BL},r}\) metrizes weak convergence on
\(\mathcal P(\mathbb R^r)\); see \cite[Theorem 1.12.4]{van1996weak}.

We first prove one-dimensional conditional convergence. Since the
conditional moment convergence \eqref{condmomconv44} holds in probability for every moment
order, it follows from a diagonal argument that every subsequence has a further subsequence, say \(N_m\), along which,
almost surely,
\[
    \int x^q\,d\nu_{N_m,f}(x)
    \rightarrow
    \int x^q\,d\nu_f(x)
    \qquad \text{for all}\quad q\ge1.
\]
The Gaussian law being moment-determinate, \(\nu_{N_m,f}\Rightarrow\nu_f\) almost surely; see
\cite[Theorem 30.2]{billingsley1995probability}. Equivalently,
\(d_{\mathrm{BL},1}(\nu_{N_m,f},\nu_f)\to0\) almost surely. By the subsequence
criterion for convergence in probability, we thus have:
\[
    d_{\mathrm{BL},1}(\nu_{N,f},\nu_f)\xrightarrow{P}0.
\]

Now let
\[
    \bm Y_N\coloneqq
    \big(\sigma_N(f_1),\ldots,\sigma_N(f_k)\big),
    \qquad
    \nu_N\coloneqq \mathcal L(\bm Y_N\mid A_N),
    \qquad
    \nu\coloneqq \mathcal N_k(\boldsymbol 0,\Sigma),
\]
where
\[
    \Sigma_{ij}
    =
    \left\langle f_i,(I-\beta T_W)^{-1}f_j\right\rangle .
\]
For \(\bm a=(a_1,\ldots,a_k)\in\mathbb R^k\), write
\(f_{\bm a} \coloneqq \sum_{j=1}^k a_jf_j\) and
\(\pi_{\bm a}(\bm x)\coloneqq \bm a^\top \bm x\). Then
\(\pi_{\bm a}(\bm Y_N)=\sigma_N(f_{\bm a})\), and the one-dimensional result applied to
\(f_{\bm a}\) gives
\[
    d_{\mathrm{BL},1}
    \left(
    \nu_N\circ\pi_{\bm a}^{-1},
    \nu\circ\pi_{\bm a}^{-1}
    \right)
    \xrightarrow{P}0,
    \qquad \bm a\in\mathbb R^k.
\]

We claim that \(d_{\mathrm{BL},k}(\nu_N,\nu)\to0\) in probability. Towards this, start with an arbitrary subsequence. By diagonal extraction over the
countable set \(\mathbb Q^k\), there is a further subsequence
\(N_m\), such that, almost surely,
\[
    \nu_{N_m}\circ\pi_{\bm a}^{-1}
    \xrightarrow{d}
    \nu\circ\pi_{\bm a}^{-1}
    \qquad \text{for all}\quad \bm a\in\mathbb Q^k.
\]
Fix a realization for which this holds. The convergence for the coordinate
vectors \(\bm e_1,\ldots, \bm e_k\in\mathbb Q^k\) implies tightness of
\(\nu_{N_m}\). Now, if \(Q\) is any subsequential weak limit of $\nu_{N_m}$,
then by the continuous mapping theorem,
\[
    Q\circ\pi_{\bm a}^{-1}
    =
    \nu\circ\pi_{\bm a}^{-1},
    \qquad \bm a\in\mathbb Q^k.
\]
Thus the characteristic functions of \(Q\) and \(\nu\) agree on
\(\mathbb Q^k\), and hence, by continuity, on all of \(\mathbb R^k\).
Therefore \(Q=\nu\). Hence every subsequential weak limit of $\nu_{N_m}$ is \(\nu\), implying that
\(\nu_{N_m}\Rightarrow\nu\). Since the original
subsequence was arbitrary, we have
\[
    d_{\mathrm{BL},k}(\nu_N,\nu)\xrightarrow{P}0.
\]
This is precisely the weak convergence in probability of
\(\mathcal L(\bm Y_N\mid A_N)\) to \(\mathcal N_k(\boldsymbol 0,\Sigma)\), as claimed in the statement of Theorem \ref{thm:isingWeakConv}.

Finally, the unconditional convergence follows from the tower property. More precisely, for
every bounded Lipschitz \(\varphi:\mathbb R^k\to\mathbb R\),
\[
    \mathbb E\varphi(\bm Y_N)
    =
    \mathbb E\left[
    \mathbb E\{\varphi(\bm Y_N)\mid A_N\}
    \right],
\]
and the inner conditional expectation converges in probability to
\(\mathbb E\varphi(\bm Y)\), where \(\bm Y\sim\nu\). By the dominated convergence theorem, we thus have:
\[
    \mathbb E\varphi(\bm Y_N)\rightarrow \mathbb E\varphi(\bm Y),
\]
which is true for all $\varphi \in \mathrm{BL}_1(\R^k)$. Thus \(\bm Y_N\xrightarrow{d}\mathcal N_k(\boldsymbol 0,\Sigma)\), completing the
proof.

\subsection{Proof of Proposition \ref{lem:donsker_fdd}}\label{sec:SNfindimcv8}
For \(1\le r\le k\), let us set
\(
    f_r=\mathbf 1_{[0,t_r]} .
\)
Then \(f_r\) is Riemann integrable and
\(
    S_N(t_r)
    =
    \sigma_N(f_r).
\)
Let \(a_1,\ldots,a_k\in\mathbb R\), and define
\[
    f=\sum_{r=1}^k a_r f_r .
\]
By the conditional form of Theorem \ref{thm:isingWeakConv}, we have
\[
    \sum_{r=1}^k a_r S_N(t_r)
    =
    \sigma_N(f)
    \xrightarrow{d}
    \mathcal N\left(
    0,
    \left\langle
    f,(I-\beta T_W)^{-1}f
    \right\rangle
    \right)
\]
in probability, conditionally on \(A_N\). Moreover,
\begin{equation}\label{covkerverif88}
    \left\langle
    f,(I-\beta T_W)^{-1}f
    \right\rangle
    =
    \sum_{r,s=1}^k
    a_r a_s
    \left\langle
    f_r,(I-\beta T_W)^{-1}f_s
    \right\rangle .
\end{equation}
By Lemma \ref{constructXt}, the covariance kernel of the process \(X\) is
\[
    K(u,v)
    =
    \left\langle
    \mathbf 1_{[0,u]},
    (I-\beta T_W)^{-1}
    \mathbf 1_{[0,v]}
    \right\rangle .
\]
Thus the limiting variance \eqref{covkerverif88} is exactly
\[
    \operatorname{Var}\left(
    \sum_{r=1}^k a_r X(t_r)
    \right).
\]
Proposition \ref{lem:donsker_fdd} now follows from the Cramér--Wold device.

\subsection{Proof of Theorem \ref{thm:donsker}}\label{proof:donsker}

Let \(\widetilde S_N\) denote the linear interpolation of \(S_N\) on the
grid \(\{0,1/N,\ldots,1\}\):
\begin{equation}\label{interpolated66}
    \widetilde S_N(t)
    =
    \frac{1}{\sqrt N}
    \sum_{i=1}^N
    \bigl((Nt-i+1)_+\wedge 1\bigr)\sigma_i,
    \qquad 0\le t\le 1 .
\end{equation}
Define the following (possibly random) elements of $\mathcal{P}(D[0,1])$:
\[
    \mu_N:=\mathcal L(S_N| A_N),\qquad
    \widetilde\mu_N:=\mathcal L(\widetilde S_N| A_N)
    \qquad \text{and}\qquad
    \mu:=\mathcal L(X).
\]
Now, let us take \(X\) to be the continuous modification furnished by
Lemma \ref{holderX}. Since $\widetilde S_N$ and $X$ have continuous paths,
$\widetilde\mu_N$ and $\mu$ may also be viewed as probability
measures on $C[0,1]$.
By Lemma \ref{lem:fourthmomentinterpolated}, there is an event
\(\Omega_0 \in \sigma(\{A_N\}_{N\ge 1})\) with \(\mathbb P(\Omega_0)=1\) such that, for every
\(\omega\in\Omega_0\), there exists \(C_\omega<\infty\) satisfying
\[
    \mathbb E\!\left[
        \left(\widetilde S_N(t)-\widetilde S_N(s)\right)^4
        \,\middle| A_N
    \right] \omega
    \le C_\omega (t-s)^2
\]
for all \(N\ge1\) and \(0\le s<t\le1\). Since
\(\widetilde S_N(0)=0\), Kolmogorov's tightness criterion
\cite[Theorem XIII.1.8]{revuzYor1999}, applied for each fixed
\(\omega\in\Omega_0\), shows that
\(\{\widetilde\mu_N \omega :N\ge1\}\) is tight in \(C[0,1]\) equipped with the supremum norm topology.

Now consider an arbitrary subsequence. By Proposition
\ref{lem:donsker_fdd} and a diagonal extraction over the countable collection of all finite tuples
of rational times (i.e. all $\bm t \in \cup_{r\ge 1} (\mathbb{Q} \cap [0,1])^r$), there exists an event $\Omega_1$ with probability $1$ and a further subsequence
\(\{N_j\}_{j\ge 1}\),
such that for every $\omega \in \Omega_1$, \(r\ge1\) and
\(t_1,\ldots,t_r\in\mathbb Q\cap[0,1]\),
\begin{equation}\label{fndmcv8288}
    \mathcal L\!\left(
        S_{N_j}(t_1),\ldots,S_{N_j}(t_r)
        \,\middle| A_{N_j}
    \right) \omega
    \xrightarrow{d}
    \mathcal L\!\left(X(t_1),\ldots,X(t_r)\right).
\end{equation}
Moreover, it follows from the definition of $\widetilde S_N$ that
\[
    \|S_N-\widetilde S_N\|_\infty\le N^{-1/2},
\]
so the same finite-dimensional convergence \eqref{fndmcv8288} holds with \(S_{N_j}\)
replaced by \(\widetilde S_{N_j}\).

Fix \(\omega\in\Omega_0\cap\Omega_1\). The sequence
\(\{\widetilde\mu_{N_j} \omega\}\) is tight in \(C[0,1]\). If \(Q\)
is any of its subsequential weak limits, continuity of the finite-dimensional coordinate-projection
maps on \(C[0,1]\) shows that \(Q\) and \(\mu\) have the same
finite-dimensional marginals at all rational times. Since these
finite-dimensional coordinate maps generate the Borel \(\sigma\)-field on \(C[0,1]\),
we have \(Q=\mu\). Consequently,
\[
    \widetilde\mu_{N_j}\omega\xrightarrow{d}\mu
    \qquad\text{in }C[0,1].
\]
Thus, every subsequence admits a further subsequence along which this
convergence holds almost surely. By the subsequence criterion in the
metrizable space $\mathcal P(C[0,1])$, we conclude that
\begin{equation}\label{eq:interpolated-C-convergence}
    \mathcal L(\widetilde S_N| A_N)
    \xrightarrow{d}
    \mathcal L(X)
\end{equation}
in probability as probability measures on $C[0,1]$.
Since the canonical inclusion
\[
    C[0,1]\rightarrow D[0,1]
\]
is continuous, the same convergence also holds as probability
measures on $D[0,1]$.

Now, let \(d_D:=d_{\mathrm{Sk}}\wedge1\), and define the associated
bounded-Lipschitz metric on \(\mathcal P(D[0,1])\) by
\[
d_{\mathrm{BL},D}(\nu,\rho)
:=
\sup_{F\in\mathrm{BL}_1(D[0,1])}
\left|
\int_{D[0,1]} F\,d\nu
-
\int_{D[0,1]} F\,d\rho
\right|,
\]
where
\[
\mathrm{BL}_1(D[0,1])
:=
\left\{
F:D[0,1]\to\mathbb R:
\|F\|_\infty\le1,\ 
|F(x)-F(y)|\le d_D(x,y)
\text{ for all }x,y\in D[0,1]
\right\}.
\]
Note that:
\begin{eqnarray*}
    d_{\mathrm{BL},D}(\mu_N,\widetilde\mu_N) &=& \sup_{F\in\mathrm{BL}_1(D[0,1])} \left|\e[F(S_N)|A_N] - \e[F(\widetilde S_N)|A_N] \right|\\&\le& \sup_{F\in\mathrm{BL}_1(D[0,1])}  \e\left[|F(S_N) -F(\widetilde S_N)|\Big|A_N\right]\\&\le&
\mathbb E\!\left[d_D(S_N,\widetilde S_N)\Big| A_N\right]\\ &\le& \mathbb E\!\left[\|S_N- \widetilde S_N\|_\infty\Big| A_N\right]
    \le N^{-1/2}.
\end{eqnarray*}
Since $d_{\mathrm{BL},D}$ metrizes weak convergence on \(\mathcal P(D[0,1])\) \cite[Theorem~1.12.4]{van1996weak}, it follows that
\[
    d_{\mathrm{BL},D}(\mu_N,\mu)\xrightarrow{P}0,
\]
which proves the asserted conditional weak convergence statement of Theorem \ref{thm:donsker}.

Finally, the unconditional convergence follows from the tower property. More precisely, for
every bounded continuous \(F:D[0,1]\to\mathbb R\),
\[
    \mathbb E F(S_N)
    =
    \mathbb E\left[
    \mathbb E\{F(S_N)\mid A_N\}
    \right],
\]
and the inner conditional expectation converges in probability to
\(\mathbb E F(X)\). Since $$|\mathbb E\{F(S_N)\mid A_N\}| \le \|F\|_\infty,$$ the dominated convergence theorem gives:
\[
    \mathbb E F(S_N)\rightarrow \mathbb E F(X),
\]
which is true for all bounded continuous \(F:D[0,1]\to\mathbb R\). Hence,
\(
    S_N\xrightarrow{d}X
\) in $D[0,1]$,
completing the proof of Theorem \ref{thm:donsker}.

\subsection{Proof of Theorem \ref{thm:process_class}}\label{proff:thmproclass222}
To begin with, let us equip $C(\mathcal C)$ with the supremum norm
\[
\|z\|_{\infty}
\coloneq
\sup_{f\in\mathcal C}|z(f)|.
\]
Since $(\mathcal C,d_{AC})$ is a compact metric space,
it is both complete and separable. It follows from Lemmas 3.97 and 3.99 in \cite{aliprantis2006infinite} that $C(\mathcal C)$, equipped with the supremum norm, is Polish.
Moreover, 
\[
M_{\mathcal C}
\coloneq
\sup \left\{
|f(1)|+\|f'\|_{L^1}: f\in\mathcal C\right\}
=
\sup_{f\in\mathcal C}d_{AC}(f,0)
<\infty
\]
satisfies that for every $x,y\in C[0,1]$,
\[
\|\Psi(x)-\Psi(y)\|_{\infty}
\le
M_{\mathcal C}\|x-y\|_\infty,
\]
and hence $\Psi:C[0,1]\to C(\mathcal C)$ is continuous.

Let $\widetilde S_N$ be the linear interpolation of $S_N$ defined in
\eqref{interpolated66}, and set
\(
\widetilde{\mathbb G}_N
\coloneq
\Psi(\widetilde S_N).
\)
By \eqref{eq:interpolated-C-convergence},
$\mathcal L(\widetilde S_N|
A_N)$ converges weakly in probability
to $\mathcal L(X)$ as probability measures on $C[0,1]$. Therefore, by
the continuous mapping theorem
 and the subsequence
criterion for convergence in probability,
$\widetilde{\nu}_N := \mathcal L(\widetilde{\mathbb G}_N| A_N)$ converges to
$\nu := \mathcal L(\mathbb G)$
weakly in probability on $C(\mathcal C)$.

We now compare $\mathbb G_N$ and
$\widetilde{\mathbb G}_N$. For $f\in\mathcal C$, discrete summation by
parts and the absolute continuity of $f$ give:
\begin{align*}
\sigma_N(f)
&=
\sum_{i=1}^N
f\left(\frac{i}{N}\right)
\left\{
S_N\left(\frac{i}{N}\right)
-
S_N\left(\frac{i-1}{N}\right)
\right\} \\
&=
f(1)S_N(1)
-
\sum_{i=1}^{N-1}
S_N\left(\frac{i}{N}\right)
\left\{
f\left(\frac{i+1}{N}\right)
-
f\left(\frac{i}{N}\right)
\right\} \\
&=
f(1)S_N(1)-\int_0^1 S_N(t)f'(t)\,\mathrm dt.
\end{align*}
Since $S_N(1)=\widetilde S_N(1)$, it follows that
\begin{eqnarray*}
\|\mathbb G_N-\widetilde{\mathbb G}_N\|_{\infty}
&=&
\sup_{f\in\mathcal C}
|\mathbb G_N(f)-\widetilde{\mathbb G}_N(f)|\\
&=&
\sup_{f\in\mathcal C}
|\sigma_N(f)-\Psi(\widetilde S_N)(f)|\\
&=&
\sup_{f\in\mathcal C}
\left|
\int_0^1
\bigl(\widetilde S_N(t)-S_N(t)\bigr)f'(t)\,\mathrm dt
\right|\\
&\le&
\left(\sup_{f\in\mathcal C}\|f'\|_{L^1}\right)
\|S_N-\widetilde S_N\|_\infty\\
&\le&
M_{\mathcal C}\|S_N-\widetilde S_N\|_\infty
\le
\frac{M_{\mathcal C}}{\sqrt N}.
\end{eqnarray*}

Next, define
\(
\nu_N\coloneq\mathcal L(\mathbb G_N| A_N)\), and for $z, w\in C(\mathcal{C})$, let
\(
d_{\mathcal C}^{\ast}(z,w)
\coloneq
\|z-w\|_{\infty}\wedge 1\).
Define the bounded-Lipschitz metric on
$\mathcal P(C(\mathcal C))$ by
\[
d_{\mathrm{BL}}(\nu,\rho)
\coloneq
\sup_{F\in\mathrm{BL}_1(C(\mathcal C))}
\left|
\int_{C(\mathcal C)}F\,\mathrm d\nu
-
\int_{C(\mathcal C)}F\,\mathrm d\rho
\right|,
\]
where
\[
\mathrm{BL}_1(C(\mathcal C))
\coloneq
\left\{
F:C(\mathcal C)\to\mathbb R:
\|F\|_\infty\le 1,\ 
|F(z)-F(w)|
\le
d_{\mathcal C}^{\ast}(z,w)
\ \text{for all }z,w\in C(\mathcal C)
\right\}.
\]
This metric
metrizes weak convergence on $\mathcal P(C(\mathcal C))$
\cite[Theorem~1.12.4]{van1996weak}. Then
\[
d_{\mathrm{BL}}(\nu_N,\widetilde\nu_N)
\le
\mathbb E\left[
\|\mathbb G_N-\widetilde{\mathbb G}_N\|_{\infty}\wedge1
\,\middle| A_N
\right]
\le
\frac{M_{\mathcal C}}{\sqrt N}.
\]
Since $\widetilde\nu_N\rightarrow\nu$ weakly in probability, the
triangle inequality yields
\[
d_{\mathrm{BL}}(\nu_N,\nu)\xrightarrow{P}0,
\]
which proves the asserted conditional weak convergence.

Finally, the unconditional convergence follows from the tower
property. Indeed, for every bounded continuous
$F:C(\mathcal C)\to\mathbb R$,
\[
\mathbb E F(\mathbb G_N)
=
\mathbb E\left[
\int F\,\mathrm d\nu_N
\right].
\]
The inner integral converges in probability to
$\int F\,\mathrm d\nu=\mathbb E F(\mathbb G)$ and is bounded in
absolute value by $\|F\|_\infty$. The dominated convergence theorem therefore gives
\[
\mathbb E F(\mathbb G_N)
\rightarrow
\mathbb E F(\mathbb G).
\]
Thus, $\mathbb G_N\xrightarrow{d} \mathbb G$ in $C(\mathcal C)$, completing
the proof of Theorem \ref{thm:process_class}.

\subsection{Proof of Theorem~\ref{thm:sob_conv}}
\label{sec:proofsobolevmain34}
Let us begin with some notations. For notational convenience, we will abbreviate the space $H^{-s}(0,1)$ by $E^s$ (for $s>\frac12$). We equip \(E^s\) with
the bounded metric
\[
\rho_s(u,v):=1\wedge\|u-v\|_{H^{-s}},
\]
which induces the \(H^{-s}\)-norm topology. For
\(\mu,\nu\in\mathcal P(E^s)\), let
\[
d_{\mathrm{BL},s}(\mu,\nu)
:=
\sup_F
\left|
\int_{E^s}F\,d\mu-\int_{E^s}F\,d\nu
\right|,
\]
where the supremum is over all \(F:E^s\to\mathbb R\) satisfying
\(\|F\|_\infty\leq1\) and
\(|F(u)-F(v)|\leq\rho_s(u,v)\). This metric metrizes weak convergence on
\(\mathcal P(E^s)\); see
\cite[Theorem~1.12.4]{van1996weak}.
Also, denote
\[
\mathbb E_N[\cdot]:=\mathbb E[\cdot| A_N],
\qquad
\mu_N:=\mathcal L(\eta_N| A_N),
\qquad \text{and}\qquad
\mu:=\mathcal L_{\mathbb P^*}(\eta).
\]
Choose \(\varepsilon>0\) sufficiently small that
\(q:=\beta(\|W\|_{\mathrm{op}}+\varepsilon)<1\), and define
\[
\mathcal G_N
:=
\left\{
\left\|
\frac{W_{A_N}}{\theta_N}-W
\right\|_{\mathrm{op}}
<\varepsilon
\right\},
\]
where \(W_{A_N}\) is the empirical step-kernel associated with
\(A_N\), as defined in \eqref{eq:Waneq6}. By
Lemma~\ref{lem:empirical-graphon-operator-convergence}, we have
\(\mathbb P(\mathcal G_N)\to1\).
It follows from \eqref{reqpr66} that on \(\mathcal G_N\),
\[
\mathbb E_N|\langle \bm u, \bm \sigma\rangle|^2
\leq
\frac1{1-q}\|\bm u\|_2^2,
\qquad \bm u\in[0,\infty)^N.
\]
Since the two-point correlations are nonnegative, for every
\(\bm c\in\mathbb R^N\), we have:
\[
\begin{aligned}
\mathbb E_N|\langle \bm c, \bm \sigma\rangle|^2
&=
\sum_{i,j=1}^Nc_ic_j\mathbb E_N[\sigma_i\sigma_j] \leq
\sum_{i,j=1}^N|c_i||c_j|
\mathbb E_N[\sigma_i\sigma_j]\\
&=
\mathbb E_N|\langle |\bm c|,\bm \sigma\rangle|^2
\leq
\frac1{1-q}\|\bm c\|_2^2.
\end{aligned}
\]

Now, for every \(k\geq0\), we have
\[
\langle\eta_N,e_k\rangle
=
\frac1{\sqrt N}\sum_{i=1}^N\sigma_i e_k\left(\frac{i}{N}\right) = \sigma_N(e_k).
\]
Applying the preceding bound with
\(c_i=N^{-1/2}e_k(i/N)\), and using
\(N^{-1}\sum_{i=1}^Ne_k(i/N)^2\leq2\), we have on \(\mathcal G_N\),
\[
\mathbb E_N|\langle\eta_N,e_k\rangle|^2
\leq
\frac{2}{1-q}
=:C_q,
\qquad k\geq0.
\]

Next, for each \(K\geq0\) and $u \in E^s$, define
\[
\Pi_Ku:=\sum_{k=0}^K\langle u,e_k\rangle e_k,
\qquad
b_K(s):=\sum_{k>K}\kappa_k^{-s}.
\]
Since \(s>1/2\), \(b_K(s)\to0\) as \(K\to\infty\). Hence, on \(\mathcal G_N\), we have
\[
\mathbb E_N
\|(I-\Pi_K)\eta_N\|_{H^{-s}}^2
=
\sum_{k>K}\kappa_k^{-s}
\mathbb E_N|\langle\eta_N,e_k\rangle|^2
\leq
C_qb_K(s).
\]
Similarly,
\[
\begin{aligned}
\mathbb E^*
\|(I-\Pi_K)\eta\|_{H^{-s}}^2
&=
\sum_{k>K}\kappa_k^{-s}
\mathbb E^*[\chi(e_k)^2]\\
&=
\sum_{k>K}\kappa_k^{-s}
\langle e_k,R_\beta e_k\rangle
\leq
\|R_\beta\|_{\mathrm{op}}b_K(s).
\end{aligned}
\]

Now, set
\[
\mu_{N,K}:=\mathcal L(\Pi_K\eta_N| A_N),
\qquad
\mu_K:=\mathcal L_{\mathbb P^*}(\Pi_K\eta).
\]
For each fixed \(K\), Theorem~\ref{thm:isingWeakConv}, applied to
\(e_0,\ldots,e_K\), gives
\[
\mathcal L\left(
\bigl(\langle\eta_N,e_k\rangle\bigr)_{k=0}^K
\Big|A_N
\right)
\xrightarrow{d}
\mathcal L_{\mathbb P^*}\left(
\bigl(\chi(e_k)\bigr)_{k=0}^K
\right)
\]
in probability. Since
\((a_0,\ldots,a_K)\mapsto\sum_{k=0}^Ka_ke_k\) is continuous from
\(\mathbb R^{K+1}\) into \(E^s\), it follows by the continuous mapping theorem, that
\begin{equation}\label{blsbound98}
d_{\mathrm{BL},s}(\mu_{N,K},\mu_K)
\xrightarrow{P}0.
\end{equation} 

Now, on \(\mathcal G_N\), we have:
\[
\begin{aligned}
d_{\mathrm{BL},s}(\mu_N,\mu_{N,K})
&\leq
\mathbb E_N
\rho_s(\eta_N,\Pi_K\eta_N)\\
&\leq
\left(
\mathbb E_N
\|(I-\Pi_K)\eta_N\|_{H^{-s}}^2
\right)^{1/2}
\leq
\sqrt{C_qb_K(s)}.
\end{aligned}
\]
Likewise,
\[
d_{\mathrm{BL},s}(\mu_K,\mu)
\leq
\sqrt{\|R_\beta\|_{\mathrm{op}}b_K(s)}.
\]
Thus, on \(\mathcal G_N\),
\begin{equation}\label{blsbound9882}
d_{\mathrm{BL},s}(\mu_N,\mu)
\leq
d_{\mathrm{BL},s}(\mu_{N,K},\mu_K)
+
\left(
\sqrt{C_q}+\sqrt{\|R_\beta\|_{\mathrm{op}}}
\right)b_K(s)^{1/2}.
\end{equation}

Fix \(\delta>0\), and choose \(K\) sufficiently large that the second
term on the RHS of \eqref{blsbound9882} is at most \(\delta/2\). It then follows from \eqref{blsbound98} and \eqref{blsbound9882} that as $N \rightarrow \infty$,
\[
\mathbb P\left(
d_{\mathrm{BL},s}(\mu_N,\mu)>\delta
\right)
\leq
\mathbb P(\mathcal G_N^c)
+
\mathbb P\left(
d_{\mathrm{BL},s}(\mu_{N,K},\mu_K)>\delta/2
\right)
\rightarrow0.
\]
This proves the asserted weak convergence in probability of the
conditional laws.

Finally, let \(\overline\mu_N:=\mathcal L(\eta_N)\) denote the
unconditional law of \(\eta_N\). Since
\(0\leq d_{\mathrm{BL},s}\leq2\), the preceding convergence in
probability implies
\(\mathbb E[d_{\mathrm{BL},s}(\mu_N,\mu)]\to0\). Moreover,
\[
d_{\mathrm{BL},s}(\overline\mu_N,\mu)
\leq
\mathbb E\!\left[
d_{\mathrm{BL},s}(\mu_N,\mu)
\right],
\]
and hence \(d_{\mathrm{BL},s}(\overline\mu_N,\mu)\to0\). Therefore the
unconditional laws of \(\eta_N\) converge weakly to the law of
\(\eta\) on \(H^{-s}(0,1)\). The proof of Theorem \ref{thm:sob_conv} is now complete.

\subsection*{Acknowledgement}S. Bhowal thanks Saraswata Sensarma for initial discussions that motivated this work and Sourav Chatterjee for pointing out several useful references. S. Mukherjee was supported by the AcRF Tier 1 grants A-8001449-00-00 and A-8002932-00-00.

\bibliographystyle{plain}
\bibliography{merged_verified_references1}

\appendix 

\section{Proofs of Theorems \ref{thm:ising-neural-network} and \ref{thm:hajek-graphon-ising}}
In this section, we prove Theorems \ref{thm:ising-neural-network} and \ref{thm:hajek-graphon-ising}. 

\subsection{Proof of Theorem \ref{thm:ising-neural-network}}\label{nntwproof}
Fix \(m\geq1\) and
\(\bm x_1,\ldots,\bm x_m\in\mathbb R^n\), and write
\[
    \bm H_i
    \coloneq
    \bigl(
        a(\langle\bm w_i,\bm x_1\rangle),\ldots,
        a(\langle\bm w_i,\bm x_m\rangle)
    \bigr)^\top,
    \qquad
    \bm\mu\coloneq\mathbb E[\bm H_1],
    \qquad
    \bm\Gamma\coloneq\operatorname{Cov}(\bm H_1).
\]
Thus, we have
\(\mu_b=m_a(\bm x_b)\) and
\(\Gamma_{bc}=C_a(\bm x_b,\bm x_c)\). Next, set
\[
    \bm P_k
    \coloneq
    \bigl(
        \mathcal P_k^a(\bm x_1),\ldots,
        \mathcal P_k^a(\bm x_m)
    \bigr)^\top,
    \qquad
    M_k\coloneq\frac{1}{\sqrt{k}}\sum_{i=1}^k s_i .
\]

Fix \(\bm u\in\mathbb R^m\), and let
\(Y_i:=\bm u^\top\bm H_i\),
\(\mu_{\bm u}:=\bm u^\top\bm\mu\), and
\(v_{\bm u}:=\bm u^\top\bm\Gamma\bm u\).
If
\(\varphi_{\bm u}(t):=\mathbb E[e^{itY_1}]\), then it follows from the finite
second-moment assumption \eqref{momentasmp82} and Theorem 3.3.8 in \cite{Durrett2010}, that
\[
    \log\varphi_{\bm u}(t)
    =
    i\mu_{\bm u}t-\frac{v_{\bm u}}{2}t^2+o(t^2),
    \qquad t\to0.
\]
Consequently, uniformly over
\(\bm s\in\{-1,1\}^k\), we have
\[
    \sum_{i=1}^k
    \log\varphi_{\bm u}\left(\frac{s_i}{\sqrt{k}}\right)
    =
    i\mu_{\bm u}M_k-\frac{v_{\bm u}}{2}+o(1).
\]
Using the independence of the hidden-layer weights from
\((A_k,\bm s)\), we obtain
\begin{align*}
    \mathbb E\left[
        e^{i\bm u^\top\bm P_k}\Big| A_k
    \right]
    &=
    \mathbb E\left[
        \prod_{i=1}^k
        \varphi_{\bm u}\left(\frac{s_i}{\sqrt{k}}\right)
        \Big|A_k
    \right] \\
    &=
    e^{-v_{\bm u}/2}
    \mathbb E\left[
        e^{i\mu_{\bm u}M_k}\Big| A_k
    \right](1+o(1)),
\end{align*}
where the final \(o(1)\) is deterministic.

Now, by Theorem~\ref{thm:isingWeakConv}, applied to the constant test
function \(\mathbf 1\), we have:
\[
    \mathbb E\left[e^{itM_k}\Big| A_k\right]
    \xrightarrow{P}
    \exp\left(-\frac{q_\beta t^2}{2}\right),
    \qquad t\in\mathbb R.
\]
Therefore,
\begin{equation}\label{condcharconv7888}
    \mathbb E\left[
        e^{i\bm u^\top\bm P_k}\Big| A_k
    \right]
    \xrightarrow{P}
    \exp\left\{
        -\frac12
        \left(
            v_{\bm u}
            +
            q_\beta\mu_{\bm u}^2
        \right)
    \right\}.
\end{equation}
Since
\[
    v_{\bm u}+q_\beta\mu_{\bm u}^2
    =
    \bm u^\top
    \bigl(
        \bm\Gamma+q_\beta\bm\mu\bm\mu^\top
    \bigr)
    \bm u,
\]
the limit is the characteristic function of
\(\mathcal N_m(\boldsymbol 0,
\bm\Gamma+q_\beta\bm\mu\bm\mu^\top)\).
 L\'evy's continuity theorem together with a subsequential argument connecting almost sure convergence and convergence in probability, now give
the asserted conditional weak convergence. The unconditional
convergence follows by taking expectations on both sides of \eqref{condcharconv7888}, and an application of L\'evy's continuity theorem.

\subsection{Proof of Theorem \ref{thm:hajek-graphon-ising}}\label{cstwproof}
To begin with, write
\[
    m_N\coloneq\frac1N\sum_{i=1}^N\sigma_i,
    \qquad
    \Delta_{i,N}
    \coloneq
    \frac{1}{N-1}\sum_{j\ne i}D_j-\frac12
    =
    \frac{Nm_N-\sigma_i}{2(N-1)}.
\]
By Theorem~\ref{thm:isingWeakConv}, applied to the constant function
\(\mathbf 1\), we have:
\begin{equation}\label{on2p1}
    \sqrt N\,m_N=\sigma_N(\mathbf 1)=O_P(1).
\end{equation}
Hence,
\[
    \max_{1\leq i\leq N}|\Delta_{i,N}|
    =
    O_P(N^{-1/2}),
\]
and \(N^{-1}\sum_iD_i=(1+m_N)/2\to1/2\) in probability. This proves
the assertion concerning the denominators, and henceforth, we work on the events that both $\sum_i D_i$ and $\sum_i (1-D_i)$ are non-zero. Next, for brevity, set
\[
    a_i=F_i\left(1,\frac12\right),\quad
    b_i=F_i\left(0,\frac12\right),\quad
    c_i=\partial_2F_i\left(1,\frac12\right),\quad
    d_i=\partial_2F_i\left(0,\frac12\right),
\]
and denote by $\bar{a}_N, \bar{b}_N, \bar{c}_N$ and $\bar{d}_N$ the empirical averages of these quantities. 
Note that:

\begin{eqnarray*}
\frac{\sum_{i=1}^N D_i a_i}{\sum_{i=1}^N D_i}
&=&
\frac{
\bar a_N+\frac1N\sum_{i=1}^N\sigma_i a_i
}{
1+m_N
}\\
&=&
\bar a_N
+
\frac{
\frac1N\sum_{i=1}^N\sigma_i(a_i-\bar a_N)
}{
1+m_N
}\\
&=&
\bar a_N
+
\frac1N\sum_{i=1}^N\sigma_i(a_i-\bar a_N)
+
O_P(N^{-1}).
\end{eqnarray*}
Here the last step uses \(m_N=O_P(N^{-1/2})\) (by
\eqref{on2p1}), together with
\[
    \frac1N\sum_{i=1}^N\sigma_i(a_i-\bar a_N)
    =
    O_P(N^{-1/2}),
\]
which follows by applying the Lindeberg central limit theorem
conditionally on $\bm\sigma$ to the quantity 
\(N^{-1/2}\sum_i\sigma_i(a_i-\mathbb E[a_1])\), using the independence
of $\{a_i\}_{i\geq1}$ from $\bm\sigma$, together with the usual
i.i.d. central limit theorem for $\bar a_N$.
Similarly,
\begin{eqnarray*}
\frac{\sum_{i=1}^N(1-D_i)b_i}
     {\sum_{i=1}^N(1-D_i)}
&=&
\frac{
\bar b_N-\frac1N\sum_{i=1}^N\sigma_i b_i
}{
1-m_N
}\\
&=&
\bar b_N
-
\frac{
\frac1N\sum_{i=1}^N\sigma_i(b_i-\bar b_N)
}{
1-m_N
}\\
&=&
\bar b_N
-
\frac1N\sum_{i=1}^N\sigma_i(b_i-\bar b_N)
+
O_P(N^{-1}).
\end{eqnarray*}
Here the last step again uses \(m_N=O_P(N^{-1/2})\) and
\[
    \frac1N\sum_{i=1}^N\sigma_i(b_i-\bar b_N)
    =
    O_P(N^{-1/2}).
\]
Next, we have:

\begin{eqnarray*}
\frac{\sum_{i=1}^N D_i c_i\Delta_{i,N}}
     {\sum_{i=1}^N D_i}
&=&
\frac{
Nm_N\sum_{i=1}^N D_i c_i
-
\sum_{i=1}^N D_i c_i\sigma_i
}{
2(N-1)\sum_{i=1}^N D_i
}\\
&=&
\frac{Nm_N-1}{2(N-1)}
\cdot\frac{\sum_{i=1}^N D_i c_i}
     {\sum_{i=1}^N D_i}
\\
&=&
\left\{
\frac{m_N}{2}+O_P(N^{-1})
\right\}
\left\{
\bar c_N+O_P(N^{-1/2})
\right\}\\
&=&
\frac{m_N}{2}\bar c_N+O_P(N^{-1}).
\end{eqnarray*}
Here the second equality uses $D_i\sigma_i=D_i$, while the last two
steps use $m_N=O_P(N^{-1/2})$ and the treated-group ratio expansion
\[
    \frac{\sum_{i=1}^N D_i c_i}
         {\sum_{i=1}^N D_i}
    =
    \bar c_N+O_P(N^{-1/2}).
\]
Similarly,
\begin{eqnarray*}
\frac{\sum_{i=1}^N(1-D_i)d_i\Delta_{i,N}}
     {\sum_{i=1}^N(1-D_i)}
&=&
\frac{
Nm_N\sum_{i=1}^N(1-D_i)d_i
-
\sum_{i=1}^N(1-D_i)d_i\sigma_i
}{
2(N-1)\sum_{i=1}^N(1-D_i)
}\\
&=&
\frac{Nm_N+1}{2(N-1)} \cdot
\frac{\sum_{i=1}^N(1-D_i)d_i}
     {\sum_{i=1}^N(1-D_i)}
\\
&=&
\left\{
\frac{m_N}{2}+O_P(N^{-1})
\right\}
\left\{
\bar d_N+O_P(N^{-1/2})
\right\}\\
&=&
\frac{m_N}{2}\bar d_N+O_P(N^{-1}).
\end{eqnarray*}
Here the second equality uses
$(1-D_i)\sigma_i=-(1-D_i)$, while the last two steps use
$m_N=O_P(N^{-1/2})$ and the control-group ratio expansion
\[
    \frac{\sum_{i=1}^N(1-D_i)d_i}
         {\sum_{i=1}^N(1-D_i)}
    =
    \bar d_N+O_P(N^{-1/2}).
\]

Let \(r_{i,N}^{(d)}\) denote the remainder in the second-order
Taylor expansion of $F_i\left(d,\frac12+\Delta_{i,N}\right)$, i.e.
\[
    F_i\left(d,\frac12+\Delta_{i,N}\right)
    =
    F_i\left(d,\frac12\right)
    +
    \partial_2F_i\left(d,\frac12\right)\Delta_{i,N}
    +
    r_{i,N}^{(d)}.
\]
The uniform second-derivative bound gives
\(
|r_{i,N}^{(d)}|\lesssim\Delta_{i,N}^2
\), and hence, since
\(\max_i|\Delta_{i,N}|=O_P(N^{-1/2})\), the two Taylor-remainder
terms for $d \in \{0,1\}$ contribute a factor of \(O_P(N^{-1})\). Therefore,
combining the preceding expansions, we have:
\begin{eqnarray}
\widehat\tau_N
&=& \frac{\sum_{i=1}^N D_i(a_i + c_i \Delta_{i,N} + r_{i,N}^{(1)})}{{\sum_{i=1}^N D_i}
} - \frac{\sum_{i=1}^N (1-D_i)(b_i + d_i \Delta_{i,N} + r_{i,N}^{(0)})}{{\sum_{i=1}^N (1-D_i)}
}
\nonumber\\&=&\bar a_N-\bar b_N
+
\frac1N\sum_{i=1}^N
\sigma_i
\bigl\{
(a_i-\bar a_N)+(b_i-\bar b_N)
\bigr\}
+
\frac{m_N}{2}(\bar c_N-\bar d_N)
+
O_P(N^{-1})
\nonumber\\
&=&
\bar a_N-\bar b_N
+
\frac1N\sum_{i=1}^N
\sigma_i
\left\{
(a_i+b_i)-(\bar a_N+\bar b_N)
+
\frac{\bar c_N-\bar d_N}{2}
\right\}
+
O_P(N^{-1})
\nonumber\\
&=&
\bar a_N-\bar b_N
+
\frac1N\sum_{i=1}^N
\sigma_i
\left\{
R_i-\bar R_N
+
\frac{\bar c_N-\bar d_N}{2}
\right\}
+
O_P(N^{-1})
\label{eq:hajek-interference-linearization-1}
\end{eqnarray}
where $\bar R_N := \coloneq
\frac1N\sum_{i=1}^N R_i
=
\bar a_N+\bar b_N$.
Next, since $\bs \stackrel{d}{=} -\bs$, we have
\(\mathbb E(\Delta_{i,N}|A_N)=0\). In addition, the
conditional second-moment bound \eqref{reqpr66} gives
\[
    \max_{1\leq i\leq N}
    \mathbb E(\Delta_{i,N}^2| A_N)
    =
    O_P(N^{-1}).
\]
A Taylor expansion in \eqref{eq:conditional-direct-effect} gives
\begin{eqnarray*}
\tau_N
&=&
\bar a_N-\bar b_N
+
\frac1N\sum_{i=1}^N
(c_i-d_i)
\mathbb E(
\Delta_{i,N}
|
A_N,F_1,\ldots,F_N
)\\
&&\quad+
\frac1N\sum_{i=1}^N
\mathbb E\left(
r_{i,N}^{(1)}-r_{i,N}^{(0)}
\,\middle|
A_N,F_1,\ldots,F_N
\right).
\end{eqnarray*}
By symmetry of the Ising measure and independence of the $F_i$'s from the assignment mechanism, we have:
\[
    \mathbb E(\Delta_{i,N}| A_N,F_1,\ldots,F_N)
    =
    \mathbb E(\Delta_{i,N}| A_N)
    =
    0.
\]
Consequently,
\[
    \tau_N
    =
    \bar a_N-\bar b_N+O_P(N^{-1}).
\]
Combining this with
\eqref{eq:hajek-interference-linearization-1}, the law of large
numbers, and \eqref{on2p1}, gives:
\begin{eqnarray}\label{tnhattnlim98}
\sqrt N\bigl(\widehat\tau_N-\tau_N\bigr)
&=&
\frac1{\sqrt N}\sum_{i=1}^N
\sigma_i
\left\{
R_i-\bar R_N+\frac{\bar c_N-\bar d_N}{2}
\right\}
+
o_P(1)
\notag\\
&=&
\frac1{\sqrt N}\sum_{i=1}^N L_i\sigma_i
+
\left\{
\mathbb E[R_1]-\bar R_N
+
\frac{\bar c_N-\bar d_N}{2}
-
Q
\right\}
\frac1{\sqrt N}\sum_{i=1}^N\sigma_i
+
o_P(1)
\notag\\
&=&
\frac1{\sqrt N}\sum_{i=1}^N L_i\sigma_i
+
o_P(1).
\end{eqnarray}
Let
\(
    \mathcal R_N
    \coloneq
    \sqrt N\bigl(\widehat\tau_N-\tau_N\bigr)
    -
    \frac1{\sqrt N}\sum_{i=1}^N L_i\sigma_i.
\)
By \eqref{tnhattnlim98}, $\mathcal R_N=o_P(1)$, and hence,
\(
    2\wedge|\mathcal R_N|
    \xrightarrow{P}0.
\)
Since $0\leq 2\wedge|\mathcal R_N|\leq2$, the dominated convergence theorem implies that
\[
    \E\left[2\wedge|\mathcal R_N|\right]
    \longrightarrow0.
\]
Therefore, by the tower property and Markov's inequality,
\[
    \E\left[
        2\wedge|\mathcal R_N|
        \,\middle|\,A_N
    \right]
    \xrightarrow{P}0.
\]
Consequently, under the natural coupling, we have:
\begin{align}
&d_{\mathrm{BL},1}\left(
    \mathcal L\left(
        \sqrt N(\widehat\tau_N-\tau_N)\mid A_N
    \right),
    \mathcal L\left(
        \frac1{\sqrt N}\sum_{i=1}^N L_i\sigma_i
        \,\middle|\,A_N
    \right)
\right)
\nonumber\\
&\qquad\leq
\E\left[
    2\wedge|\mathcal R_N|
    \,\middle|\,A_N
\right]
\xrightarrow{P}0.
\label{eq:hajek-conditional-slutsky}
\end{align}

Next, let
\(\varphi_{L}(t)=\mathbb E[e^{itL_1}]\). Since \(L_1\) has a finite
second moment, Theorem 3.3.8 in \cite{Durrett2010} gives:
\[
    \log\varphi_{L}(t)
    =
    i\kappa_1t
    -
    \frac{\kappa_2-\kappa_1^2}{2}t^2
    +
    o(t^2),
    \qquad t\to0.
\]
Using the independence of the \(L_i\)'s from \((A_N,\bs)\), we thus have:
\begin{eqnarray*}
    \mathbb E \left(
        \exp \left\{
            \frac{it}{\sqrt N}
            \sum_{i=1}^N L_i\sigma_i
        \right\}
        \Bigg|\,A_N\right)
     &=& \e \left[ \mathbb E
        \left(\exp\left\{
            \frac{it}{\sqrt N}
            \sum_{i=1}^N L_i\sigma_i
        \right\}
        \,\middle|\,A_N,\bs\right)
        \Bigg|\,A_N
    \right]\\
    &=&
    \e\left[
        \prod_{i=1}^N
        \varphi_L\left(
            \frac{t\sigma_i}{\sqrt N}
        \right)
        \Bigg|\,A_N
    \right]\\
    &=&
    \e\left[
        \exp\left\{
            \sum_{i=1}^N
            \log\varphi_L\left(
                \frac{t\sigma_i}{\sqrt N}
            \right)
        \right\}
        \Bigg|\,A_N
    \right]\\
    &=&
    \exp\left\{
        -\frac{t^2}{2}
        (\kappa_2-\kappa_1^2)
        +o(1)
    \right\}
    \mathbb E\left[
        e^{it\kappa_1\sigma_N(\mathbf 1)}
        \,\middle|\,A_N
    \right].
\end{eqnarray*}
Theorem~\ref{thm:isingWeakConv} now implies that the last expression
converges in probability to
\[
    \exp\left\{
        -\frac{t^2}{2}
        \left(
            \kappa_2-\kappa_1^2
            +
            \kappa_1^2q_\beta
        \right)
    \right\}.
\]
L\'evy's continuity theorem, together with a subsequential argument
connecting almost sure convergence and convergence in probability,
shows that
\(
    \mathcal L\left(
        \frac1{\sqrt N}\sum_{i=1}^N L_i\sigma_i
        \,\middle|\,A_N
    \right)
\)
converges weakly in probability to
\(
    \mathcal N\left(
        0,
        \kappa_2-\kappa_1^2+\kappa_1^2q_\beta
    \right)
    =
    \mathcal N\left(
        0,
        \kappa_2+\kappa_1^2(q_\beta-1)
    \right).
\)
Combining this with
\eqref{eq:hajek-conditional-slutsky} and the triangle inequality for
$d_{\mathrm{BL},1}$ gives the asserted conditional weak convergence
of
\(
    \sqrt N\bigl(\widehat\tau_N-\tau_N\bigr).
\)
The unconditional conclusion follows by taking expectations.

\section{Some Technical Results}
In this section, we state and prove some technical lemmas which are crucial in proving the main results in this paper.

\begin{lemma}\label{circlem}
   Let $f:[0,1]\to\mathbb R$ be Riemann integrable, and suppose that
\(
    \beta\|T_W\|_{\mathrm{op}}
    +
    2|t|\|f\|_2^2
    <1.
\)
Then, as $N\to\infty$,
\[
    \E_\mu Y_{f,N}
    \longrightarrow
    \frac{1}{\sqrt{\dett(I-T_K)}},
\]
where
\(
    K=\beta W+2t\opb{f}{f},
\)
and $Y_{f,N}$ is defined in \eqref{circ22}.
    \end{lemma}
\begin{proof}
First, note that by \eqref{eq:absolute-kernel-operator-bound}, 
\(
    \|T_{|K|}\|_{\mathrm{op}}<1.
\)
Moreover, the argument leading to
\cite[Eq.~(A.51)]{mukherjee2026isinginference} gives, for every fixed
$L>0$,
\begin{equation}\label{eq:QN-exponential-bound}
    \limsup_{N\to\infty}
    \E_\mu\exp\{LQ_N(\bs)\}
    <\infty.
\end{equation}
Combining \eqref{eq:QN-exponential-bound} with
Lemma~\ref{generalexpUniInt} and H\"older's inequality, we obtain,
for some $\gamma>0$,
\[
    \sup_{N\geq1}
    \E_\mu Y_{f,N}^{1+\gamma}
    <\infty.
\]
Thus, $\{Y_{f,N}\}_{N\geq1}$ is uniformly integrable under $\mu$.

On the other hand, Lemma~\ref{lem:Qnlp} gives
\[
    Q_N(\bs)\xrightarrow{P}0
\]
under $\mu$. Hence, the one-copy consequence of Lemma~\ref{quadJointAsym},
together with Slutsky's theorem, gives
\[
    Y_{f,N}
    \xrightarrow{d}
    \exp\left\{
        \frac12
        \sum_{j=1}^{\infty}
        \lambda_j(K)(Z_j^2-1)
    \right\}.
\]
Uniform integrability of $\{Y_{f,N}\}_{N\geq1}$ therefore yields
\[
    \lim_{N\to\infty}\E_\mu Y_{f,N}
    =
    \prod_{j=1}^{\infty}
    \frac{\exp\{-\lambda_j(K)/2\}}
         {\sqrt{1-\lambda_j(K)}}
    =
    \frac{1}{\sqrt{\dett(I-T_K)}}
\]
completing the proof of Lemma \ref{circlem}.
\end{proof}

\begin{lemma}\label{var_sum}
 Fix a Riemann-integrable function
$f:[0,1]\to\mathbb R$ and $t\in\mathbb R$ such that
\(
    \beta\|T_W\|_{\mathrm{op}}
    +
    2|t|\|f\|_2^2
    <1.
\) Then, we have:
    $$
        \var \left(\frac{\hat Z_N\left(\beta\right)}{\E \hat Z_N\left(\beta\right)}\right) = O\left(\frac{1}{N\theta_N}\right) \quad\text{and}\quad  \var \left(\frac{\hat Z_N^f\left(\beta,t\right)}{\E \hat Z_N^f\left(\beta,t\right)}\right) = O\left(\frac{1}{N\theta_N}\right)
    $$
    where $\hat{Z}_N(\beta)$ and $\hat{Z}^f_N(\beta,t)$ are as defined in \eqref{def:znhatising} and \eqref{def:znhatfising}, respectively.
\end{lemma}

\begin{proof}
   The first variance bound follows directly from Lemma A.1 in \cite{mukherjee2026isinginference}, so we prove the second bound. Throughout the proof, write
   \[
        W_{ij}\coloneqq W\left(\frac{i}{N},\frac{j}{N}\right),
        \qquad
        f_i\coloneqq f\left(\frac{i}{N}\right).
   \]
   To begin with, it follows from Lemmas \ref{lem:expect2} and \ref{lem:cov2}, that:
\begin{eqnarray*}
        &&\var \hat Z^f_N(\beta,t)\\ &=& \frac{1}{4^N} \sum_{\bs, \bt} e^{\frac{t}{N}\bs^T \bldf \bldf^T\bs+\frac{t}{N}\bt^T \bldf \bldf^T\bt}\cov \left(T(\bs),T(\bt)\right)\\                                                                                           \\
        &=& \frac{1}{4^N} \sum_{\bs, \bt}e^{\frac{t}{N}\bs^T \bldf \bldf^T\bs+\frac{t}{N}\bt^T \bldf \bldf^T\bt} \E T(\bs) \E T(\bt) \left(e^{R_N(\bs,\bt)+O(Q_N(\bs))+O(Q_N(\bt))+O\left(\frac{1}{N\theta_N}Q_N(\bs\bt)\right)}-1\right)
\end{eqnarray*}
where $\bs\bt
\coloneq
(\sigma_1\tau_1,\ldots,\sigma_N\tau_N) \in \{-1,+1\}^N$ and
\begin{equation}\label{RNdef:def8}
        R_N(\bs,\bt)\coloneqq \frac{\beta^2}{\theta_N N^2} \sum_{i>j}W_{ij}\left(1-\theta_N W_{ij}\right)\sigma_i\sigma_j\tau_i\tau_j
    \end{equation}
    and
    \begin{equation}\label{QNdef:def8}
        Q_N(\bs)\coloneqq \frac{1}{N^3 \theta_N^2}\left|\sum_{i>j}W_{ij}\sigma_i \sigma_j\right|.
    \end{equation}
Denoting $\nu$ to be the uniform measure on $\{-1,+1\}^{2N}$, we have the following from Lemma \ref{lem:taylor_expRn}:
\begin{eqnarray}\label{vznf_expansion}
    \var \hat Z^f_N(\beta,t) &=& \E_{(\bs, \bt) \sim \nu}\Bigg[ e^{\frac{t}{N}\bs^T \bldf \bldf^T\bs+\frac{t}{N}\bt^T \bldf \bldf^T\bt}\E T(\bs) \E T(\bt)\Bigg(R_N(\bs,\bt)\nonumber\\ 
    &+& O(Q_N(\bs))+O(Q_N(\bt))+O\left(\frac{1}{N\theta_N}Q_N(\bs\bt)\right)+O_{L^p}\left(\frac{1}{N\theta_N}\right)\Bigg) \Bigg]
\end{eqnarray}
for every $p\ge 1$.
Next, note that $R_N(\bs,\bt) = \tilde{R}_N(\bs,\bt) - R_N'(\bs,\bt)$, where
\begin{eqnarray*}
    \tilde{R}_N(\bs,\bt)&=& \frac{\beta^2}{\theta_N N^2} \sum_{i>j}W_{ij}\sigma_i\sigma_j\tau_i\tau_j,\quad\text{and}\\
    R_N'(\bs,\bt)&=&\frac{\beta^2}{N^2} \sum_{i>j} W_{ij}^2\sigma_i\sigma_j\tau_i\tau_j.
\end{eqnarray*}
Note that $|\tilde{R}_N(\bs,\bt)|=\beta^2 N\theta_N Q_N(\bs\bt)$. Therefore, by Lemma \ref{lem:Qnlp}, $\tilde{R}_N(\bs,\bt)=O_{L^p}(1/N\theta_N)$. It also follows from the proof of Lemma A.8 in \cite{mukherjee2026isinginference}, that the moments of $N\theta_N |R_N'(\bs,\bt)|$ are bounded, and hence, $R_N'(\bs,\bt)=O_{L^p}(1/N\theta_N)$. Consequently, $R_N(\bs,\bt) = O_{L^p}(1/N\theta_N)$, too. Combining this with Lemma \ref{lem:Qnlp}, we have the following from \eqref{vznf_expansion} for every $p\ge 1$:

\begin{equation}\label{1ststsmpl4}
    \var \hat Z^f_N(\beta,t) = \E_{(\bs, \bt) \sim \nu}\Bigg[ e^{\frac{t}{N}\bs^T \bldf \bldf^T\bs+\frac{t}{N}\bt^T \bldf \bldf^T\bt}\E T(\bs) \E T(\bt) O_{L^p}\left(\frac{1}{N\theta_N}\right)\Bigg]
\end{equation}

Next, by Lemma \ref{lem:expect2}, there exists a deterministic sequence $a_N>0$ such that, uniformly in $\bs$,
\begin{equation}\label{ETsigma_an}
    \E T(\bs)
    =
    a_N
    \exp\left(
        \frac{\beta}{N}\sum_{i>j}W_{ij}\sigma_i\sigma_j
        +O(Q_N(\bs))
    \right).
\end{equation}
Also,
\begin{equation*}
    \frac{t}{N}\bs^T\bldf\bldf^T\bs
    =
    \frac{t}{N}\sum_{i=1}^N f_i^2
    +
    \frac{2t}{N}\sum_{i>j}f_i f_j\sigma_i\sigma_j .
\end{equation*}

Moreover, the argument leading to
\cite[Eq.~(A.51)]{mukherjee2026isinginference} gives, for every
fixed $L>0$,
\[
    \limsup_{N\to\infty}
    \E_{(\bs,\bt)\sim\nu}
    \exp\left\{
        L\bigl(Q_N(\bs)+Q_N(\bt)\bigr)
    \right\}
    <\infty.
\]
Consequently, by H\"older's inequality, the factor
\(
    \exp\left\{
        O(Q_N(\bs))+O(Q_N(\bt))
    \right\}
\)
arising from \eqref{ETsigma_an} can be absorbed into the
$O_{L^p(\nu)}(1/(N\theta_N))$ term, after increasing the moment
order if necessary. Thus, by \eqref{1ststsmpl4},
\eqref{ETsigma_an}, and Lemma~\ref{lem:Qnlp}, we have:

\begin{equation}\label{rawvar_bound_with_an}
    \begin{aligned}
    \var \hat Z^f_N(\beta,t)
    \lesssim &
    ~ a_N^2
    \exp\left(\frac{2t}{N}\sum_{i=1}^N f_i^2\right)  \\
    &\times
    \E_{(\bs,\bt)\sim \nu}\left[
    \exp\left(\frac{1}{N}\sum_{i>j}K\left(\frac{i}{N},\frac{j}{N}\right)(\sigma_i\sigma_j+\tau_i\tau_j)\right)
    O_{L^p(\nu)}\left(\frac{1}{N\theta_N}\right)\right],
    \end{aligned}
\end{equation}
for all $p\ge 1$, where recall that
\(
    K(x,y) \coloneqq \beta W(x,y) + 2t f(x) f(y).
\)
Since
\(
    \|T_{|K|}\|_{\mathrm{op}}
    \leq
    \beta\|T_W\|_{\mathrm{op}}
    +
    2|t|\|f\|_2^2
    <1,
\)
choose $\delta>0$ as in Lemma~\ref{generalexpUniInt}.
By H\"older's inequality,
\begin{equation}\label{rawvar_holder_bound}
\begin{aligned}
    \var \hat Z^f_N(\beta,t)
    \lesssim &
    ~ a_N^2
    \exp\left(\frac{2t}{N}\sum_{i=1}^N f_i^2\right) \\
    &\times \E_{(\bs,\bt)\sim \nu}\left[\exp\left(\frac{1+\delta}{N}\sum_{i>j}K\left(\frac{i}{N},\frac{j}{N}\right)(\sigma_i\sigma_j+\tau_i\tau_j)\right)\right]^{\frac{1}{1+\delta}}
    O\left(\frac{1}{N\theta_N}\right).
    \end{aligned}
\end{equation}
By Lemma \ref{generalexpUniInt}, the expectation in \eqref{rawvar_holder_bound} is bounded uniformly in $N$. Hence,
\begin{equation}\label{rawvar_final_bound}
    \var \hat Z^f_N(\beta,t)
    \leq
    a_N^2
    \exp\left(\frac{2t}{N}\sum_{i=1}^N f_i^2\right)
    O\left(\frac{1}{N\theta_N}\right).
\end{equation}

Next, recalling that $\mu$ denotes the uniform measure on $\{-1,+1\}^N$, we have the following by Lemma \ref{lem:expect2}:
\begin{equation}\label{EZf_factor}
    \E \hat Z^f_N(\beta,t)
    =
    a_N
    \exp\left(\frac{t}{N}\sum_{i=1}^N f_i^2\right)
    \E_{\mu}Y_{f,N},
\end{equation}
where
\begin{equation*}
    Y_{f,N}
    =
    \exp\left(
        \frac{1}{N}\sum_{i>j}K\left(\frac{i}{N},\frac{j}{N}\right)\sigma_i\sigma_j
        +O(Q_N(\bs))
    \right).
\end{equation*}
By Lemma \ref{circlem}, we have
\begin{equation}\label{Yf_mean_lower}
    \E_{\mu}Y_{f,N}
    \rightarrow
    \frac{1}{\sqrt{\dett(I-T_K)}}.
\end{equation}
In particular, $\E_{\mu}Y_{f,N}$ is bounded away from $0$ for all sufficiently large $N$. Combining \eqref{rawvar_final_bound} and \eqref{EZf_factor}, we therefore get
\begin{eqnarray*}
    \var\left(
        \frac{\hat Z^f_N(\beta,t)}
        {\E \hat Z^f_N(\beta,t)}
    \right)
    &=&
    \frac{\var \hat Z^f_N(\beta,t)}
    {\left(\E \hat Z^f_N(\beta,t)\right)^2} \\
    &\leq&
    \frac{
    a_N^2
    \exp\left(\frac{2t}{N}\sum_{i=1}^N f_i^2\right)
    O\left(\frac{1}{N\theta_N}\right)
    }{
    a_N^2
    \exp\left(\frac{2t}{N}\sum_{i=1}^N f_i^2\right)
    \left(\E_{\bs\sim \mu}Y_{f,N}\right)^2
    } \\
    &=&
    O\left(\frac{1}{N\theta_N}\right).
\end{eqnarray*}
 This completes the proof of Lemma \ref{var_sum}.
\end{proof}

\begin{lemma}
    \label{generalexpUniInt}
    Suppose that the entries of $\bs \coloneqq (\sigma_1,\ldots,\sigma_N)$ are i.i.d. Rademacher with mean $0$. Let $K:[0,1]^2\to\mathbb R$ be symmetric and Riemann
integrable, and suppose that
\(
    \|T_{|K|}\|_{\mathrm{op}}<1.
\)
    Then there exists $\delta >0$ such that
    \begin{equation}\label{eq:unifint622}
        \limsup_{N \to \infty} \E \left[\exp\left(\frac{1+\delta}{2N}\sum_{i\neq j}K\left(\frac{i}{N},\frac{j}{N}\right)\sigma_i \sigma_j\right)\right] < \infty.
    \end{equation}
    Consequently,
    \[
        \exp \left(\frac{1}{2 N} \sum_{i \neq j} K\left(\frac{i}{N}, \frac{j}{N}\right) \sigma_{i} \sigma_j\right)
    \]
    is uniformly integrable.
\end{lemma}

\begin{proof}
    Let $\mu$ denote the uniform measure on $\{-1,+1\}^{N}$. For $\gamma>0$, define
    \begin{equation*}
        Z^{K}_N(\gamma) \coloneqq  \E_\mu \left[\exp\left(\frac{\gamma}{2N}\sum_{i\neq j}K\left(\frac{i}{N},\frac{j}{N}\right)\sigma_i \sigma_j\right)\right].
    \end{equation*}
    To begin with, note that
\begin{align*}
    Z^K_N(\gamma)
    =& \E_\mu\left[\prod_{i>j} \left(\cosh\left(\frac{\gamma}{N} K\left(\frac{i}{N},\frac{j}{N}\right)\right)+\sigma_i\sigma_j \sinh\left(\frac{\gamma}{N} K\left(\frac{i}{N},\frac{j}{N}\right)\right)\right)\right]                            \\
    =& \prod_{i>j} \cosh\left(\frac{\gamma}{N} K\left(\frac{i}{N},\frac{j}{N}\right)\right)
    \E_\mu\left[\prod_{i>j} \left(1+\sigma_i\sigma_j \tanh\left(\frac{\gamma}{N} K\left(\frac{i}{N},\frac{j}{N}\right)\right)\right)\right].
\end{align*}
For ease of notation, define
\begin{equation}
    \label{eq:defZ_Ntilde}
    \tilde{Z}^K_N(\gamma) \coloneqq \E_\mu\left[\prod_{i>j} \left(1+\sigma_i\sigma_j \tanh\left(\frac{\gamma}{N} K\left(\frac{i}{N},\frac{j}{N}\right)\right)\right)\right].
\end{equation}
Then
\begin{equation}\label{exprlogzn7}
    Z^K_N(\gamma) =  \tilde{Z}^K_N(\gamma)\prod_{i>j}\cosh\left(\frac{\gamma}{N} K\left(\frac{i}{N},\frac{j}{N}\right)\right).
\end{equation}
Now, define
\[
    \tilde{\bK}_N(i,j)=
    \begin{cases}
        \tanh\left(\frac \gamma N K\left(\frac{i}{N}, \frac{j}{N}\right)\right), & i\neq j,\\
        0, & i=j.
    \end{cases}
\]
It follows from \eqref{eq:defZ_Ntilde} that
\[
    \tilde{Z}^K_N(\gamma)
    =
    \sum_{\Gamma \in \cE_N} \prod_{(i,j) \in E(\Gamma)} \tilde{\bK}_N(i,j),
\]
where $\cE_N$ denotes the collection of all spanning subgraphs of the complete graph on $N$ vertices
in which every vertex has even degree. Therefore,
\begin{align*}
    |\tilde{Z}^K_N(\gamma)|
    &\leq \sum_{\Gamma \in \cE_N} \prod_{(i,j) \in E(\Gamma)} |\tilde{\bK}_N(i,j)|\\
    &= \sum_{\Gamma \in \cE_N} \prod_{(i,j) \in E(\Gamma)}
    \tanh\left(\frac \gamma N \left|K\left(\frac{i}{N}, \frac{j}{N}\right)\right|\right)
    =
    \tilde{Z}^{|K|}_N(\gamma).
\end{align*}
Since $\cosh$ is an even function, it follows from \eqref{exprlogzn7} that
\[
    Z_N^K(\gamma)\leq  |\tilde{Z}^K_N(\gamma)|    \prod_{i>j} \cosh\left(\frac{\gamma}{N} \left|K\left(\frac{i}{N},\frac{j}{N}\right)\right|\right) \leq \tilde{Z}^{|K|}_N(\gamma)    \prod_{i>j} \cosh\left(\frac{\gamma}{N} \left|K\left(\frac{i}{N},\frac{j}{N}\right)\right|\right) = Z_N^{|K|}(\gamma).
\]

Next, define the matrix $\bK_N$ by
\[
    \bK_N(i,j)=
    \begin{cases}
        \left|K\left(\frac{i}{N}, \frac{j}{N}\right)\right|, & i\neq j,\\
        0, & i=j.
    \end{cases}
\]
Choose $\delta>0$ small enough so that with $\gamma=1+\delta$, we have
$\gamma\|T_{|K|}\|_{\mathrm{op}}<1.$
Let
\(
    I_{i,N}\coloneq\left(\frac{i-1}{N},\frac{i}{N}\right],
\)
and define the step kernel
\[
    K_N^\circ(x,y)
    \coloneq
    \sum_{\substack{1\leq i,j\leq N\\i\neq j}}
    \left|
        K\left(\frac{i}{N},\frac{j}{N}\right)
    \right|
    \mathbf 1_{I_{i,N}}(x)\mathbf 1_{I_{j,N}}(y).
\]
Then
\(
    \|T_{K_N^\circ}\|_{\mathrm{op}}
    =
    \frac1N\|\bK_N\|_{\mathrm{op}}.
\)
Since $|K|$ is bounded and Riemann integrable, we have
\[
    \|K_N^\circ-|K|\|_{L^2}
    \longrightarrow0,
\]
and hence,
\[
    \left|
        \frac1N\|\bK_N\|_{\mathrm{op}}
        -
        \|T_{|K|}\|_{\mathrm{op}}
    \right|
    \leq
    \|T_{K_N^\circ}-T_{|K|}\|_{\mathrm{op}}
    \leq
    \|K_N^\circ-|K|\|_{L^2}
    \longrightarrow0.
\]
Therefore,
we have the following for some $\varepsilon>0$ and all large enough $N$,
\[
    \frac{\gamma}{N}\|\bK_N\|_{\mathrm{op}}<1-\varepsilon.
\]
Moreover, by the Gaussian domination argument used in the proof of Lemma 7.1 of \cite{bhattacharya2018inference}, and the nonnegativity of the entries of $\bK_N$,
\[
    Z_N^K(\gamma)
    \leq
    Z_N^{|K|}(\gamma)
    \leq
    \E\left[\exp \left(\frac{\gamma}{2N}  \bm Z^{\top} \bK_N \bm Z\right)\right],
\]
where $\bm Z \sim \cN(0, I_N)$.
Let $\lambda_1(\bK_N),\ldots,\lambda_N(\bK_N)$ be the eigenvalues of $\bK_N$. Then
\[
    \E\left[\exp \left(\frac{\gamma}{2N}  \bm Z^{\top} \bK_N \bm Z\right)\right]
    =
    \prod_{i=1}^N
    \left(1-\frac{\gamma\lambda_i(\bK_N)}{N}\right)^{-1/2}.
\]
Since $\tr(\bK_N)=0$, we have $\sum_i \lambda_i(\bK_N)=0$, and hence
\begin{equation}\label{logexpan9}
    \log \E\left[\exp \left(\frac{\gamma}{2N}  \bm Z^{\top} \bK_N \bm Z\right)\right]
    =
    \frac12
    \sum_{i=1}^{N}
    \left[
        -\log \left(1-\frac{\gamma \lambda_i(\bK_N)}{N}\right)
        -
        \frac{\gamma \lambda_i(\bK_N)}{N}
    \right].
\end{equation}
Now, for every $\varepsilon>0$, there exists $C_\varepsilon<\infty$ such that
\[
    -\log(1-x)-x \leq C_\varepsilon x^2
    \qquad\text{whenever } |x|<1-\varepsilon.
\]
Using this in \eqref{logexpan9}, we get
\[
    \log \E\left[\exp \left(\frac{\gamma}{2N}  \bm Z^{\top} \bK_N \bm Z\right)\right]
    \leq
    \frac{C_\varepsilon\gamma^2}{2N^2}\tr(\bK_N^2).
\]
Finally, note that
\[
    \frac{1}{N^2}\tr(\bK_N^2)
    =
    \frac{1}{N^2}\sum_{i\neq j}
    K\left(\frac{i}{N},\frac{j}{N}\right)^2
    =
    O(1),
\]
since the function $K$ is bounded. Therefore,
\[
    \limsup_{N\to\infty} Z_N^K(1+\delta)<\infty.
\]
This proves \eqref{eq:unifint622}. For the final assertion, note that if we define
\[
    X_N
    =
    \exp \left(\frac{1}{2 N} \sum_{i \neq j} K\left(\frac{i}{N}, \frac{j}{N}\right) \sigma_{i} \sigma_j\right),
\]
then we have just shown that 
$\E X_N^{1+\delta}
    =
    O(1)$.
This is enough to establish uniform integrability of $\{X_N\}_{N\geq1}$, thereby completing the proof of Lemma \ref{generalexpUniInt}.
\end{proof}

\begin{lemma}
\label{lem:fourthmomentinterpolated}
Suppose that $\widetilde S_N$ denotes the linearly interpolated process
\eqref{interpolated66}. Then 
\[
        \sup_{N\ge 1, ~0\le s<t\le 1}
        (t-s)^{-2}~\mathbb E\left[
        \left(
        \widetilde S_N(t)-\widetilde S_N(s)
        \right)^4
        \,\middle|\, A_N
        \right]
        <\infty
\]
almost surely.
\end{lemma}

\begin{proof}
To begin with, define the empirical step-kernel associated with $A_N$ by
\begin{equation}\label{eq:Waneq6}
        W_{A_N}(x,y)
        :=
        \sum_{i,j=1}^N
        A_N(i,j)\mathbf 1\{\lceil Nx\rceil = i, \lceil Ny\rceil = j \}.
\end{equation}
For a kernel $U$, write
\[
        \|U\|_{\mathrm{op}}:=\|T_U\|_{L^2\to L^2}.
\]
By Lemma \ref{lem:empirical-graphon-operator-convergence}, the event
\[
        \mathcal G
        :=
        \left\{
        \left\|
        \frac{W_{A_N}}{\theta_N}-W
        \right\|_{\mathrm{op}}
        \rightarrow 0
        \right\}
        \in \sigma\bigl(\{A_N\}_{N\ge1}\bigr)
\]
has probability $1$. Choose $\varepsilon>0$, such that
\(
        q:=\beta(\lambda_1(W)+\varepsilon)<1,
\)
where $\lambda_1(W)=\|T_W\|_{\mathrm{op}}=\|W\|_{\mathrm{op}}$
(indeed, since $W\geq0$, the operator $T_W$ is positivity preserving,
so its spectral radius is a nonnegative eigenvalue of maximal modulus). Fix
$\omega\in\mathcal G$, whence there exists
a positive integer $N_0 \omega$ such that for all $N\ge N_0 \omega$,
\begin{equation}\label{pf172}
       \left\|
        \frac{W_{A_N \omega}}{\theta_N}-W
        \right\|_{\mathrm{op}}
        <\varepsilon .
\end{equation}

Let $\bar A_N$ be the matrix obtained from $A_N$ by setting its diagonal entries equal
to zero. Since the diagonal terms contribute only a constant to the Hamiltonian, the spin law conditional on $A_N$
may be written as
\[
        \mathbb P_N(\bs)
        \propto
        \exp\left\{
        \sum_{1\le i<j\le N}
        K_N(i,j)\sigma_i\sigma_j
        \right\},
        \qquad
        K_N(i,j):=
        \frac{\beta}{N\theta_N}\bar A_N(i,j).
\]
Further, define the symmetric nonnegative matrix
\[
        B_N(i,j):=
        \begin{cases}
        \tanh K_N(i,j), & i\ne j,\\
        0, & i=j.
        \end{cases}
\]
Since $0\le B_N(i,j)\le K_N(i,j) \le \frac{\beta}{N\theta_N} A_N(i,j)$ for all $i,j$, the Perron--Frobenius
theorem gives
\begin{equation}\label{pf272}
     \|B_N\|_{\mathrm{op}} = \rho(B_N)
        \le
        \rho(K_N)
        =
        \frac{\beta}{N\theta_N}\lambda_1(\bar A_N) \le \frac{\beta}{N\theta_N}\lambda_1(A_N).
\end{equation}
Hence, for every $N\ge N_0 \omega$, it follows from \eqref{pf172} and \eqref{pf272} that at $\omega$:
\[
\begin{aligned}
        \|B_N\|_{\mathrm{op}}
       \le
        \frac{\beta}{N\theta_N}\lambda_1(A_N)       =
        \beta
        \left\|
        \frac{W_{A_N}}{\theta_N}
        \right\|_{\mathrm{op}}                       \le
        \beta\left(
        \|W\|_{\mathrm{op}}
        +
        \left\|
        \frac{W_{A_N}}{\theta_N}-W
        \right\|_{\mathrm{op}}
        \right)                                      \le
        \beta(\lambda_1(W)+\varepsilon)
        =
        q<1.
\end{aligned}
\]
We therefore have the following expansion for all $N \ge N_0\omega$:
\begin{equation}\label{expansionmatrix888}
      \Gamma_N\omega:=(I-B_N\omega)^{-1}
        =
        \sum_{m\ge0}B_N^m\omega .
\end{equation}
Also, $\Gamma_N$ is entrywise nonnegative and it follows from \eqref{expansionmatrix888}, that
\(
        \|\Gamma_N \omega\|_{\mathrm{op}}
        \le
        \frac1{1-q}
\)
for all $N \ge N_0\omega$.

Now, for $i\ne j$, it follows from \cite[Eq.~(46)]{SaadeKrzakalaZdeborova2017} (also see \cite[Section 2.2.1]{DuminilCopin2016}) that
\[
        \mathbb E_N[\sigma_i\sigma_j]
        \le
        \sum_{\gamma \in P(i\to j)}
        \prod_{e\in E(\gamma)} B_N(e) \le \sum_{m\ge 1}\sum_{1\le i_1,\ldots,i_{m-1}\le N} B_N(i,i_1) B_N(i_1,i_2)\ldots B_N(i_{m-1},j),
\]
where $P(i\to j)$ denotes the set of all paths from $i$ to $j$, $E(\gamma)$ denotes the edge-set of $\gamma$, and $\e_N(\cdot)$ denotes the conditional expectation operator $\e(\cdot|A_N)$. Therefore, for all $i\ne j$, we have for all $N \ge N_0\omega$:
\[
        \mathbb E_N[\sigma_i\sigma_j]\omega
        \le
        \sum_{m\ge1} B_N^m(i,j)\omega
        \le
        \Gamma_N(i,j)\omega.
\]
Also, for $i=j$, we trivially have:
\[
        \mathbb E_N[\sigma_i^2]=1\le \Gamma_N (i,i)\omega.
\]
Thus, for every $N\ge N_0 \omega$, we have \(
        \mathbb E_N[\sigma_i\sigma_j]\omega
        \le
        \Gamma_N(i,j)\omega
\) for all $i,j$,
 and
consequently, for every $\bm u\in[0,\infty)^N$,

\begin{equation}\label{reqpr66}
        \mathbb E_N\langle \bm u,\bs\rangle^2 \omega
        =
        \sum_{i,j=1}^N u_i u_j\,
        \mathbb E_N[\sigma_i\sigma_j] \omega \le
        \bm u^\top \Gamma_N \omega\bm u                                      \le
        \|\Gamma_N \omega\|_{\mathrm{op}}\|\bm u\|_2^2                     \le
        \frac{\|\bm u\|_2^2}{1-q}.
\end{equation}

Next, by the Lebowitz four-point inequality \cite[p.~91, Remark~(i)]{Lebowitz1974}, we have:
\[
\begin{aligned}
        \mathbb E_N[\sigma_i\sigma_j\sigma_k\sigma_\ell]
        \le
        \mathbb E_N[\sigma_i\sigma_j]
        \mathbb E_N[\sigma_k\sigma_\ell] +
        \mathbb E_N[\sigma_i\sigma_k]
        \mathbb E_N[\sigma_j\sigma_\ell] +
        \mathbb E_N[\sigma_i\sigma_\ell]
        \mathbb E_N[\sigma_j\sigma_k],
\end{aligned}
\]
for distinct $i,j,k,\ell$.
If some indices coincide, the same displayed bound follows directly from
the facts that $\sigma_a^2=1$ for all $1\le a\le N$, and that the two-point correlations are all nonnegative (see \cite[Eq. 12]{DuminilCopin2016}). Therefore, for every $\bm u\in[0,\infty)^N$ and all $N\ge N_0 \omega$, we have:
\[
\begin{aligned}
        \mathbb E_N\langle \bm u, \bm\sigma\rangle^4
       \omega &=
        \sum_{i,j,k,\ell=1}^N
        u_i u_j u_k u_\ell\,
        \mathbb E_N[\sigma_i\sigma_j\sigma_k\sigma_\ell]\omega\\
        &\le
        3
        \left(
        \sum_{i,j=1}^N
        u_i u_j\,
        \mathbb E_N[\sigma_i\sigma_j]\omega
        \right)^2\\
        &=
        3\left(\mathbb E_N\langle \bm u,\bm \sigma\rangle^2 \omega\right)^2\\
        &\le
        \frac{3}{(1-q)^2}\|\bm u\|_2^4 .
\end{aligned}
\]
Now, recall from \eqref{interpolated66}, that
\[
        \widetilde S_N(t)
        =
        \frac1{\sqrt N}
        \sum_{i=1}^N h_{i,N}(t)\sigma_i,
        \qquad \text{where}~
        h_{i,N}(t) = \bigl(Nt-(i-1)\bigr)_+\wedge1 .
\]
For $0\le s<t\le1$, define
\[
        u_{i,N}(s,t):=
        \frac{h_{i,N}(t)-h_{i,N}(s)}{\sqrt N},
        \qquad i=1,\ldots,N.
\]
Then $\bm u(s,t) := (u_{i,N}(s,t))_{1\le i\le N} \in[0,\infty)^N$ and
\(
        \widetilde S_N(t)-\widetilde S_N(s)
        =
        \langle \bm u(s,t),\bm \sigma\rangle .
\)
Moreover,
\[
\begin{aligned}
        \|\bm u(s,t)\|_2^2
        &=
        \frac1N
        \sum_{i=1}^N
        \{h_{i,N}(t)-h_{i,N}(s)\}^2\\
        &\le
        \frac1N
        \sum_{i=1}^N
        \{h_{i,N}(t)-h_{i,N}(s)\}
        =
        t-s,
\end{aligned}
\]
because $0\le h_{i,N}(t)-h_{i,N}(s) \le1$. 
Hence, for every
$N\ge N_0 \omega$, we have:
\[
        \mathbb E_N
        \left(
        \widetilde S_N(t)-\widetilde S_N(s)
        \right)^4
        \omega
        \le
        \frac{3}{(1-q)^2}(t-s)^2 .
\]

It remains only to control the LHS of the above inequality for all $N<N_0 \omega$. Towards this, note that for every
$N$ and every $0\le s<t\le1$,
\[
        \left|
        \widetilde S_N(t)-\widetilde S_N(s)
        \right|
        \le
        \frac1{\sqrt N}
        \sum_{i=1}^N
        (h_{i,N}(t)-h_{i,N}(s))
        =
        \sqrt N\,(t-s).
\]
Therefore, for all $N < N_0\omega$, we have:
\[
        \mathbb E_N       \left(
        \widetilde S_N(t)-\widetilde S_N(s)
        \right)^4
        \omega
        \le
        N^2(t-s)^4
        <
        (N_0 \omega)^2(t-s)^2
\]
for all $N<N_0 \omega$. Thus the choice
\[
        C_\omega
        :=
        \max\left\{
        \frac{3}{(1-q)^2},
       (N_0 \omega)^2
        \right\}
\]
gives
\[
        \mathbb E_N
        \left(
        \widetilde S_N(t)-\widetilde S_N(s)
        \right)^4
        \omega
        \le
        C_\omega (t-s)^2
\]
for every $N\ge1$ and every $0\le s<t\le1$. Taking supremum over
$N\ge1$ and every $0\le s<t\le1$, we have:
\begin{equation}\label{suptsineq882}
    \sup_{N\ge 1, ~0\le s<t\le 1}
        (t-s)^{-2}~\mathbb E\left[
        \left(
        \widetilde S_N(t)-\widetilde S_N(s)
        \right)^4
        \,\middle|\, A_N
        \right]\omega
        \le C_\omega < \infty.
\end{equation}
Since \eqref{suptsineq882} is true for all $\omega\in \mathcal G$ and since $\p(\mathcal G) =1$, Lemma \ref{lem:fourthmomentinterpolated} follows.
\end{proof}

\begin{lemma}
\label{lem:empirical-graphon-operator-convergence}
Let $W_{A_N}$ be defined as in \eqref{eq:Waneq6}. Then, under Assumption~\ref{assumption12}~\textup{(2)--(3)},
\[
        \left\|
        \frac{W_{A_N}}{\theta_N}-W
        \right\|_{\mathrm{op}}
        \rightarrow 0
        \qquad\text{almost surely}.
\]
\end{lemma}

\begin{proof}
To begin with, define the deterministic step-kernel
\[
        W_N(x,y)
        :=
        \sum_{i,j=1}^N
        W\left(\frac{i}{N},\frac{j}{N}\right)
        \mathbf 1\{\lceil Nx\rceil = i, \lceil Ny\rceil = j \},
\]
and note that \(
        \mathbb E W_{A_N}=\theta_N W_N.
\)
By the triangle inequality,
\begin{equation}
\label{eq:step-kernel-decomp}
        \left\|
        \frac{W_{A_N}}{\theta_N}-W
        \right\|_{\mathrm{op}}
        \le
        \left\|
        \frac{W_{A_N}}{\theta_N}-W_N
        \right\|_{\mathrm{op}}
        +
        \|W_N-W\|_{\mathrm{op}}.
\end{equation}
Since $W$ is Riemann integrable by
Assumption~\ref{assumption12}~\textup{(2)}, we have:
\begin{equation}\label{2nsdetbdes8}
  \|W_N-W\|_{\mathrm{op}}
        =
        \|T_{W_N-W}\|_{L^2\to L^2}
        \le
        \|W_N-W\|_{L^2} \le  \|W_N-W\|_{L^1}^{\frac12}
        \rightarrow0.
\end{equation}

It remains to control the random term in the RHS of \eqref{eq:step-kernel-decomp}. Towards this, set 
\[
K_N:=\frac{W_{A_N}}{\theta_N}-W_N
\qquad \text{and}\qquad
\widetilde A_N :=A_N-\E A_N,
\]
and let \(I_j^N :=((j-1)/N,j/N]\). Fix \(f\in L^2[0,1]\) and define
\(b_j:=N\int_{I_j^N}f(y)\,dy\). Then, for \(x\in I_i^N\),
\[
(T_{K_N}f)(x)=\frac{1}{N\theta_N}(\widetilde A_N\bm b)_i,
\]
and hence
\[
\|T_{K_N}f\|_{L^2}^2
=
\frac{\|\widetilde A_N\bm b\|_2^2}{N^3\theta_N^2}
\le
\frac{\|\widetilde A_N\|_{\op}^2}{N^2\theta_N^2}\|f\|_{L^2}^2,
\]
where in the last inequality, we used
\(\|\bm b\|_2^2\le N\|f\|_{L^2}^2\), a simple consequence of the Cauchy--Schwarz inequality. Since this is true for all $f \in L^2[0,1]$, we have:
\[
\|K_N\|_{\op}\le\frac{\|\widetilde A_N\|_{\op}}{N\theta_N}.
\]
Conversely, choose \(\bm v\in\mathbb R^N\) with
\(\|\bm v\|_2=1\) such that
\(\|\widetilde A_N\bm v\|_2=\|\widetilde A_N\|_{\op}\), and set
\(f :=\sqrt N\sum_{j=1}^Nv_j\one_{I_j^N}\). Then
\(\|f\|_{L^2}=1\), \(\bm b=\sqrt N\,\bm v\), and
\[
\|T_{K_N}f\|_{L^2}
= \frac{\|\widetilde A_N(\sqrt{N}\bm v)\|_2}{N^{3/2}\theta_N}=
\frac{\|\widetilde A_N\|_{\op}}{N\theta_N}\quad\implies \quad\|K_N\|_{\op} = \|T_{K_N}\|_\op \ge \frac{\|\widetilde A_N\|_{\op}}{N\theta_N}.
\]
Therefore, we have:
\begin{equation}\label{eq:matrix-kernel-norm}
    \left\|
\frac{W_{A_N}}{\theta_N}-W_N
\right\|_{\op}
=
\frac{\|\widetilde A_N\|_{\op}}{N\theta_N}.
\end{equation}

Throughout the rest of the proof, for notational convenience, let us write
\[
        p_{ij}:=\theta_N W\left(\frac{i}{N},\frac{j}{N}\right).
\]
For \(1\le i\le j\le N\), define
\[
X_{ij}
:=
\begin{cases}
(A_N(i,j)-p_{ij})
(e_i e_j^\top+e_j e_i^\top)\quad
& \text{if}~ i<j,\\[4pt]
(A_N(i,i)-p_{ii})e_i e_i^\top \quad
& \text{if}~i=j.
\end{cases}
\]
It is straightforward to see that
\[
        \widetilde A_N=\sum_{1\le i\le j\le N}X_{ij},
\]
where the summands are independent, self-adjoint and centered. Moreover,
\(
        \|X_{ij}\|_{\mathrm{op}}\le1.
\)
Also, a direct calculation gives
\[
        \sum_{1\le i\le j\le N}\mathbb E X_{ij}^2
        =
        \operatorname{diag}\left(
        \sum_{j=1}^N p_{ij}(1-p_{ij})
        \right)_{i=1}^N,
\]
and therefore
\begin{equation}
\label{eq:variance-proxy}
        \left\|
        \sum_{1\le i\le j\le N}\mathbb E X_{ij}^2
        \right\|_{\mathrm{op}}
        \le
        \max_{1\le i\le N}\sum_{j=1}^N p_{ij}
        \le
        N\theta_N.
\end{equation}

Applying the matrix Bernstein inequality
\cite[Theorem~1.4]{Tropp2012} to $\sum_{i\le j}X_{ij}$ and to
$-\sum_{i\le j}X_{ij}$, and using \eqref{eq:variance-proxy}, we have the following for
every $r>0$,
\[
        \mathbb P\left(\|\widetilde A_N\|_{\mathrm{op}}\ge r\right)
        \le
        2N
        \exp\left\{
        -\frac{r^2}
        {2\left(N\theta_N+r/3\right)}
        \right\}.
\]
Taking $r=\varepsilon N\theta_N$, for some fixed $\varepsilon>0$, yields
\begin{equation}
\label{eq:bernstein-adjacency}
        \mathbb P\left(
        \frac{\|\widetilde A_N\|_{\mathrm{op}}}{N\theta_N}\ge\varepsilon
        \right)
        \le
        2N\exp\{-c_\varepsilon N\theta_N\},
        \qquad \text{where}~
        c_\varepsilon :=\frac{\varepsilon^2}{2(1+\varepsilon/3)}>0.
\end{equation}
By Assumption~\ref{assumption12}~\textup{(3)}, there exists a constant $c>0$ such that
\(
        N\theta_N\ge cN^{1/3}
\)
for all sufficiently large $N$. Hence, for every fixed $\varepsilon>0$,
\[
        \sum_{N=1}^\infty
        \mathbb P\left(
        \frac{\|\widetilde A_N\|_{\mathrm{op}}}{N\theta_N}\ge\varepsilon
        \right)
        <\infty,
\]
which immediately implies that
\[
        \frac{\|\widetilde A_N\|_{\mathrm{op}}}{N\theta_N}
        \rightarrow0
        \qquad\text{almost surely}.
\]
Together with \eqref{eq:matrix-kernel-norm}, this implies that
\begin{equation}
\label{eq:random-step-conv}
        \left\|
        \frac{W_{A_N}}{\theta_N}-W_N
        \right\|_{\mathrm{op}}
        \rightarrow0
        \qquad\text{almost surely}.
\end{equation}
Combining \eqref{eq:step-kernel-decomp},
\eqref{2nsdetbdes8} and \eqref{eq:random-step-conv}, we thus conclude that
\[
        \left\|
        \frac{W_{A_N}}{\theta_N}-W
        \right\|_{\mathrm{op}}
        \rightarrow0
        \qquad\text{almost surely},
\]
 which completes the proof of Lemma \ref{lem:empirical-graphon-operator-convergence}.
\end{proof}

\begin{lemma}\label{soblimexists}
    In the notations of Section \ref{sobolevcan8}, the limit
\(
\iota_s\nu
:=
\lim_{m\to\infty}\iota_m\nu
\)
exists in \(H^{-s}(0,1)\) for every $s>1/2$. Further, the map
\(
\iota_s:\mathcal M([0,1])\to H^{-s}(0,1)
\)
is linear and injective.
\end{lemma}
\begin{proof}
    Since \(\|e_k\|_\infty\leq\sqrt{2}\), for every \(s>1/2\), we have
\[
\sum_{k=0}^{\infty}
\kappa_k^{-s}|\widehat{\nu}_k|^2
\leq
2\|\nu\|_{\mathrm{TV}}^2
\sum_{k=0}^{\infty}\kappa_k^{-s}
<\infty.
\]
Consequently, the finite cosine sums
\[
\iota_m\nu
:=
\sum_{k=0}^m\widehat{\nu}_k e_k
\]
form a Cauchy sequence in \(H^{-s}(0,1)\) for every \(s>1/2\), since
\[
\sup_{n\geq m}
\|\iota_n\nu-\iota_m\nu\|_{H^{-s}}^2
=
\sup_{n\geq m}
\sum_{k=m+1}^n
\kappa_k^{-s}|\widehat{\nu}_k|^2
\leq
\sum_{k=m+1}^{\infty}
\kappa_k^{-s}|\widehat{\nu}_k|^2
\xrightarrow[m\to\infty]{}0.
\]
Since each \(\iota_m\nu\in S^\#\) and \(H^{-s}(0,1)\) is the
completion of \(S^\#\) under \(\|\cdot\|_{H^{-s}}\), the limit
\[
\iota_s\nu
:=
\lim_{m\to\infty}\iota_m\nu
\]
exists in \(H^{-s}(0,1)\).  Moreover, if \(\iota_s\nu=0\), then
\(\widehat{\nu}_k=0\) for every \(k\geq0\). Since cosine polynomials
are uniformly dense in \(C[0,1]\), it follows that
\(\int h\,d\nu=0\) for every \(h\in C[0,1]\), and hence \(\nu=0\).
\end{proof}

\begin{lemma}\label{lchiprop8}
   In the notations of Section \ref{sec:limelem688}, the RHS of \eqref{eqchifser} converges in $L^2(\Omega^*)$. Further, \(\chi\) is a well-defined centered Gaussian linear process
on \(L^2[0,1]\) with covariance
$$\e^*[\chi(f)\chi(g)] = \langle f, R_\beta g\rangle.$$
\end{lemma}
\begin{proof}
Setting
\(
a_j:=\langle R_\beta^{1/2}f,\psi_j\rangle
\)
and
\(Y_j:=a_jZ_j\),
an application of Parseval's identity gives
\[
\sum_{j=1}^{\infty}\operatorname{Var}(Y_j)
=
\sum_{j=1}^{\infty}|a_j|^2
=
\|R_\beta^{1/2}f\|_2^2
=
\langle f,R_\beta f\rangle
<\infty.
\]
Hence, Kolmogorov's two-series theorem implies that
\(\sum_{j\geq1}Y_j\) converges almost surely. Moreover,
\[
\sup_{n\geq m}
\mathbb E^*\left|
\sum_{j=m+1}^{n}Y_j
\right|^2
=
\sup_{n\geq m}
\sum_{j=m+1}^{n}|a_j|^2
\longrightarrow0
\qquad\text{as }m\to\infty.
\]
Thus, the partial sums are Cauchy in \(L^2(\Omega^*)\), and therefore
the series \eqref{eqchifser} converges in \(L^2(\Omega^*)\) as
well. Thus, \(\chi\) is a well-defined centered Gaussian linear process
on \(L^2[0,1]\) with covariance
\begin{eqnarray*}
\mathbb E^*[\chi(f)\chi(g)]
&=&
\mathbb E^*\left[
\left(
\sum_{j=1}^{\infty}
\langle R_\beta^{1/2}f,\psi_j\rangle Z_j
\right)
\left(
\sum_{\ell=1}^{\infty}
\langle R_\beta^{1/2}g,\psi_\ell\rangle Z_\ell
\right)
\right]
\\
&=&
\sum_{j,\ell=1}^{\infty}
\langle R_\beta^{1/2}f,\psi_j\rangle
\langle R_\beta^{1/2}g,\psi_\ell\rangle
\mathbb E^*[Z_jZ_\ell]
\\
&=&
\sum_{j=1}^{\infty}
\langle R_\beta^{1/2}f,\psi_j\rangle
\langle R_\beta^{1/2}g,\psi_j\rangle
\\
&=&
\langle R_\beta^{1/2}f,R_\beta^{1/2}g\rangle
\qquad\text{(by Parseval's identity)}
\\
&=&
\langle f,R_\beta g\rangle.
\end{eqnarray*}
Here the passage to the infinite sums is justified by their
\(L^2(\Omega^*)\)-convergence.    
\end{proof}

\begin{lemma}\label{sobactlem552}
    Consider the setup of Remark \ref{remk:sobac72}. The series in the RHS of \eqref{etahactiondef6} is absolutely convergent almost surely. Moreover, for every \(h\in H^1(0,1)\), the two random variables
\(\langle\eta,h\rangle_{-1,1}\) and \(\chi(h)\) are equal
\(\mathbb P^*\)-almost surely.
\end{lemma}

\begin{proof}
    It follows from the Cauchy--Schwarz inequality that with probability \(1\),
\[
\begin{aligned}
\sum_{k=0}^{\infty}
\left|\chi(e_k)\langle h,e_k\rangle\right|
&=
\sum_{k=0}^{\infty}
\bigl(\kappa_k^{-1/2}|\chi(e_k)|\bigr)
\bigl(\kappa_k^{1/2}|\langle h,e_k\rangle|\bigr)\\
&\leq
\left(
\sum_{k=0}^{\infty}
\kappa_k^{-1}\chi(e_k)^2
\right)^{1/2}
\left(
\sum_{k=0}^{\infty}
\kappa_k|\langle h,e_k\rangle|^2
\right)^{1/2}\\
&=
\|\eta\|_{H^{-1}}\|h\|_{H^1}
<\infty,
\end{aligned}
\]
so the series in the RHS of \eqref{etahactiondef6} is absolutely convergent almost surely. Moreover, we
have
\begin{eqnarray*}
\langle\eta,h\rangle_{-1,1}
&=&
\left\langle
\sum_{k=0}^{\infty}\chi(e_k)e_k,h
\right\rangle_{-1,1}
=
\sum_{k=0}^{\infty}
\chi(e_k)\langle h,e_k\rangle
\\
&=&
\chi\left(
\sum_{k=0}^{\infty}
\langle h,e_k\rangle e_k
\right)
=
\chi(h)
\qquad\text{in }L^2(\Omega^*).
\end{eqnarray*}
The equalities follow by passing to the finite cosine sums, using the
continuity of the \(H^{-1}\)--\(H^1\) duality pairing, the convergence
of the cosine expansion of \(h\) in \(L^2[0,1]\), and the bounded
linearity of
\(\chi:L^2[0,1]\to L^2(\Omega^*)\); we skip the details.
Thus, for each fixed \(h\in H^1(0,1)\), the two random variables
\(\langle\eta,h\rangle_{-1,1}\) and \(\chi(h)\) are equal
\(\mathbb P^*\)-almost surely.
\end{proof}

\section{Fredholm Determinants and Rank-One Perturbations}
In this section, we state some fundamental results in linear algebra, which are crucial in our analysis. The first lemma follows from Lemma 2.8 in \cite{krajenbrink2020} on taking their background kernel $K_s$ to be $0$.

\begin{lemma}
    \label{detrank1perturb}
    Let $f,g \in L^2[0,1]$. Then $\det(I+\opb{f}{g})=1+\ipb{f}{g}.$
\end{lemma}

The next lemma is crucial in the proof of Theorem \ref{thm:isingWeakConv}.

\begin{lemma}
    \label{lem:detfact}
    Let $f: [0,1] \to \R$ be Riemann integrable and $\beta < \frac{1}{\|W\|_{\mathrm{op}}}$. Then for small enough $t$,
    \begin{equation*}
        e^{t\int f^2}\sqrt{\frac{\dett(I-\beta T_W)}{\dett(I-\beta T_W-2t\opb{f}{f})}}=\frac{1}{\sqrt{(1-2t\ipb{f}{(I-\beta T_W)^{-1}f})}}.
    \end{equation*}
\end{lemma}

\begin{proof}
    Since \(\beta\|T_W\|_{\mathrm{op}}<1\), the operator \(A\coloneqq I-\beta T_W\) is invertible. Let
    \(g\coloneqq -2tA^{-1}f\). Then,
    \[
        A(I+\opb{g}{f})
        =
        A+A\opb{g}{f}
        =
        A+\opb{Ag}{f}
        =
        I-\beta T_W-2t\opb{f}{f}.
    \]
    Hence, using the following multiplicative identity for the 2-modified Fredholm determinant:
    \[\dett((I+B)(I+C))=\dett(I+B)\dett(I+C)e^{-\tr(BC)}\] for Hilbert--Schmidt operators
    \(B,C\) (see Chapter~9, Remark~2 in \cite{simon-2010}), with \(B=-\beta T_W\) and
    \(C=\opb{g}{f}\), we get
    \begin{equation}\label{eq:det2_factor_rankone}
        \dett(I-\beta T_W-2t\opb{f}{f})
        =
        \dett(I-\beta T_W)\dett(I+\opb{g}{f})
        e^{\tr(\beta T_W\opb{g}{f})}.
    \end{equation}

    Now, since \(\opb{g}{f}\) has rank one, it is trace--class, and therefore we have:
    \[\dett(I+\opb{g}{f})=\det(I+\opb{g}{f})e^{-\tr(\opb{g}{f})}.\] Also, by Lemma
    \ref{detrank1perturb}, we have:
    \[
        \det(I+\opb{g}{f})
        =
        1+\ipb{f}{g}
        =
        1-2t\ipb{f}{A^{-1}f}.
    \]
    Thus, \eqref{eq:det2_factor_rankone} gives:
    \[
        \dett(I-\beta T_W-2t\opb{f}{f})
        =
        \dett(I-\beta T_W)
        \left(1-2t\ipb{f}{A^{-1}f}\right)
        e^{-\tr(\opb{g}{f})+\tr(\beta T_W\opb{g}{f})}.
    \]
    The exponent can be simplified as:
    \[
        -\tr(\opb{g}{f})+\tr(\beta T_W\opb{g}{f})
        =
        -\tr(A\opb{g}{f})
        =
        -\tr(\opb{Ag}{f})
        =
        2t\tr(\opb{f}{f})
        =
        2t\int_0^1 f^2 .
    \]
    Therefore, we have:
    \[
        \dett(I-\beta T_W-2t\opb{f}{f})
        =
        \dett(I-\beta T_W)
        \left(1-2t\ipb{f}{(I-\beta T_W)^{-1}f}\right)
        e^{2t\int f^2}.
    \]
    Rearranging the above identity, we obtain:
    \[
        e^{t\int f^2}
        \sqrt{\frac{\dett(I-\beta T_W)}
        {\dett(I-\beta T_W-2t\opb{f}{f})}}
        =
        \frac{1}{\sqrt{1-2t\ipb{f}{(I-\beta T_W)^{-1}f}}}.
    \]
    This completes the proof of Lemma \ref{lem:detfact}.
\end{proof}

\section{Properties of the Gaussian process $X(t)$}
In this section, we state some properties of the Gaussian process $X(t)$ defined in \eqref{defXt}.

\begin{lemma}
    \label{constructXt}
    The Gaussian process $X(t)$ defined in \eqref{defXt} is well-defined with covariance kernel
    \begin{equation*}
        K(s,t)=\min\{s,t\} + \sum_{i=1}^\infty \frac{\beta \lambda_i}{1-\beta \lambda_i} \Phi_i(t)\Phi_i(s).
    \end{equation*}
\end{lemma}

\begin{proof}
Write
\(
    a_i\coloneqq (1-\beta\lambda_i)^{-1/2} -1\), and recall the following notations:
  \[
    \Phi_i(t)\coloneqq \int_0^t \phi_i(x)\,dx,
    \qquad
    Z_i\coloneqq \int_0^1 \phi_i(t)\,dB(t).
\]
Then the random variables \(Z_i\) are independent standard Gaussians, since
\(\{\phi_i\}\) is orthonormal, and by It\^o isometry,
\[
    \mathbb E[Z_iZ_j]
    =
    \langle \phi_i,\phi_j\rangle
    =
    \delta_{ij}.
\]
Since \(T_W\) is Hilbert--Schmidt, \(\sum_i\lambda_i^2<\infty\). Moreover,
since \(\beta\|W\|_{\mathrm{op}}<1\), we have:
\[
    |a_i|\le C|\lambda_i|
\]
for some constant \(C = C(\beta,\|W\|_{\mathrm{op}}) <\infty\). 
Also, for each fixed \(t\in[0,1]\), note that:
\[
    |\Phi_i(t)|
    =
    |\langle \phi_i,\mathbf 1_{[0,t]}\rangle|
    \le
    \|\phi_i\|_2\|\mathbf 1_{[0,t]}\|_2
    \le 1,
\]
and hence,
\[
    \sum_i a_i^2\Phi_i(t)^2
    \le
    C^2\sum_i\lambda_i^2
    <
    \infty.
\]
Therefore, the sequence $R_n(t) := \sum_{i=1}^n a_i\Phi_i(t)Z_i$ is Cauchy in $L^2$, and hence, the limit
\(
    R(t)\coloneqq \sum_{i=1}^\infty a_i\Phi_i(t)Z_i
\)
is well-defined as an \(L^2\) object. Consequently,
\(
    X(t)=B(t)+R(t)
\)
is well-defined. Moreover, every finite vector
\((X(t_1),\ldots,X(t_m))\) is Gaussian, since the process $X(t)$ is an \(L^2\)-limit of
Gaussian processes.

It remains to compute the covariance of the process $X(t)$. 
Since \(R_n(t)\to R(t)\) in \(L^2\), we may compute covariances by passing to
the limit. First,
\[
    \mathbb E[B(s)Z_i]
    =
    \mathbb E\left[
    \left(\int_0^1\mathbf 1_{[0,s]}(x)\,dB(x)\right)
    \left(\int_0^1\phi_i(x)\,dB(x)\right)
    \right]
    =
    \Phi_i(s).
\]
Thus,
\[
    \mathbb E[B(s)R(t)]
    =
    \sum_{i=1}^\infty a_i\Phi_i(s)\Phi_i(t).
\]
The convergence of this series follows from Cauchy--Schwarz and Bessel's
inequality:
\[
    \sum_i |a_i\Phi_i(s)\Phi_i(t)|
    \le
    \left(\sum_i a_i^2\Phi_i(t)^2\right)^{1/2}
    \left(\sum_i \Phi_i(s)^2\right)^{1/2}
    \le
    C.
\]
Similarly,
\[
    \mathbb E[R(s)R(t)]
    =
    \sum_{i=1}^\infty a_i^2\Phi_i(s)\Phi_i(t),
\]
and this series is absolutely convergent because
\[
    \sum_i |a_i^2\Phi_i(s)\Phi_i(t)|
    \le
    \left(\sum_i a_i^2\Phi_i(s)^2\right)^{1/2}
    \left(\sum_i a_i^2\Phi_i(t)^2\right)^{1/2}
    <\infty .
\]
Therefore,
\begin{align*}
    \operatorname{Cov}(X(s),X(t))
    &=
    \mathbb E[B(s)B(t)]
    +
    \mathbb E[B(s)R(t)]
    +
    \mathbb E[B(t)R(s)]
    +
    \mathbb E[R(s)R(t)] \\
    &=
    \min\{s,t\}
    +
    \sum_{i=1}^\infty
    \left(2a_i+a_i^2\right)\Phi_i(s)\Phi_i(t).
\end{align*}
Finally, note that:
\[
    2a_i+a_i^2
    =
    \left(1+a_i\right)^2-1
    =
    \frac{1}{1-\beta\lambda_i}-1
    =
    \frac{\beta\lambda_i}{1-\beta\lambda_i}.
\]
Hence, we have:
\[
    \operatorname{Cov}(X(s),X(t))
    =
    \min\{s,t\}
    +
    \sum_{i=1}^\infty
    \frac{\beta\lambda_i}{1-\beta\lambda_i}
    \Phi_i(s)\Phi_i(t),
\]
as claimed.
\end{proof}

\begin{lemma}\label{holderX}
    The Gaussian process \(X\) admits a modification whose sample paths are
    \(\gamma\)-Hölder continuous for every \(0<\gamma<\frac12\).
\end{lemma}

\begin{proof}
    To begin with, for \(t\in[0,1]\), define \(f_t\coloneq\one_{[0,t]}\). By Lemma
    \ref{constructXt}, we have:
    \[
        \E[X(s)X(t)]
        =
        \left\langle
        f_s,(I-\beta T_W)^{-1}f_t
        \right\rangle .
    \]
    Hence, for \(0\le s<t\le1\), we have:
    \begin{align*}
        \E\!\left[(X(t)-X(s))^2\right]
        &=
        \left\langle
        f_t-f_s,(I-\beta T_W)^{-1}(f_t-f_s)
        \right\rangle \\
        &\le
        \bigl\|(I-\beta T_W)^{-1}\bigr\|_{\mathrm{op}}
        \|f_t-f_s\|_{L^2}^2 \\
        &\le
        \frac{t-s}{1-\beta\|T_W\|_{\mathrm{op}}},
    \end{align*}
    where the last inequality follows from the fact that for $\|\beta T_W\|_{\mathrm{op}} <1$, we have:
    \[
        (I-\beta T_W)^{-1}
        =
        \sum_{m=0}^\infty \beta^m T_W^m .
    \]
    Since \(X(t)-X(s)\) is centered Gaussian, for every \(p\ge2\),
    \[
        \E|X(t)-X(s)|^p
        \le C_p|t-s|^{p/2},
    \]
    where \(C_p<\infty\) is independent of \(s\) and \(t\). It thus follows from the Kolmogorov--Chentsov theorem
    \cite[Theorem~2.9]{leGall2016}, that for every $p>2$, $X$ admits a modification \(X^{(p)}\)
    whose paths are \(\gamma\)-H\"older continuous for every
    \[
        0<\gamma<
        \frac{p/2-1}{p}
        =
        \frac12-\frac1p .
    \]

    Now, let \(\widetilde X\coloneq X^{(3)}\). For every integer \(p\ge3\),
    the processes \(\widetilde X\) and \(X^{(p)}\) are continuous
    modifications of \(X\), and hence are indistinguishable, i.e. they agree
    almost surely at every time $t \in [0,1]$. Intersecting the underlying countably many
    probability-one events shows that almost surely, \(\widetilde X\) is
    \(\gamma\)-Hölder continuous for every
    \(0<\gamma<\frac12-\frac{1}{p}\) and every $p \ge 3$, which in turn, means that almost surely, \(\widetilde X\) is
    \(\gamma\)-Hölder continuous for every
    \(0<\gamma<\frac12\). This completes the proof.
\end{proof}

\section{Supplementary Results}
In this section, we collect several supplementary results from the literature that play an important role in the proofs of our main results. The following result corresponds to the first part of Lemma F.1 (Eq. (F.1)) in \cite{mukherjee2026isinginference}. Throughout Lemmas~\ref{lem:expect2} and~\ref{lem:cov2}, all
$O(\cdot)$ terms are uniform over the displayed spin
configurations, with constants independent of $N$.

\begin{lemma}\label{lem:expect2}
    For all $\bs \in \{-1,+1\}^N$, as long as $\theta_N = \Omega(N^{-2/3})$, we have:
    \begin{multline*}
        \E T\left(\bs\right) = \exp\left(-\frac{\beta^2}{4N^2\theta_N}\sum_{i=1}^{N}W\left(\frac{i}{N},\frac{i}{N}\right)
        -\frac{\beta^2}{2N^2}\sum_{i> j}W^2\left(\frac{i}{N},\frac{j}{N}\right)-\frac{\beta^4}{12 N^4\theta_N^3}\sum_{i> j} W\left(\frac{i}{N},\frac{j}{N}\right)\right.\\
        \left.+o\left(\frac 1{N\theta_N}\right)+\frac{\beta}{N}\sum_{i> j}W\left(\frac{i}{N},\frac{j}{N}\right)\sigma_i \sigma_j + O\left(\frac{1}{N^3 \theta_N^2}\left|\sum_{i>j}W\left(\frac{i}{N},\frac{j}{N}\right)\sigma_i \sigma_j\right| \right)\right).
    \end{multline*}
\end{lemma}

The next result corresponds to the second part of Lemma F.1 (Eq. (F.3)) in \cite{mukherjee2026isinginference}.
\begin{lemma}\label{lem:cov2}
    For all $\bs, \bt \in \{-1,+1\}^N$, as long as $\theta_N = \Omega(N^{-2/3})$, we have:
    \begin{eqnarray*}
        &&\E T\left(\bs\right) T\left(\bt\right)\\ &=&
        \exp\Bigg(-\frac{\beta^2}{2N^2\theta_N}\sum_{i=1}^{N}W\left(\frac{i}{N},\frac{i}{N}\right)
        -\frac{\beta^2}{N^2}\sum_{i> j}W^2\left(\frac{i}{N},\frac{j}{N}\right)-\frac{\beta^4}{6 N^4\theta_N^3}\sum_{i> j} W\left(\frac{i}{N},\frac{j}{N}\right) + o\left(\frac 1{N\theta_N}\right)\\
        &+& \frac{\beta}{N}\sum_{i> j}W\left(\frac{i}{N},\frac{j}{N}\right)(\sigma_i \sigma_j + \tau_i\tau_j)+\frac{\beta^2}{\theta_N N^2} \sum_{i>j}W\left(\frac{i}{N},\frac{j}{N}\right)\left(1-\theta_N W\left(\frac{i}{N},\frac{j}{N}\right)\right)\sigma_i\sigma_j\tau_i\tau_j\\
        &+& O\left(\frac{1}{N^3 \theta_N^2}\sum_{i>j}W\left(\frac{i}{N},\frac{j}{N}\right)\sigma_i \sigma_j \right) + O\left(\frac{1}{N^3 \theta_N^2}\sum_{i>j}W\left(\frac{i}{N},\frac{j}{N}\right)\tau_i \tau_j\right)\\&+&O\left(\frac{1}{\theta_N^3 N^4}\sum_{i>j}W\left(\frac{i}{N},\frac{j}{N}\right)\sigma_i\sigma_j\tau_i\tau_j\right).
    \end{eqnarray*}
\end{lemma}

The proof of the next lemma follows the argument of Lemma A.6 in \cite{mukherjee2026isinginference}, and is omitted.

\begin{lemma}
    \label{lem:taylor_expRn}
    Let $\nu$ denote the uniform probability
measure on \(\{-1,+1\}^{2N}\). Define
    $$\Lambda_N \coloneqq R_N(\bs,\bt)+C_0\left(Q_N(\bs)+Q_N(\bt)\right)+ \frac{C_1}{N \theta_N} Q_N(\bs \bt)$$
    for some positive constants $C_0$ and $C_1$, where $R_N$ and $Q_N$ are as defined in \eqref{RNdef:def8} and \eqref{QNdef:def8}, respectively.
    Then, for every $1\le p<\infty$, we have:
    \begin{equation*}
        \exp(\Lambda_N) = 1 + \Lambda_N + o_{L^p(\nu)}\left(\frac{1}{N\theta_N}\right)
    \end{equation*}
\end{lemma} 

The following result corresponds to Lemma A.7 in \cite{mukherjee2026isinginference}.

\begin{lemma}\label{lem:Qnlp}
    For every $1\leq p<\infty$, we have:
    \begin{equation*}
        \left(\E_{\bs \sim \mathrm{Unif}(\{-1,+1\}^N)} |Q_N(\bs)|^p \right)^{\frac{1}{p}} = O\left(\frac{1}{(N\theta_N)^2}\right).
    \end{equation*}
\end{lemma}

The next lemma is the analogue of Proposition F.1 in \cite{mukherjee2026isinginference} for possibly signed kernels.
\begin{lemma}
\label{quadJointAsym}
Let
\(
    K,K':[0,1]^2\rightarrow\mathbb R
\)
be symmetric, Riemann-integrable kernels, and suppose that
$\{\sigma_i\}_{i\geq1}$ and $\{\tau_i\}_{i\geq1}$ are two independent
sequences of i.i.d. Rademacher random variables. Let
$\{\lambda_r(K)\}_{r\geq1}$ and
$\{\lambda_r(K')\}_{r\geq1}$ denote the nonzero eigenvalues of
$T_K$ and $T_{K'}$, respectively, repeated according to
multiplicity. Then, we have:
\[
\begin{pmatrix}
\displaystyle
\frac{1}{2N}
\sum_{1\leq i\neq j\leq N}
K\left(\frac{i}{N},\frac{j}{N}\right)
\bigl(\sigma_i\sigma_j+\tau_i\tau_j\bigr)
\\[4mm]
\displaystyle
\frac{1}{2N}
\sum_{1\leq i\neq j\leq N}
K'\left(\frac{i}{N},\frac{j}{N}\right)
\sigma_i\tau_i\sigma_j\tau_j
\end{pmatrix}
\xrightarrow{d}
\begin{pmatrix}
\displaystyle
\frac12\sum_{r=1}^{\infty}
\lambda_r(K)(X_r^2-1)
+
\frac12\sum_{r=1}^{\infty}
\lambda_r(K)(Y_r^2-1)
\\[4mm]
\displaystyle
\frac12\sum_{r=1}^{\infty}
\lambda_r(K')(Z_r^2-1)
\end{pmatrix},
\]
where
$\{X_r\}_{r\geq1}$,
$\{Y_r\}_{r\geq1}$, and
$\{Z_r\}_{r\geq1}$ are mutually independent sequences of
i.i.d. standard Gaussian random variables.
All three infinite series appearing in the distributional limit converge in $L^2$ and almost surely.
\end{lemma}

\begin{proof}
The proof of Lemma \ref{quadJointAsym} closely follows the cumulant argument of
\cite[Proposition~F.1]{mukherjee2026isinginference}, so we sketch the deviations.
Although that proposition is stated for graphons, its proof uses only
that the kernels are bounded, symmetric, and Riemann integrable. Indeed, for fixed $a,b\in\mathbb R$, the weighted-$U$-statistic
argument gives limiting cumulants
\[
    \frac{(r-1)!}{2}
    \left\{
        2a^r c_r(K)+b^r c_r(K')
    \right\},
    \qquad r\geq2,
\]
where
\[
    c_2(K)=\int_{[0,1]^2}K(x,y)^2\,\mathrm dx\,\mathrm dy
\]
and, for $r\geq3$,
\[
    c_r(K)
    =
    \int_{[0,1]^r}
    \prod_{s=1}^r K(x_s,x_{s+1})
    \,\mathrm dx_1\cdots\mathrm dx_r,
    \qquad x_{r+1}=x_1.
\]
The same Rademacher parity argument as in the cited proof shows that
all mixed terms vanish.

Since $T_K$ and $T_{K'}$ are self-adjoint Hilbert--Schmidt
operators, we have
\[
    c_r(K)=\sum_{j\geq1}\lambda_j(K)^r,
    \qquad
    c_r(K')=\sum_{j\geq1}\lambda_j(K')^r,
    \qquad r\geq2.
\]
These are precisely the cumulants of the corresponding linear
combination of the Gaussian-chaos variables appearing on the distributional limits in the
right-hand side. Their moment-generating functions are finite in a
neighborhood of zero. The conclusion therefore follows from the
Cram\'er--Wold device.
\end{proof}

In this paper, we actually use the following one-copy consequence of
Lemma~\ref{quadJointAsym}:
\begin{equation}\label{onecopy2002}
    \frac{1}{2N}
    \sum_{1\leq i\neq j\leq N}
    K\left(\frac{i}{N},\frac{j}{N}\right)
    \sigma_i\sigma_j
    \xrightarrow{d}
    \frac12
    \sum_{r=1}^{\infty}
    \lambda_r(K)(Z_r^2-1)
\end{equation}
where $\{Z_r\}_{r\ge 1}$ is a sequence of i.i.d. standard Gaussian random variables. Indeed, since $\{\sigma_i\tau_i\}_{i\geq1}$ is again an i.i.d. mean-zero
Rademacher sequence, the second-coordinate convergence in
Lemma~\ref{quadJointAsym} gives \eqref{onecopy2002}.

Finally, the following corresponds to Lemma F.6 in \cite{mukherjee2026isinginference}.

\begin{lemma}\label{l2ratio762}
   Suppose that $\{X_N\}_{N\ge 1}, \{Y_N\}_{N\ge 1}$ are sequences of random variables all defined on a common probability space $\Omega$. Suppose that there exists $\delta>0$ such that for every $t\in (-\delta,\delta)$,  $$\E \left(e^{tX_N}\Big |Y_N\right) \xrightarrow{P} \E e^{tV}$$ for some random variable $V$. Then,
   \begin{enumerate}
       \item $X_N \xrightarrow{d} V$,
       \item For every $p\ge 1$, $\E (X_N^p|Y_N) \xrightarrow{P} \E V^p$. 
   \end{enumerate}
\end{lemma}

\section{A stationary Ornstein--Uhlenbeck representation of \(X(t)\)}
\label{rem:OU-representation-X}
In this section, we give an alternative dynamical representation of the process $X(t)$ in \eqref{defXt} in terms of stationary Ornstein--Uhlenbeck processes.     To begin with, note that indicator functions can be represented in $L^2[0,1]$ with respect to
the complete orthonormal basis $\{\phi_i\}_{i\geq1}$ as
$$
\mathbf 1_{[0,t]}
=
\sum_{i=1}^\infty
\left\langle \mathbf 1_{[0,t]},\phi_i\right\rangle\phi_i
=
\sum_{i=1}^\infty \Phi_i(t)\phi_i.
$$
Since the Wiener integral is linear and continuous with respect to
$L^2$-convergence, the preceding expansion may be passed through the
stochastic integral. Hence, for every fixed $t\in[0,1]$,
$$
B(t)
=
\int_0^1 \mathbf 1_{[0,t]}(x)\,\mathrm dB(x)
=
\sum_{i=1}^\infty
\Phi_i(t)\int_0^1\phi_i(x)\,\mathrm dB(x)
=
\sum_{i=1}^\infty \Phi_i(t)Z_i,
$$
where the final series converges in mean square. With the notation
$a_i:=1-\beta\lambda_i$, the representation
\eqref{defXt} may thus equivalently be written as
\begin{equation}
\label{eq:X-pure-spectral}
    X(t)
    =
    \sum_{i=1}^{\infty}
    \frac{\Phi_i(t)}{\sqrt{a_i}}\,Z_i,
    \qquad 0\leq t\leq1.
\end{equation}

Now, for each $i\geq1$, let $\{V_i(s)\}_{s\in[0,1]}$ be independent
solutions of the SDE
\begin{equation}
\label{ouprocess26}
    \mathrm dV_i(s)
    =
    -a_iV_i(s)\,\mathrm ds
    +
    \sqrt{2}\,\mathrm dB_i(s),
\end{equation}
where $\{B_i\}_{i\geq1}$ are independent standard Brownian motions.
Assume that the variables $\{V_i(0)\}_{i\geq1}$ are independent,
independent of all the Brownian motions, and satisfy
$V_i(0)\sim\mathcal N(0,a_i^{-1})$. Thus, $V_i$ is a stationary Ornstein--Uhlenbeck process with
mean-reversion rate $a_i$, long-run mean $0$, and diffusion
coefficient $\sqrt{2}$. The explicit solution of
\eqref{ouprocess26} is given by:
$$
V_i(s)
=
e^{-a_i s}V_i(0)
+
\sqrt{2}\int_0^s e^{-a_i(s-r)}\,\mathrm dB_i(r).
$$
Hence, $V_i(s)$ is centered Gaussian and
$$
\operatorname{Var}(V_i(s))
=
\frac{e^{-2a_i s}}{a_i}
+
2\int_0^s e^{-2a_i(s-r)}\,\mathrm dr
=
\frac{1}{a_i}.
$$
Moreover,
$\operatorname{Cov}(V_i(s),V_i(r))
=a_i^{-1}e^{-a_i|s-r|}$, so the processes
$\{V_i\}_{i\geq1}$ are independent and stationary. In particular,
$$
V_i(s)\sim\mathcal N\left(0,\frac{1}{a_i}\right)
\qquad\text{for every }s\in[0,1].
$$

For every fixed $s\in[0,1]$, define
$$
U_s(t)
:=
\sum_{i=1}^{\infty}\Phi_i(t)V_i(s),
\qquad 0\leq t\leq1.
$$
The series converges in mean square for every fixed $t$, since, by
Parseval's identity,
$$
\sum_{i=1}^{\infty}
\mathbb E\left[\Phi_i(t)^2V_i(s)^2\right]
=
\sum_{i=1}^{\infty}\frac{\Phi_i(t)^2}{a_i}
\leq
\frac{1}{1-\beta\|W\|_{\mathrm{op}}}
\sum_{i=1}^{\infty}\Phi_i(t)^2
=
\frac{t}{1-\beta\|W\|_{\mathrm{op}}}.
$$
Moreover, for every fixed $s$, the random variables
$\{\sqrt{a_i}V_i(s)\}_{i\geq1}$ are i.i.d.\ standard Gaussian, and
hence
$$
\bigl(V_i(s)\bigr)_{i\geq1}
\stackrel{d}{=}
\left(\frac{Z_i}{\sqrt{a_i}}\right)_{i\geq1}.
$$
Comparing with \eqref{eq:X-pure-spectral}, $U_s$ and $X$ have the
same finite-dimensional distributions. Also, the increment estimate in
Lemma~\ref{holderX} applies verbatim to $U_s$, so it admits a
continuous modification. Therefore, for every $s\in[0,1]$, we have
$$
\{U_s(t)\}_{t\in[0,1]}
\stackrel{d}{=}
\{X(t)\}_{t\in[0,1]}
\qquad\text{in }C[0,1].
$$
Thus, the limiting process $X$ may be realized in law as the weighted
sum of any fixed-time slice $\{V_i(s)\}_{i\geq1}$ of the independent
stationary Ornstein--Uhlenbeck processes $\{V_i\}_{i\geq1}$. Retaining
the auxiliary time parameter $s$ yields a two-parameter
Gaussian process
$\{U_s(t):s, t\in[0,1]\}$ that is stationary in $s$, every fixed-$s$ slice of which has
the law of $X$, while its dependence across $s$ reflects the distinct
Ornstein--Uhlenbeck relaxation rates
$a_i=1-\beta\lambda_i$ associated with the graphon eigenfunctions $\phi_i$. In particular, this Ornstein--Uhlenbeck embedding gives a
reversible coupling of a continuum of copies of $X$, allowing correlated
realizations of $X$ with explicitly controlled dependence to be simulated
simultaneously by varying their separation in the auxiliary time parameter.

\section{An alternative derivation of Theorem~\ref{thm:donsker} in the Curie--Weiss model}
\label{rem:CW-alternative-donsker}
The Curie--Weiss model, namely, the Ising model on the complete graph,
is the special case of the Erd\H{o}s--R\'enyi Ising model with
$\theta_N=1$. It follows from Theorem~\ref{thm:donsker} and
Example~\ref{example:block-limiting-process} that, for
$0<\beta<1$,
\begin{equation}
\label{CWDonsker4}
    \{S_N(t)\}_{t\in[0,1]}
    \xrightarrow{d}
    \left\{
        B(t)
        +
        \left(
            \frac{1}{\sqrt{1-\beta}}-1
        \right)tB(1)
    \right\}_{t\in[0,1]}
\end{equation}
in $D[0,1]$. We now give a direct derivation of
\eqref{CWDonsker4} using the classical auxiliary-variable
representation of the Curie--Weiss model.
Towards this, let $Y_N$ have density
\[
    f_{Y_N}(y)
    =
    \frac{\exp\{-N\psi_\beta(y)\}}
    {\displaystyle
        \int_{-\infty}^{\infty}
        \exp\{-N\psi_\beta(z)\}\,\mathrm dz},
    \qquad
    \psi_\beta(y)
    :=
    \frac{\beta y^2}{2}
    -
    \log\cosh(\beta y).
\]
Conditional on $Y_N$, let $s_1,\ldots,s_N$ be independent with
\[
    \mathbb P(s_i=x\mid Y_N)
    =
    \frac{\exp\{x\beta Y_N\}}
    {2\cosh(\beta Y_N)},
    \qquad x\in\{-1,1\}.
\]
Then
the unconditional law of $\bm s$ is the Curie--Weiss law;
see \cite[Lemma~3]{mukherjee2018global}. 

For $0<\beta<1$, the function $\psi_\beta$ has a unique global
minimum at zero. In fact,
\[
    \psi_\beta'(y)
    =
    \beta(y-\tanh(\beta y)),
    \qquad
    \psi_\beta''(0)
    =
    \beta(1-\beta).
\]
Consequently, the Laplace method yields
\[
    \sqrt N\,Y_N
    \xrightarrow{d}
    G,
    \qquad
    G\sim
    \mathcal N\left(
        0,\frac{1}{\beta(1-\beta)}
    \right).
\]
Writing
\(
    \mu_N:=\tanh(\beta Y_N),
\)
the delta method therefore gives
\begin{equation}
\label{eq:CW-latent-mean-limit}
    \sqrt N\,\mu_N
    \xrightarrow{d}
    Z,
    \qquad
    Z\sim
    \mathcal N\left(
        0,\frac{\beta}{1-\beta}
    \right).
\end{equation}

Now define the conditionally centered partial-sum process
\[
    U_N(t)
    :=
    \frac{1}{\sqrt N}
    \sum_{i=1}^{\lfloor Nt\rfloor}
    (s_i-\mu_N),
    \qquad 0\leq t\leq1.
\]
Conditional on $Y_N$, the summands are independent, centered, bounded,
and have variance
\[
    \operatorname{Var}(s_i| Y_N)
    =
    1-\mu_N^2
    \xrightarrow{P}1.
\]
The classical Donsker theorem for i.i.d. sum processes therefore gives
\[
    \mathcal L(U_N| Y_N)
    \xrightarrow{d}
    \mathcal L(B)
\]
in probability as probability measures on $D[0,1]$. Since the
conditional limit does not depend on $Y_N$, this convergence together
with \eqref{eq:CW-latent-mean-limit} implies
\[
    \left(U_N,\sqrt N\,\mu_N\right)
    \xrightarrow{d}
    (B,Z),
\]
where $B$ is a standard Brownian motion and $Z$ is independent of
$B$. Now, note that
\[
    S_N(t)
    =
    U_N(t)
    +
    \frac{\lfloor Nt\rfloor}{N}\sqrt N\,\mu_N,
\]
and
\[
    \sup_{0\leq t\leq1}
    \left|
        \frac{\lfloor Nt\rfloor}{N}\sqrt N\,\mu_N
        -
        t\sqrt N\,\mu_N
    \right|
    \leq
    \frac{|\sqrt N\,\mu_N|}{N}
    \xrightarrow{P}0.
\]
It follows that
\begin{equation}
\label{eq:CW-independent-decomposition}
    S_N
    \xrightarrow{d}
    \{B(t)+tZ\}_{t\in[0,1]}.
\end{equation}

To identify this limit with the process in \eqref{CWDonsker4}, let us write
\[
    B^\circ(t):=B(t)-tB(1)
\]
for the Brownian bridge associated with $B$. Note that $B^\circ, B(1), Z$ are independent, and
\begin{equation}\label{bridge12}
       B(t)+tZ
    =
    B^\circ(t)+t\{B(1)+Z\}
\end{equation}
where
\begin{equation}\label{bridge48}
    B(1)+Z
    \sim
    \mathcal N\left(0,\frac{1}{1-\beta}\right).
\end{equation}
Thus, if $\widetilde B$ is a standard Brownian motion, it follows from \eqref{bridge12}, \eqref{bridge48} and the independence of $\{\widetilde B(t) - t \widetilde B(1)\}_{t\in [0,1]}$ (the Brownian bridge associated with $\widetilde B$) and $\widetilde B(1)$, that
\[
    \{B(t)+tZ\}_{t\in[0,1]}
    \stackrel{d}{=}   \left\{
        \widetilde B(t) - t \widetilde B(1)
        + 
            \frac{t}{\sqrt{1-\beta}}\widetilde B(1)
    \right\}_{t\in[0,1]} =
    \left\{
        \widetilde B(t)
        +
        \left(
            \frac{1}{\sqrt{1-\beta}}-1
        \right)t\widetilde B(1)
    \right\}_{t\in[0,1]},
\]
which proves \eqref{CWDonsker4}.

\end{document}